\documentclass[11pt, letterpaper]{amsart}
\usepackage[left=1in,right=1in,bottom=0.5in,top=0.8in]{geometry}

\usepackage{amsfonts}
\usepackage{amsmath, amssymb}
\usepackage{graphicx}

\usepackage{esvect}
\usepackage[font=small,labelfont=bf]{caption}
\usepackage{subcaption}

\usepackage{epstopdf}

\usepackage{amsthm}
\usepackage{float}
\usepackage{pgfplots}
\usepackage{listings}
\usepackage{longtable}
\usepackage{mathrsfs}
\usepackage{tikz-cd}
\usepackage{tikz}
\usetikzlibrary{arrows.meta,positioning,fit,matrix,backgrounds}
\usepackage[all]{xy}
\usepackage{adjustbox}

\usepackage[pdfpagelabels,hyperindex]{hyperref}

\hypersetup{
pdftitle={Circular foliations and stretched maps of hyperbolic cone-surfaces},
pdfsubject={Mathematics, Geometric topology, differential geometry},
pdfauthor={Qiyu Chen, Youliang Zhong},
pdfkeywords={hyperbolic cone-surfaces; circular foliations; maximally stretched maps}
}

\newtheorem{theorem}{\rm\bf Theorem}[section]

\theoremstyle{definition}

\theoremstyle{remark}

\newtheorem{proposition}[theorem]{\rm\bf Proposition}
\newtheorem{construction}[theorem]{\rm\bf Construction}

\newtheorem{lemma}[theorem]{\rm\bf Lemma}
\newtheorem{corollary}[theorem]{\rm\bf Corollary}
\newtheorem{definition}[theorem]{\rm\bf Definition}
\newtheorem{remark}[theorem]{\rm\bf Remark}

\newtheorem{question}[theorem]{\rm\bf Question}

\newcommand{\up}{\underline{p}}
\newcommand{\Sp}{(S,\up)}
\newcommand{\TSp}{\text{Tri}\Sp}

\newcommand{\lAC}{\ell^{(\cAC)}}
\newcommand{\lG}{\ell^{(G)}}
\newcommand{\iG}{i^{(G)}}
\newcommand{\OG}{\Omega_{G}}

\newcommand{\cMF}{\mathcal{MF}}
\newcommand{\cMFp}{\mathcal{MF} (S,\up)}
\newcommand{\cMFo}{\cMF_o(S,\up)}
\newcommand{\cC}{\mathcal{C}}
\newcommand{\cSR}{\mathcal{SR}}
\newcommand{\cR}{\mathcal{R}}
\newcommand{\cA}{\mathcal{A}}

\newcommand{\cAC}{\mathcal{A} \cup \mathcal{C}}
\newcommand{\cCG}{\mathcal{C}_{G}}

\newcommand{\cT}{\mathcal{T}}
\newcommand{\cTp}{\mathcal{T} (S,\up)}
\newcommand{\cTG}{\mathcal{T} (G)}

\newcommand{\cU}{\mathcal{U}}

\newcommand{\bR}{\mathbb{R}}
\newcommand{\bN}{\mathbb{N}}
\newcommand{\bRn}{\mathbb{R}_{\geq 0}}
\newcommand{\bRp}{\mathbb{R}_{> 0}}
\newcommand{\bH}{\mathbb{H}}

\newcommand{\cS}{\mathcal{S}}
\usepackage{xcolor}

\definecolor{energy}{RGB}{114,0,172}
\definecolor{freq}{RGB}{45,177,93}
\definecolor{spin}{RGB}{251,0,29}
\definecolor{signal}{RGB}{203,23,206}
\definecolor{circle}{RGB}{217,86,16}
\definecolor{average}{RGB}{203,23,206}
\newcommand{\DEF}{}    %

\colorlet{shadecolor}{gray!20}
\pgfplotsset{compat=1.9}

\makeatletter
\let\c@equation\c@thm
\makeatother
\numberwithin{equation}{section}

\author[Qiyu Chen]{Qiyu Chen}
\address{Qiyu Chen:
School of Mathematics, South China University of Technology, 510641, Guangzhou, P. R. China.}
\email{qiyuchen@scut.edu.cn}
\thanks{Q. Chen was partially supported by NSFC (No. 12471075, No. 12271174) and Guangdong Basic and Applied Basic Research Foundation (No. 2025A1515011281). 
}

\author{Youliang Zhong}
\address{Youliang Zhong:
School of Mathematics, South China University of Technology, 510641, Guangzhou, P. R. China.}
\email{zhongyl@scut.edu.cn}
\thanks{Y. Zhong was partially supported by NSFC (No. 12271174) and Guangdong Basic and Applied Basic Research Foundation (No. 2023A1515010929).} 

\date{\today}

\usepackage[style=british]{csquotes}

\usepackage[
  backend=biber,
  style=alphabetic,
  sorting=nyt,
  sortcites=false,
  giveninits=true,
  maxalphanames=5,
  minalphanames=5,
  maxnames=99, minnames=99,
  doi=false, url=false, isbn=false
]{biblatex}

\renewbibmacro{in:}{}

\DeclareFieldFormat[article,incollection,inproceedings]{title}{\mkbibquote{#1}}

\DeclareFieldFormat[article]{volume}{#1}

\DeclareFieldFormat[article]{pages}{#1}
 
\title[Circular foliations and shear--radius coordinates]{Circular foliations and shear--radius coordinates on Teichm\"uller spaces of hyperbolic cone surfaces}

\begin{document}

\begin{abstract}
We study the Teichm\"uller space \(\cT(S,\underline{p})\) of hyperbolic cone-surfaces of fixed topological type with marked cone singularities.
Fix a combinatorial triangulation \(G\), and let \(\cT(G)\subset \cT(S,\underline{p})\) be the locus where \(G\) admits a geodesic realization; varying \(G\), these loci form an open cover of \(\cT(S,\underline{p})\).
On \(\cT(G)\) we construct a circular foliation adapted to geodesic triangular complementary regions, which is naturally decomposed into interior and peripheral parts.
This decomposition defines shear parameters on edges and radius parameters at the singularities, and yields global coordinates on \(\cT(G)\): the resulting shear--radius map is a homeomorphism onto an explicit open cone in a finite-dimensional real vector space.
In the spirit of Thurston, we then introduce partial stretch and anti-stretch deformations by rescaling the transverse measures of the interior or peripheral components.
Peripheral stretch rays converge, in the simple-curve length-spectrum topology, to the cusped hyperbolic metric determined by the shear data, while interior anti-stretch rays converge to a circle-packed hyperbolic cone metric determined by the radii.
Finally, we give criteria for the realization of prescribed cone angles for fixed \(G\) and prove sharp upper bounds for admissible cone angles on the universally triangulable locus.
\end{abstract}

 \maketitle
\tableofcontents
\section{Introduction}
\label{sec:intro}

The aim of this paper is to develop a geometric parametrization and deformation theory for the Teichm\"uller space \(\cTp\) of hyperbolic cone-surfaces.
Building on Thurston's theory of horocyclic foliations and stretch maps for punctured surfaces, we introduce \emph{circular foliations} and associated \emph{shear--radius coordinates} for cone-surfaces.
These coordinates provide hyperbolic charts that remain valid for arbitrary cone angles and furnish a natural framework for deformations that vary cone angles and relate cone-metrics to cusped metrics and (generalized) circle packings.

\subsection{Hyperbolic cone-surfaces and their Teichm\"uller space}
\label{subsec:Teich}

In this subsection we fix notation, recall the existence and uniqueness theory for hyperbolic cone-metrics, and introduce the Teichm\"uller space \(\cTp\).

Let \(S\) be an oriented closed surface of genus \(g\) with a finite ordered tuple of $n\geq 1$ marked points \(\underline{p} = (p_1,\dots,p_n)\).
We assume throughout that \(2g-2+n > 0\).
A \emph{hyperbolic cone-metric} on \((S,\underline{p})\) is a metric of constant curvature \(-1\) on \(S \setminus \underline{p}\) whose completion has cone singularities at the marked points, with cone angles
\[
  \underline{\theta} = (\theta_1,\dots,\theta_n) \in (0,+\infty)^n.
\]
We denote by \(\dot{S} = S \setminus \underline{p}\) the punctured surface.
The cone angles are constrained by the Gauss--Bonnet condition:
\begin{equation}\label{eq:Gauss-Bonnet condition}
  \chi(\dot{S}) + \sum_{i=1}^n \frac{\theta_i}{2\pi} < 0.
\end{equation}
Under this assumption, there is a unique hyperbolic cone-metric with prescribed cone angles $\underline{\theta}$ in a given conformal class \cite{Hei62, McO88, Tro91, FX25} (see Remark \ref{rmk:ang.class.Teich.space}). 

Beyond the classical case, the Teichm\"uller space of cone-metrics exhibits new analytic, geometric and dynamical features, ranging from local deformation theory and elliptic conic operators \cite{MW17} to counting problems \cite{ES22}, billiards \cite{ELS22}, volumes \cite{DN09, AN22} and identities \cite{TWZ06, BPT25}.
On the three-dimensional side, quasi-Fuchsian hyperbolic manifolds with particles have convex cores whose boundary components carry hyperbolic metrics with cone singularities \cite{KS07, MS09, LS14}; in the anti-de Sitter setting, an analogous picture holds \cite{KS07, BS09, CS20}.

The Teichm\"uller space \(\cTp\) is the space of marked hyperbolic cone-metrics on \((S,\underline{p})\), modulo isotopy fixing \(\underline{p}\) pointwise.
From the conformal viewpoint, \(\cTp\) may be identified with the Teichm\"uller space of Riemann surfaces with prescribed cone angles at $\up$ of a conformal hyperbolic metric (see e.g. \cite{McO88, Tro91, ST05, Mon10, MW17}). 
Our perspective in this paper is hyperbolic and combinatorial: we seek coordinates and deformations on \(\cTp\) that are intrinsic to the hyperbolic geometry and together they yield an atlas on $\cTp$ that applies to every cone-angle vector allowed by the Gauss–Bonnet condition.
 
\subsection{Motivations and guiding questions}
\label{subsec:motivation}

In this subsection we position our work with respect to classical hyperbolic parametrizations and formulate two guiding questions that drive the paper.

On punctured hyperbolic surfaces, Teichm\"uller space admits several robust hyperbolic parametrizations, such as Fenchel--Nielsen coordinates, shear coordinates \cite{Bon96, Thu98} and Penner's \(\lambda\)-lengths \cite{Pen87}.
These descriptions, built from ideal triangulations and geodesic laminations, underpin deformation theories such as earthquakes and stretch maps \cite{Thu84, Thu98}.

As recalled in Section~\ref{subsec:Teich}, hyperbolic cone-metrics have been studied from analytic, dynamical and three-dimensional perspectives. 
From the viewpoint of hyperbolic parametrizations, however, most existing constructions for cone-surfaces either impose an a priori upper bound on cone angles (typically \(\theta_i \in (0,\pi)\)) (see e.g. \cite{DP07, Pan17}) or are phrased in analytic terms rather than in combinatorial hyperbolic coordinates (see e.g. \cite{MW17}).
When a cone angle crosses \(\pi\), simple closed geodesics may cease to be smooth \cite[\S 4.6]{CHK00}.
In particular, there is currently no triangulation-based hyperbolic coordinate system on the Teichm\"uller space \(\cTp\) that yields an atlas for all cone angles allowed by the Gauss--Bonnet condition. 
This leads to our first guiding question.

\begin{question}[Coordinates without angle restriction]
  \label{Q:coor.Teich.}
  \textit{Is there a parametrization of \(\cTp\) from a hyperbolic and combinatorial viewpoint, built from triangulations or foliations, that provides a coordinate atlas for all cone-angle vectors allowed by the Gauss--Bonnet condition, without imposing any a priori upper bound on the angles?}
\end{question}

Turning to deformations, classical earthquake theory along measured laminations provides rich families of deformations that preserve cone angles and is central to the geometry and dynamics of Teichm\"uller space \cite{Thu84, Ker83, BS09}.
Thurston's stretch deformations and their Lipschitz and length-spectrum refinements, including the analysis of lengthening cones on singular tori in \cite{Gue15}, give further distinguished one-parameter families in the punctured or singular setting.
Beyond earthquakes, related deformation phenomena have also been studied in more specialized configurations, such as partial stretch maps on the once-holed torus \cite{HP21} and discrete conformal deformations in triangulation-based settings, including combinatorial Ricci flows \cite{CL03,GLSW18} and their degenerate circle-packing limits \cite{HLTZZ25}. 
In the cone-surface setting we consider, however, the constructions we are aware of typically keep cone angles fixed or rely on specific topological or combinatorial setups.
This leads to our second guiding question.

\begin{question}[Deformations changing cone angles]
  \label{Q:deform.Teich.}
  \textit{For hyperbolic cone-surfaces, can we define natural geometric deformations that change the cone angles in a controlled way, while keeping compatible with a hyperbolic coordinate system on \(\cTp\), and whose asymptotic behaviour can be described in geometric terms (for instance, in terms of cusped or circle-packed limits)?}
\end{question}

The main contribution of this paper is to answer these two questions from a unified perspective, using a combinatorial model space of measured foliations and a circular foliation map (see below).

\subsection{Main results and geometric overview}
\label{subsec:result}

In this subsection we summarise our main results and describe the geometric picture that emerges.
The central idea is to build a bridge between the hyperbolic geometry of cone-surfaces and a model space of measured foliations \(\cMFp\) that are {\em trivial} around the marked points (Definition \ref{def:MF.surf.trivial.}).
This bridge is provided by a circular foliation map
\[
  \cCG : \cT(G) \longrightarrow \cMFp,
\]
which assigns to each hyperbolic cone-metric geodesically realizing a triangulation \(G\in \TSp\) (Definition~\ref{def:triangulation}) a canonical measured foliation class.
On one hand, we endow \(\cMFp\) with shear--radius coordinates adapted to the triangulation \(G\).
On the other hand, we show that \(\cCG\) is compatible with these coordinates and yields charts on \(\cT(G)\).
This leads to partial stretch and anti-stretch deformations with controlled asymptotics, as well as a description of maximal cone angles on a universally triangulable locus.

\subsubsection{Measured foliations and shear--radius parameters}
\label{subsubsec:intro.MF.shear-radius}

We introduce the model space \(\cMFp\) of measured foliations and describe its shear--radius parametrization, which will later serve as the combinatorial target for the circular foliation map. 

The space \(\cMFp\) consists of measured foliation classes on \(\Sp\) that are trivial in a neighbourhood of each marked point, in the sense of the punctured foliation formalism developed by Mosher \cite{Mos88} and by Papadopoulos--Penner \cite{PP93}.
It provides a convenient analogue of the classical space of measured foliations on punctured surfaces, adapted here to the presence of cone points. 
We show that these foliation classes are equivalent to those {\em totally transverse} (Definition \ref{def:MF.trans.triang.surf.}) to a triangulation $G$ on $\Sp$, see Proposition~\ref{prop:t.a.imply.t.t.G}.

Fix a triangulation \(G \in \TSp\). For each foliation class \([F] \in \cMFp\), we associate a \emph{shear--radius vector}
\[
  \mathbf{sr}_G([F]) = \bigl(\mathbf{s}_G([F]), \mathbf{r}_G([F])\bigr) \in \bR^{E(G)} \times \bR^n_{>0},
\]
where the shear component \(\mathbf{s}_G([F])\) records the transverse shear parameters across the edges of \(G\), and the radius component \(\mathbf{r}_G([F])\) records the radii of the peripheral components around the marked points.
The precise definitions of the shear and radius maps are given in Definitions~\ref{def:MF.shear.map} and~\ref{def:MF.radius}. 
The image of \(\mathbf{sr}_G\) is an explicit convex cone cut out by linear constraints at the vertices of \(G\).
More concretely, we define \(\Lambda_G \subset \bR^{E(G)} \times \bR^n_{>0}\), the \emph{vertex-balanced cone}, to be the set of pairs \((\mathbf{s},\mathbf{r})\) whose component \(\mathbf{s}\) satisfies a vertex balance condition at each marked point (Definition~\ref{def:vertex.balanc.part}).
Our first main result on the model space shows that these parameters give global coordinates on \(\cMFp\).

\begin{theorem}[Shear--radius parametrization of \texorpdfstring{\(\cMFp\)}{MF(S,\underline{p})}]
\label{thm:homeo.MFp.shear.radius}
For each triangulation \(G \in \TSp\), the shear--radius map
\[
  \mathbf{sr}_G : \cMFp \longrightarrow \Lambda_G
\]
is a homeomorphism onto the vertex-balanced cone \(\Lambda_G\).
In particular, shear and radius parameters provide a global coordinate system on \(\cMFp\).
\end{theorem}

The proof of Theorem~\ref{thm:homeo.MFp.shear.radius} will be given at the end of Section~\ref{subsec:invers.contin.}.

\subsubsection{Circular foliations and local coordinates}
\label{subsubsec:intro.circ-local}

We explain how circular foliations produce local coordinate charts on \(\cTp\) by linking hyperbolic cone-metrics to the model space \(\cMFp\).

Given a triangulation \(G \in \TSp\), let \(\cT(G) \subset \cTp\) denote the open locus of hyperbolic cone-metrics for which \(G\) is realized as a geodesic triangulation (Definition~\ref{def:Teich.cone.surf.}).
For each such \(G\), the key geometric tool is the \emph{circular foliation map}
\(\cC_G : \cT(G) \longrightarrow \cMFp\) (Definition~\ref{def:circ.foli.map.}).
For each hyperbolic cone-surface \(X \in \cT(G)\), we obtain a canonical measured foliation class \(\cC_G(X) \in \cMFp\), called the \emph{circular foliation class}, whose leaves are families of concentric circular arcs centered at the cone points. This construction is analogous to Thurston's horocyclic foliation map for punctured surfaces \cite{Thu98} and encodes geometric information about \(X\) in combinatorial and measure-theoretic terms.

Every \(X\in\cTp\) is \(G\)-triangulable for some \(G\in\TSp\) (Remark~\ref{rem:GB.and.triang}); hence the loci \(\cT(G)\) form an open cover of \(\cTp\).
Our first main result shows that the circular foliation maps \(\cC_G\) provide an atlas of local charts on \(\cTp\).
The detailed construction and proof are given in Section~\ref{subsec:circ.foli.surf.coor.}.

\begin{theorem}[Circular foliation coordinates]\label{thm:Teich.chart}
The Teichm\"uller space \(\cTp\) is a manifold of real dimension \(6g-6+3n\).
For each triangulation \(G\in\TSp\), let \(\cT(G)\subset \cTp\) be the locus where \(G\) admits a geodesic realization.
Then the circular foliation map
\[
\cC_G:\cT(G)\longrightarrow \cMFp
\]
is a homeomorphism.
Moreover, the family \(\{(\cT(G),\cC_G)\}_{G\in\TSp}\) provides an atlas of \(\cTp\) modelled on the common space \(\cMFp\cong \bR^{6g-6+3n}\).
\end{theorem}

To obtain explicit coordinates on these charts, we compose the circular foliation maps with the shear--radius parametrization of \(\cMFp\) from Theorem~\ref{thm:homeo.MFp.shear.radius}.
For each triangulation \(G \in \TSp\), we define the \emph{shear--radius map}
\[
  \cSR_G := \mathbf{sr}_G \circ \cC_G : \cT(G) \longrightarrow \bR^{E(G)} \times \bR^n_{>0},
\]
where \(\mathbf{sr}_G : \cMFp \to \bR^{E(G)} \times \bR^n_{>0}\) is the shear--radius map introduced in Section~\ref{subsubsec:intro.MF.shear-radius}.
The following result records this coordinate description and follows directly from Theorems~\ref{thm:homeo.MFp.shear.radius} and~\ref{thm:Teich.chart}.

\begin{corollary}[Shear--radius coordinates]
    \label{cor:homeo.TS.shear.radius}
For each triangulation \(G \in \TSp\), the map
\[
  \cSR_G : \cT(G) \longrightarrow \Lambda_G
\]
is a homeomorphism onto the vertex-balanced cone \(\Lambda_G\).
\end{corollary}

In particular, for every \(G \in \TSp\) the shear and radius parameters give global coordinates on the locus \(\cT(G)\), and the family of maps \(\{\cSR_G\}_{G \in \TSp}\) defines a system of local coordinate charts on \(\cTp\), which we refer to as the \emph{shear--radius coordinates}.
These coordinates extend the classical Bonahon--Thurston shear coordinates \cite{Bon96, Thu98} and may be compared with the shear--type frameworks of Pan--Wolf \cite{PW22} and Calderon--Farre \cite{CF24} (see Remark~\ref{rem:MF.shear.classical}).

\subsubsection{Partial stretch and anti-stretch deformations}
\label{subsubsec:intro.partial-stretch}

We use the circular foliation coordinates (Theorem~\ref{thm:Teich.chart}) to define deformations that change cone angles, thereby addressing Question~\ref{Q:deform.Teich.}, and to describe their asymptotic limits in terms of cusped surfaces and circle packings.

By rescaling, for a parameter \(t \in \bR\), the transverse measures on either the interior or the peripheral components of the circular foliation (see Section~\ref{subsubsec:intro.circ-local}) by a factor \(e^t\), we obtain deformations of \(\cT(S,\underline{p})\); for \(t>0\) these are stretches and for \(t<0\) anti-stretches.
This construction yields one-parameter families of partial (interior or peripheral) stretch/anti-stretch deformations on \(\cT(S,\underline{p})\) (Definition~\ref{def:part.stret.}).
These may be regarded as cone-surface analogues of Thurston stretch deformations, but they differ from the partial stretch maps of Huang--Papadopoulos on once-holed tori \cite{HP21} in that our constructions are organized via circular foliation maps rather than via Lipschitz-optimal maps (see Remark~\ref{rem:HP-vs-us} for a detailed comparison).

Our first asymptotic result describes the behaviour of the peripheral stretch rays, along which the radius parameters tend to \(+\infty\) while the shear parameters remain fixed.

\begin{theorem}[Asymptotic behaviour of peripheral stretch rays]\label{thm:asym.behavior}
For each \(G \in \TSp\) and each \(X_0 \in \cT(G)\), the peripheral stretch ray \((X_t)_{t \ge 0}\) based at \(X_0\) converges, in the simple-curve length-spectrum topology, to the hyperbolic surface with cusps at \(\underline{p}\) which is described by the shear coordinates of \(X_0\).
\end{theorem}

Thus, stretching only the peripheral part of the circular foliation while keeping the shear parameters fixed
realizes the classical cusped surface associated to the shear data of \(X_0\) as a geometric limit within the cone-surface setting.
Here the convergence is understood in the simple-curve length-spectrum topology (Definition~\ref{def:topology.on.T}):
although the peripheral stretch ray typically diverges in the arc-curve length-spectrum topology, its restriction to closed curve lengths admits a well-defined limit.
We defer a more detailed explanation of this point to Remark~\ref{rem:topology.convergence.light};
see also Remark~\ref{rem:HP-vs-us} for a discussion of related work on stretch rays and their interaction with the Thurston asymmetric metric and harmonic stretch lines.

By contrast, if we keep the radius parameters fixed and apply only the interior anti-stretch deformation to the circular foliation, we obtain a limiting description in terms of circle packings.
On the level of coordinates, setting the shear entries in \(\bR^{E(G)}\) equal to zero (denoted by \(\underline{0}\)) selects the subset
\[
  \cSR_G^{-1}\bigl(\{\underline{0}\} \times \bR^n_{>0}\bigr) \subset \cT(G),
\]
whose elements we refer to as \emph{circle-packed hyperbolic cone-surfaces} (see for instance \cite{BHS23}).

\begin{theorem}[Asymptotic behaviour of interior anti-stretch rays]\label{thm:asym.behavior1}
For each \(G \in \TSp\) and each \(X_0 \in \cT(G)\), the interior anti-stretch ray \((X_t)_{t \le 0}\) based at \(X_0\) converges, in \(\cT(G)\), to the hyperbolic circle-packed cone-surface with centers at \(\underline{p}\) which is described by the radius coordinates of \(X_0\).
\end{theorem}

Thus the peripheral stretch and interior anti-stretch rays interpolate between the cusped regime determined by the shear data and the circle-packed locus determined by the radii, and they provide a concrete way to vary cone angles within our charts.
We turn next to the maximal-cone-angle problem in Section~\ref{subsubsec:intro.max-angle}; definitions and proofs of Theorems~\ref{thm:asym.behavior} and~\ref{thm:asym.behavior1} are deferred to Section~\ref{subsec:application}. \subsubsection{Maximal cone angles and the universally triangulable locus}
\label{subsubsec:intro.max-angle}

We use the circular foliation coordinates on the chart domains \(\cT(\underline G)\) and, in particular, on the universally triangulable locus, to obtain a sharp upper bound on admissible cone angles.

For each triangulation \(G \in \TSp\), write \(\cT(G) \subset \cTp\) for the corresponding coordinate domain as in Section~\ref{subsubsec:intro.circ-local}.
Their common intersection \(\cT_{\rm u.t.}(S,\up) := \bigcap_{G \in \TSp} \cT(G)\) is called the \emph{universally triangulable locus} (see Definition~\ref{def:Teich.cone.surf.}).
As an immediate consequence of Theorem~\ref{thm:Teich.chart} we obtain:

\begin{corollary}[Embedding of the universally triangulable locus]\label{cor:embed}
For each triangulation \(G \in \TSp\), the map
\[
  \cCG : \cT_{\rm u.t.}\Sp \longrightarrow \cMFp
\]
is an embedding.
\end{corollary}

By the shear--radius coordinates (Corollary~\ref{cor:homeo.TS.shear.radius}), 
each $\cT(G)$ is homeomorphic to the vertex-balanced cone $\Lambda_G$.
This concrete cone model makes it natural to study how geometric quantities, and in particular the cone angles at the marked points, behave on a fixed chart and on their common intersection $\cT_{\rm u.t.}(S,\underline{p})$.
In Section~\ref{sec:max.cone.angle} we use this parametrization to address the following guiding question:

\begin{question}[Maximal cone angles under triangulability]
\label{q:maximal.angle}
Given a hyperbolic cone-surface \(X\in\cTp\), how large can the cone
angles at the marked points be under triangulability constraints?
More precisely:
\begin{enumerate}
  \item[(a)] for a fixed triangulation \(G\in\TSp\), how large can the cone angles at the marked points be among cone-metrics \(X\in\cT(G)\)?
  \item[(b)] how large can the cone angles at the marked points be along the
  universally triangulable locus \(\cT_{\rm u.t.}(S,\up)\subset\cTp\)?
\end{enumerate}
\end{question}

As a first step towards Question~\ref{q:maximal.angle},
Proposition~\ref{prop:angle.face} determines, for each pair \((G,p)\) with
\(G\in\TSp\) and \(p\in\up\), the supremum of the cone angles at \(p\) among all
cone-metrics in \(\cT(G)\).

We now turn to the universally triangulable locus. We denote by \(\Theta\Sp\) the \emph{universally maximal cone angle} of \(\Sp\), that is, the supremum of the cone angles which can be realized at any marked point by a cone-metric in \(\cT_{\rm u.t.}(S,\up)\) (see Definition~\ref{def:max.cone.angles} for the precise formulation).
Combining Proposition~\ref{prop:angle.face} with the description of the universally triangulable locus, we obtain the following sharp statement.

\begin{theorem}[Value of the universally maximal cone angle]\label{thm:adm.cone.angle}
The following statements hold:
\begin{enumerate}
  \item If \(n = g = 1\), then \(\Theta\Sp = 2\pi\) and \(\cTp = \cT_{\rm u.t.}(S,\up)\).  Otherwise, \(\Theta\Sp = \pi\).
  \item The value \(\Theta\Sp\) cannot be achieved by any element of \(\cT_{\rm u.t.}(S,\up)\), except in the case \(n = 1\) and \(g \ge 2\).
\end{enumerate}
\end{theorem}

As a consequence, we obtain a particularly simple picture in the once-marked torus case.

\begin{corollary}[Global coordinates in the once-marked torus case]
\label{cor:once.punctured.torus.global}
If \(g = n = 1\), then for every triangulation \(G \in \TSp\) the shear--radius map \(\cSR_G : \cTp \longrightarrow \Lambda_G \) is a homeomorphism.  
In other words, the shear--radius coordinates are global coordinates for \(\cT(S,\underline{p})\).
\end{corollary}

The proofs of Theorem~\ref{thm:adm.cone.angle} and Corollary~\ref{cor:once.punctured.torus.global} will be given at the end of Section~\ref{subsec:universal.max.angle}.

\subsection{Organization of the paper}
\label{subsec:outline}

The paper is organized as follows.
Section~\ref{sec:preliminary} fixes notation for arcs, curves and triangulations, and for the associated function spaces.
Section~\ref{sec:MF.} introduces the space \(\cMFp\) of measured foliation classes trivial around \(\underline{p}\); for each triangulation \(G\) we describe the intersection coordinates \(i^{(G)}_{\cdot}\) and construct the shear--radius map \(\mathbf{sr}_G\).

Section~\ref{sec:Teich.} reviews the Teichm\"uller space \(\cTp\) and its \(G\)-triangulable loci \(\cT(G)\), and develops edge-length coordinates on \(\cT(G)\).
In Section~\ref{sec:circ.foli.coor.app.} we construct the circular foliation map \(\cCG\) and establish the shear--radius coordinate charts on \(\cT(G)\) (Theorem~\ref{thm:Teich.chart} and Corollary~\ref{cor:homeo.TS.shear.radius}).

Section~\ref{subsec:application} defines the partial stretch/anti-stretch deformations and proves the asymptotic behaviour of the associated deformation rays (Theorems~\ref{thm:asym.behavior} and~\ref{thm:asym.behavior1}).
Finally, Section~\ref{sec:max.cone.angle} studies maximal cone angles and the universally triangulable locus, and proves Theorem~\ref{thm:adm.cone.angle}.

The logical dependencies between the main definitions and theorems established in this paper are illustrated in Figure~\ref{fig:pipeline}.

\usetikzlibrary{arrows.meta,positioning,fit,matrix,backgrounds,calc}

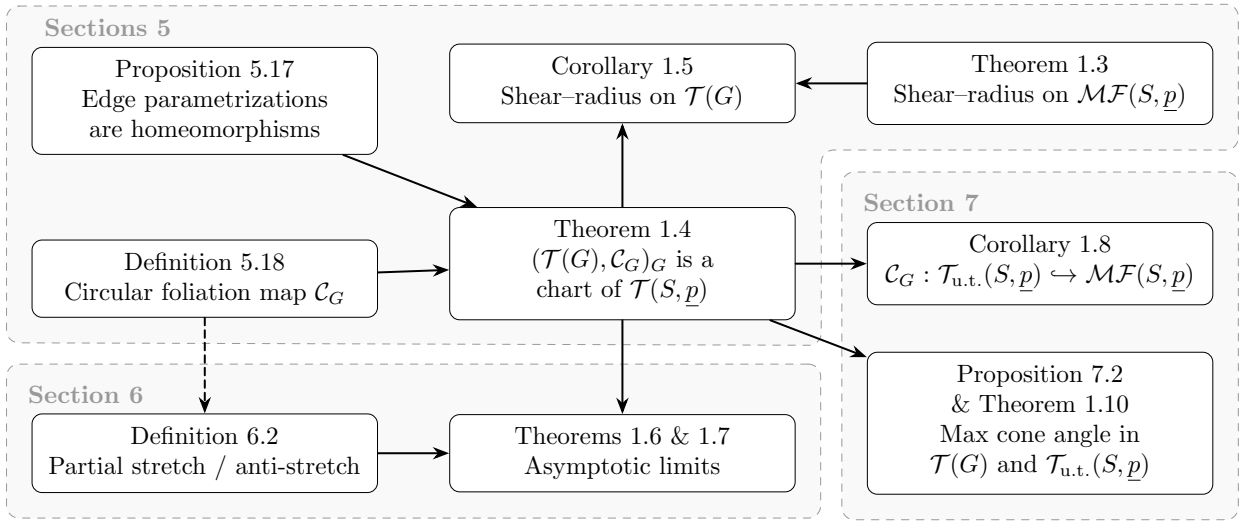
\begin{figure}[h]
\centering
\adjustbox{max width=\linewidth, max totalheight=0.85\textheight, center}{\begin{tikzpicture}[
  >=Stealth,
  node distance=1.6cm and 1.0cm, 
  box/.style={
    draw, 
    rounded corners, 
    inner sep=5pt, 
    font=\small, 
    align=center, 
    text width=4.5cm, 
    fill=white
},
group/.style={
    draw=gray!80,       
    dashed,             
    inner sep=10pt,     rounded corners=5pt,
    fill=gray!5,        
label={[anchor=north west, text=gray!80, font=\bfseries\small, xshift=5pt, yshift=-5pt]north west:Section~\ref{#1}} 
  },
  dep/.style={->, thick},            
  use/.style={->, densely dashed, thick}, 
  every path/.style={line cap=round}
]

\node[box] (thm_Teich_chart) 
  {Theorem~\ref{thm:Teich.chart}\\\((\cT(G),\cC_G)_G\) is a chart of \(\cTp\)};

\node[box, left=of thm_Teich_chart, yshift=-6pt] (def_circ_foli_map) 
  {Definition~\ref{def:circ.foli.map.}\\Circular foliation map \(\cC_G\)};

\node[box, right=of thm_Teich_chart] (cor_embed) 
  {Corollary~\ref{cor:embed}\\\(\cC_G:\cT_{\rm u.t.}(S,\up) \hookrightarrow \cMFp\)};

\node[box, above=of thm_Teich_chart, yshift=-12pt] (cor_homeo_TS_shear_radius) 
  {Corollary~\ref{cor:homeo.TS.shear.radius}\\Shear--radius on \(\cT(G)\)};

\node[box, left=of cor_homeo_TS_shear_radius, yshift=-7pt] (prop_edge_leng) 
  {Proposition~\ref{prop:edge.leng.homeo.}\\Edge parametrizations are  homeomorphisms};

\node[box, right=of cor_homeo_TS_shear_radius] (thm_homeo_MFp_shear_radius) 
  {Theorem~\ref{thm:homeo.MFp.shear.radius}\\Shear--radius on \(\cMFp\)};

\node[box, below=of thm_Teich_chart, yshift=8pt] (thm_asym) 
  {Theorems~\ref{thm:asym.behavior} \&~\ref{thm:asym.behavior1}\\Asymptotic limits};

\node[box, left=of thm_asym] (def_part_stret) 
  {Definition~\ref{def:part.stret.}\\Partial stretch / anti-stretch};

\node[box, right=of thm_asym, yshift=12pt] (angle_results) 
  {Proposition~\ref{prop:angle.face} \& Theorem~\ref{thm:adm.cone.angle}\\Max cone angle in \(\cT(G)\) and \(\cT_{\rm u.t.}(S,\up)\)};

\begin{scope}[on background layer]

\path let \p1 = (prop_edge_leng.west), \p2 = (cor_homeo_TS_shear_radius.north) 
        in coordinate (TopLeft) at (\x1 - 10pt, \y2 + 15pt);
        
\path (thm_homeo_MFp_shear_radius.north east) ++(10pt, 15pt) coordinate (TopRight);
  
\path (thm_homeo_MFp_shear_radius.south east) ++(10pt, -10pt) coordinate (StepRightTop);
  
\path ($(thm_Teich_chart.east)!0.5!(cor_embed.west)$) ++(-4pt, 0pt) coordinate (MidX);
  \coordinate (StepInnerCorner) at (MidX |- StepRightTop);
  
\path (thm_Teich_chart.south -| MidX) ++(0pt, -10pt) coordinate (StepRightBottom);
  
\path (def_circ_foli_map.south west) ++(-10pt, -10pt) coordinate (BottomLeft);

\draw[gray!80, dashed, fill=gray!5, rounded corners=5pt] 
    (TopLeft) -- (TopRight) -- (StepRightTop) -- (StepInnerCorner) -- (StepRightBottom) -- (BottomLeft) -- cycle;
    
\node[anchor=north west, text=gray!80, font=\bfseries\small, inner sep=0pt, xshift=15pt, yshift=-5pt] 
    at (TopLeft) {Sections~\ref{sec:circ.foli.coor.app.}};

\path (thm_asym.north) ++(0, 10pt) coordinate (top_sec6);
  \node[group={subsec:application}, fit=(def_part_stret) (thm_asym) (top_sec6)] {};
  
\path (cor_embed.north) ++(0, 10pt) coordinate (top_sec7);
  \node[group={sec:max.cone.angle}, fit=(cor_embed) (angle_results) (top_sec7)] {};
  
\end{scope}
\draw[dep] (def_circ_foli_map) -- (thm_Teich_chart);
\draw[dep] (prop_edge_leng) -- (thm_Teich_chart);
\draw[dep] (thm_Teich_chart) -- (cor_embed);
\draw[dep] (thm_Teich_chart) -- (cor_homeo_TS_shear_radius);
\draw[dep] (thm_homeo_MFp_shear_radius) -- (cor_homeo_TS_shear_radius);
\draw[use] (def_circ_foli_map) -- (def_part_stret);
\draw[dep] (def_part_stret) -- (thm_asym);
\draw[dep] (thm_Teich_chart) -- (thm_asym);
\draw[dep] (thm_Teich_chart) -- (angle_results);

\end{tikzpicture}
}
\caption{Dependency graph of the main theorems, definitions, and propositions, grouped by section. 
Solid arrows indicate logical implications, while dashed arrows indicate conceptual dependencies. 
The layout highlights the progression from the definition of the circular foliation map to its asymptotic and geometric consequences.}
\label{fig:pipeline}
\end{figure} 
 \section{Preliminaries}
    \label{sec:preliminary}

In this section we fix the combinatorial and functional notation used throughout the paper. 
We recall the classes of arcs and curves on the marked surface \(\Sp\), introduce triangulations and the local combinatorics around vertices, and define the weak topology on functions on \(\cAC\). 
The main object of this section is the admissible domain \(\Omega_G\) (Definition~\ref{def:tri.ineq.region}), which will serve as the target of the edge-length parametrization of \(\cT(G)\) and the edge-intersection parametrization of \(\cMFp\)
(Proposition~\ref{prop:edge.leng.homeo.}) in later sections.

\subsection{Arcs and curves on surfaces}

\subsubsection{Arcs on surfaces}

We begin by fixing the notion of (essential) arcs on \(\Sp\) and their homotopy classes relative to \(\up\).

\begin{definition}
An \DEF{\em (oriented) arc} on \(\Sp\) is a continuous map \(\alpha \colon [0,1] \to S\) with endpoints \(\alpha(0), \alpha(1) \in \up\), such that the image \(\alpha((0,1))\) is contained in \(\dot{S} = S \setminus \up\). 
We allow \(\alpha(0) = \alpha(1)\).
The arc \(\alpha\) is \DEF{\em simple} if it is injective on \([0,1)\).
Two arcs \(\alpha_0\) and \(\alpha_1\) on \(\Sp\) are \DEF{\em homotopic relative to \(\up\)} if there exists a continuous family \(\{ f_t \colon [0,1] \to S \}_{t \in [0,1]}\) such that
\begin{itemize}
  \item \(f_0 = \alpha_0\) and \(f_1 = \alpha_1\);
  \item \(f_t(0)\) and \(f_t(1)\) are constant in \(\up\) for all \(t \in [0,1]\);
  \item \(f_t((0,1)) \subset \dot{S}\) for all \(t \in [0,1]\).
\end{itemize}
We say that $\gamma$ is \DEF{\em essential} if it is not homotopic (relative to $\up$) in $S$ to a point.
We denote by \(\cA\Sp\) the set of homotopy classes (relative to \(\up\)) of essential arcs on \(\Sp\).
\end{definition}

Two (distinct) homotopy classes of arcs are \DEF{\em disjoint} if they admit representatives whose interiors are pairwise disjoint in \(\dot{S}\).

\subsubsection{Curves on surfaces}
We now recall the corresponding notation for curves on \(\Sp\).

\begin{definition}
An \DEF{\em (oriented) curve} on \(\Sp\) is a continuous map \(\gamma \colon [0,1] \to S\) with \(\gamma(0) = \gamma(1)\), whose image is contained in $\dot{S}$. 
The curve \(\gamma\) is \DEF{\em simple} if it is injective on \([0,1)\).
We say that \(\gamma\) is \DEF{\em essential} if it is not homotopic in \(\dot{S}\) to a point or to a peripheral curve around a puncture.
Two curves \(\gamma_0\) and \(\gamma_1\) are \DEF{\em homotopic} if they are related by a homotopy \(\{ g_t \colon [0,1] \to \dot{S} \}_{t \in [0,1]}\) with \(g_0 = \gamma_0\) and \(g_1 = \gamma_1\).
We denote by \(\cC\Sp\) (resp. \(\cS\Sp\)) the set of homotopy classes of essential curves (resp. essential simple curves) on \(\Sp\).
\end{definition}

For convenience, we identify a homotopy class of arcs or curves with any of its representatives, and we tacitly assume that all arcs and curves under consideration are essential.
When no confusion arises, we use the same symbol \(\alpha\) (resp. \(\gamma\)) for both an arc (resp. a curve) and its homotopy class. 
\subsection{Triangulations of surfaces}
    \label{subsec:triangulation}
Triangulations of \(\Sp\) will provide the combinatorial scaffolding for our coordinate charts on \(\cTp\). 
In this subsection we fix the terminology for topological triangles, triangulations by arcs, and the local star at a marked point.

\begin{definition}
\label{def:top.triang.}
A \DEF{\em topological triangle} \(\Delta\) is a closed disc \(D \subset S\) together with three distinguished points on \(\partial D\), labelled in counterclockwise order as \(v_1, v_2, v_3\).
These points are the \DEF{\em vertices} of \(\Delta\) and form the set \(V(\Delta)\).
The closure of each connected component of \(\partial D \setminus V(\Delta)\) is called an \DEF{\em edge} of \(\Delta\); the set of edges is denoted by \(E(\Delta)\).
\end{definition}

\begin{definition}
\label{def:triangulation}
A \DEF{\em triangulation} \(G\) of \(\Sp\) is a maximal collection \(\{ \alpha_1, \dots, \alpha_k \}\) of mutually disjoint homotopy classes of essential simple arcs on \(\Sp\). The marked points in \(\up\) are the \DEF{\em vertices} of \(G\), and the homotopy classes \(\alpha_i\) are the \DEF{\em edges} of \(G\).
We write \(E(G)\) for the set of edges.
Given a choice of pairwise (internally) disjoint representatives for all edges of \(G\), the connected components of the complement
\(  \Sp \setminus \bigcup_{i=1}^k \alpha_i \)
are called the \DEF{\em faces} of \(G\); their set is denoted by \(F(G)\).
Each face \(f \in F(G)\) has a naturally associated topological triangle \(\Delta_f\), obtained by taking the closure of \(f\) and recording its vertices and edges.
We denote by \(\TSp\) the set of all triangulations of \(\Sp\).
\end{definition}

\begin{remark}[Terminology and conventions]
Our triangulations are purely topological: they are maximal collections of pairwise disjoint essential simple arc classes on \(S\). We allow self-folded faces (once-punctured monogons) and multiple edges between the same pair of vertices (see Figure \ref{fig:Delta} for instance), so the resulting decomposition is a CW-complex rather than a simplicial complex.
This convention agrees with the standard notion of an \emph{ideal triangulation by arcs} in the literature, see for instance \cite{Pen87, FST08}.
\end{remark}

A simple Euler characteristic computation shows that, for any triangulation \(G \in \TSp\), the number of edges and faces is given by
\[
  \lvert E(G) \rvert = 6g - 6 + 3n, \qquad
  \lvert F(G) \rvert = 4g - 4 + 2n,
\]
see also \cite[Prop.~2.10]{FST08}.

For later use we record the local combinatorics around a marked point.

\begin{definition}
\label{def:face.star}
Let \(G \in \TSp\) and \(p \in \up\).
Let \(E_p(G) \subset E(G)\) denote the set of edges of \(G\) incident to \(p\), and let \(F_p(G) \subset F(G)\) denote the set of faces incident to \(p\).
The set \(E_p(G)\) is called the \DEF{\em star at \(p\) in \(G\)}.
\end{definition}

\subsection{Function spaces on arcs and curves}

We now introduce the function spaces that will support our coordinate systems.
For convenience we write
\[
  \cAC := \cA\Sp \cup \cC\Sp.
\]

\begin{definition}[Weak topology on function spaces]
\label{def:weak.top.func.space}
Let \(\bRn^{\cAC}\) denote the space of functions \(f \colon \cAC \to \bR\).
We equip \(\bRn^{\cAC}\) with the \DEF{\em weak topology} (or topology of pointwise convergence) defined as follows.
For any \(f_0 \in \bRn^{\cAC}\), any finite subset \(\underline{\beta} \subset \cAC\) and any collection of positive numbers \(\underline{\epsilon} = (\epsilon_{\beta})_{\beta \in \underline{\beta}}\), we set
\[
  \cU_{f_0}(\underline{\beta}, \underline{\epsilon})
  :=
  \left\{
    f \in \bRn^{\cAC}
    \;\middle|\;
    \lvert f(\beta) - f_0(\beta) \rvert < \epsilon_{\beta}
    \text{ for all } \beta \in \underline{\beta}
  \right\}.
\]
The sets \(\cU_{f_0}(\underline{\beta}, \underline{\epsilon})\) form a neighbourhood basis of \(f_0\) and hence define a topology on \(\bRn^{\cAC}\), called the weak topology.
Similarly, for the set \(\cS\subset \cAC\) of simple closed curve classes,  we denote by \(\bRn^{\cS}\) the space of functions \(f:\cS\to\bR\),
equipped with the weak topology defined by the analogous neighbourhood basis
\(\cU_{f_0}(\underline{\beta},\underline{\epsilon})\), where \(f_0\in\bR^{\cS}_{\geq 0}\) and \(\underline{\beta}\subset\cS\) is finite.
\end{definition}

\begin{remark}[Relation with the topology in FLP]
The topologies in Definition~\ref{def:weak.top.func.space} are modeled on the weak topology used in the study of simple closed curves and measured foliations on compact orientable surfaces (possibly with boundary). 
In particular, the space of measured foliations is topologized via the corresponding intersection functionals into \(\bRn^{\cS}\) \cite[Expos\'e~4]{FLP12}.
\end{remark}

For each triangulation \(G \in \TSp\), the edge set \(E(G)\) is a subset of \(\cA\Sp\), and we obtain the restriction space \(\bRn^{E(G)}\)  by restricting functions to edges. It inherits the weak topology from \(\bRn^{\cAC}\). 
We will later single out an open cone \(\Omega_G \subset \bRp^{E(G)}\), which will appear in the parametrizations of \(\cMFp\) and of the chart \(\cT(G)\) in Proposition~\ref{prop:edge.leng.homeo.}.

\begin{definition}[Admissible domain]
\label{def:tri.ineq.region}
Let \(G \in \TSp\).
The \DEF{\em admissible domain} \(\Omega_G \subset \bR^{E(G)}_{>0}\) is the subset of edge weightings \(\ell \colon E(G) \to \bR_{>0}\) such that, for every face \(f \in F(G)\) bounded by edges \(e_1, e_2, e_3\), the strict triangle inequalities
\[
  \ell(e_i)+\ell(e_j)>\ell(e_k)
\]
hold for all distinct indices \(i,j,k \in \{1,2,3\}\).
Equivalently, \(\Omega_G\) consists of all assignments of positive lengths to the edges of \(G\) for which the associated triangles \(\Delta_f\) can be realized as hyperbolic geodesic triangles.
In particular, \(\Omega_G\) is an open cone in \(\bR^{E(G)}_{>0}\).
\end{definition}

  \section{Measured foliations}
    \label{sec:MF.}

This section develops the combinatorial model space for our theory: the space \(\cMFp\) of measured foliations that are trivial around the marked points. Relative to a fixed triangulation \(G\in\TSp\), we construct two coordinate systems:
\begin{enumerate}
    \item the \emph{edge intersection coordinates} \(\iG_{\cdot}\) (Proposition~\ref{prop:edge.length.MF.biject.}), which identify \(\cMFp\) with the admissible cone \(\Omega_G\) of edge weights satisfying the triangle inequalities;
    \item the \emph{shear--radius coordinates} \(\mathbf{sr}_G\) (Theorem~\ref{thm:homeo.MFp.shear.radius}), which decompose the foliation data into interior shears and peripheral radii and identify \(\cMFp\) with the vertex-balanced cone \(\Lambda_G\) (Definition~\ref{def:vertex.balanc.part}).
\end{enumerate}
As recalled in the introduction, these coordinates are designed to match the shear--radius coordinates on the $G$-triangulable locus \(\cT(G)\) of Teichm\"uller space (Definition \ref{def:Teich.cone.surf.}) and to parametrize the partial (anti)-stretch deformations introduced in Section~\ref{subsec:application}.

\subsection{Measured foliations near the marked points and intersection functions}
\label{subsec:MF.and.geo.inter.}

We only recall the minimal background needed to fix notation, referring to \cite{FLP12} and \cite{PP93} for details.

    \subsubsection{Foliations on marked surfaces}
\label{subsec:MF.surf.}

Let \(\cMF(S)\) denote the space of Whitehead equivalence classes of measured foliations on the closed surface \(S\).
We consider foliations on the marked surface \(\Sp\) up to isotopy and Whitehead moves that do not cross the marked points.
By \cite[Exp.~5]{FLP12}, any class \([F]\) admits a representative whose singularities in \(\dot{S} = S \setminus \up\) are all three-pronged.
The \emph{geometric intersection number} \(i([F], [\beta])\) is defined in the usual way for any class \([\beta] \in \cAC\).

Our focus is on foliations compatible with the cone singularities.
Following Mosher \cite{Mos88} and Papadopoulos--Penner \cite{PP93}, we say that a foliation is \DEF{\em trivial around a marked point \(p_i\)} if, in a punctured neighbourhood of \(p_i\), it is a foliation by closed leaves forming an annulus.   

\begin{definition}[Foliation classes trivial around marked points]
    \label{def:MF.surf.trivial.}
    A measured foliation class on $\Sp$ is \DEF{\em trivial around \(\up\)} if it has a representative trivial around each marked point in $\up$. Let \DEF{$\cMFp$} denote the space of measured foliation classes on \(\Sp\) that are trivial around \(\up\).
\end{definition}

\begin{remark}[Conventions]
\label{rem:MFp.convention}
For \([F]\in\cMFp\) we adopt the following standing assumptions, which match the domain of our coordinate systems:
\begin{enumerate}
  \item the intersection number \(i([F],[\beta])\) is finite for all \([\beta]\in\cAC\);
  \item for each marked point \(p_i\), the restriction of \(F\) to a neighbourhood of \(p_i\) carries positive transverse measure.
\end{enumerate}
\end{remark}
 \subsubsection{Intersection functions}
\label{subsubsec:inter.func.}

The topology on \(\cMFp\) is the weak topology induced by the geometric intersection functions
\[
i^{(\cAC)}_{\cdot}([F]) = \big(i([F],[\beta])\big)_{[\beta]\in\cAC}.
\]
Restricting these intersection functions to the finite subset of \(\cAC\) given by the edges of a triangulation yields our first coordinate candidate.

\begin{definition}[Edge intersection map]
\label{def:MF.edge.inter.}
For a triangulation \(G\in\TSp\), the \DEF{\em edge intersection map} is
\[
\iG_{\cdot}: \cMFp \longrightarrow \bR_{\ge0}^{E(G)}, \quad [F] \longmapsto \big(i([F],e)\big)_{e\in E(G)}.
\]
\end{definition}

     \subsubsection{The product structure and radii} \label{subsec:prod.str.MF.}

The condition of triviality around \(\up\) induces a natural decomposition.
Any measured foliation \(F\) representing a class in \(\cMFp\) decomposes uniquely into:
\begin{itemize}
    \item \emph{peripheral components} \(\{F^{(i)}\}_{i=1}^n\): foliated annuli by closed leaves around each \(p_i\);
    \item an \emph{interior component} \(F^{(0)}\): the restriction of \(F\) to the complement of the peripheral annuli.
\end{itemize}
This geometric decomposition induces a homeomorphism \cite[Prop.~2.2]{PP93}
\begin{equation}
\label{map:prod.str.MF}
    d: \cMFp \xrightarrow{\cong} \cMFo \times \prod_{i=1}^n \cMF(p_i),
\end{equation}
where \(\cMFo\) consists of foliations with compact support in \(\dot{S}\) (equivalently, with empty peripheral components), and \(\cMF(p_i) \cong \bR_{>0}\) parametrizes the transverse width of the annulus at \(p_i\).

\begin{definition}[Radius]
\label{def:radius.peripheral}
For \([F]\in\cMFp\), the \DEF{\em radius} \(r_i([F])\) at \(p_i\) is the total transverse measure of the peripheral annulus \(F^{(i)}\); equivalently, \(r_i([F])\) is the transverse measure of any arc that joins \(p_i\) to the boundary of the interior component \(F^{(0)}\) and crosses \(F^{(i)}\) exactly once.
\end{definition}

\begin{remark}
    Our `trivial around $\up$' convention matches the decorated measured foliations of \cite{PP93}: 
    the interior component corresponds to a compactly supported foliation carried by the dual null-gon track, while the peripheral radii are the collar weights on the null-gon components (with each equal to the transverse measure of a connecting arc across the respective foliated annulus) \cite[Sec. 2]{PP93}. 
    The resulting map \eqref{map:prod.str.MF} agrees with the fiber-bundle description in \cite[Prop. 2.2]{PP93}.
\end{remark}

\subsection{Measured foliations totally transverse to a triangulation}
    \label{subsec:MF.trans.triang.}

We recall the local models for measured foliations on a single topological triangle; see also \cite[Def.~3.9]{PT07}.

\begin{definition}[Foliations trivial around a point in a triangle]
    \label{def:MF.triv.pt}  
Let $\Delta$ be a topological triangle (Definition~\ref{def:top.triang.}), and let $F$ be a measured foliation on $\Delta$.  
\begin{itemize}
    \item For $x\in\partial\Delta$, we say that $F$ is \DEF{\em trivial around $x$} if there is a neighbourhood $U$ of $x$ in $\Delta$ such that all leaves of $F$ in $U$ are arcs with endpoints on the two sides of $x$ along $\partial\Delta$.
    \item For $x\in\mathrm{int}(\Delta)$, we say that $F$ is \DEF{\em trivial around $x$} if there is a neighbourhood $U$ of $x$ in $\Delta$ on which the leaves of $F$ are simple closed curves encircling $x$ (the standard local model for measured foliations).
\end{itemize}
\end{definition}

See Figure~\ref{fig:MF.triang.}(a) for an example of a foliation trivial around interior and boundary points.

\begin{definition}[Foliation class totally transverse to a triangle]
    \label{def:MF.trans.tri.}
Let $\Delta$ be a topological triangle with edges $e_1,e_2,e_3$.
A measured foliation $F$ on $\Delta$ is \DEF{\em totally transverse to $\Delta$} if
\begin{enumerate}
    \item $F$ is transverse to $\partial\Delta$;
    \item $F$ is trivial exactly around each vertex of $\Delta$ in the sense of Definition~\ref{def:MF.triv.pt};
    \item the geometric intersection number $i(F,e_j)$ is strictly positive for each edge $e_j$.
\end{enumerate}
We denote by $\cMF(\Delta^{\perp})$ the space of measured foliations on $\Delta$ that are totally transverse to $\Delta$, modulo isotopy and Whitehead moves fixing $\partial\Delta$.
\end{definition}

Examples of totally transverse foliations are shown in Figure~\ref{fig:MF.triang.}(b)–(c).

\begin{figure}[htbp]
    \centering
    \begin{subfigure}{0.28\textwidth}
        \centering
\includegraphics[width=\linewidth]{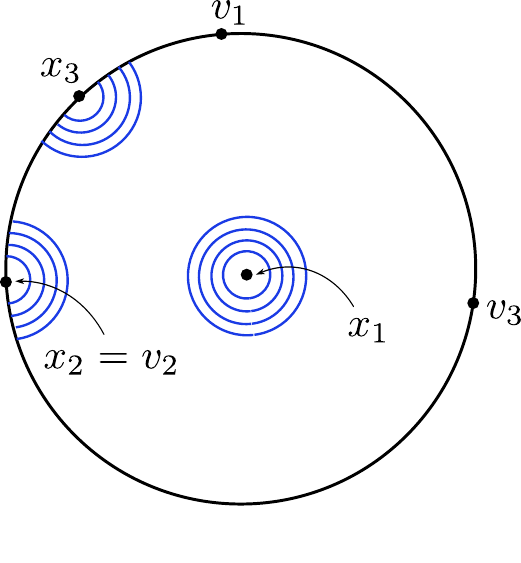}
        \caption{}
\label{subfig:MF.triang.trivial.}
    \end{subfigure}
\begin{subfigure}{0.28\textwidth}
        \centering
\includegraphics[width=\linewidth]{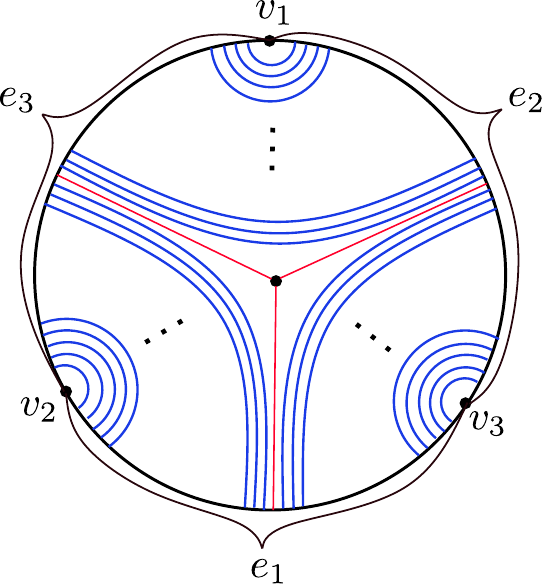}
     \caption{}
\label{subfig1:MF.tria.ineq.}
    \end{subfigure}
\begin{subfigure}{0.28\textwidth}
        \centering     \includegraphics[width=\linewidth]{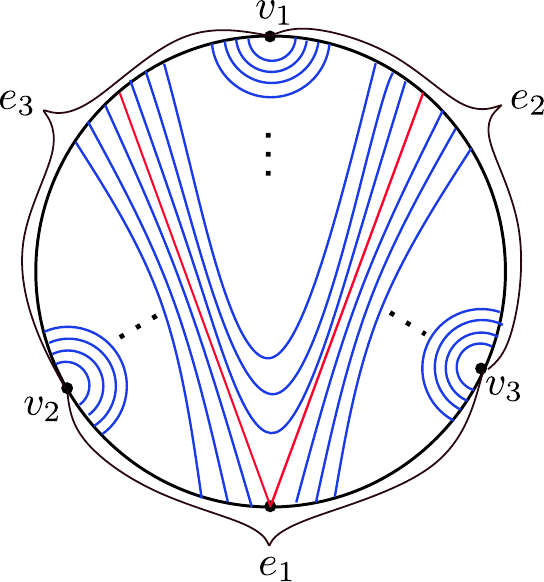}
        \caption{}
        \label{subfig2:MF.tria.ineq.}
    \end{subfigure}
    \caption{The left picture is an example of a measured foliation (partially drawn) on a topological triangle $\Delta$, which is trivial around an interior point $x_1$ and two boundary points $x_2$ and $x_3$, where $x_2$ is a vertex of $\Delta$. The middle and right pictures are two examples of measured foliations (partially drawn) on $\Delta$ which are totally transverse to $\Delta$.}
    \label{fig:MF.triang.}
\end{figure}

\begin{remark}
Condition (1) together with Condition ~(3) ensures that the transverse measure of $F$ on any nondegenerate sub-arc contained in an edge of $\Delta$ is strictly positive.
This matches the terminology “totally transverse (and trivial near punctures)” on punctured surfaces in \cite[Sec.~3]{PT07} and the standard properties of transverse measures in \cite[Exp.~5]{FLP12}.
\end{remark}

We now globalize the previous definition to a triangulation of the marked surface.

\begin{definition}[Foliation class totally transverse to $G$]
    \label{def:MF.trans.triang.surf.}
Let $G\in \TSp$ be a combinatorial triangulation of the marked surface $\Sp$.
A measured foliation $F$ on $\Sp$ is said to be \DEF{\em totally transverse to $G$} if for every triangle $\Delta_f$ associated to a face $f$ of $G$, the restriction $F|_{\Delta_f}$ is totally transverse to $\Delta_f$ in the sense of Definition~\ref{def:MF.trans.tri.}.
We denote the corresponding space of foliation classes by $\cMF(G^{\perp})$.
\end{definition}

\begin{remark}
\label{rem:MF.Gperp.to.MFp}
By construction, a measured foliation totally transverse to a triangulation \(G\in\TSp\) is transverse to every edge of \(G\), trivial around each vertex, and hence trivial around the set \(\up\) of marked points.
In particular, we have \(\cMF(G^{\perp}) \subset \cMFp\).
\end{remark}

\subsection{Triviality near the marked points and transversality to a triangulation}

The aim of this subsection is to identify, for each triangulation \(G \in \TSp\), measured foliations that are trivial around \(\up\) with those that are totally transverse to \(G\).
Combined with the triangular coordinates on a single face (Lemma~\ref{lem:MF.triang.ineq.bij}), this will yield the parametrization of \(\cMFp\) by the edge intersection map \(\iG_{\cdot}\) with values in the admissible cone \(\Omega_G \subset \bR^{E(G)}_{>0}\) (Proposition~\ref{prop:edge.length.MF.biject.}).  This parametrization will be used later in the proof of Theorem~\ref{thm:homeo.MFp.shear.radius} in Section~\ref{subsec:invers.contin.}.

\begin{proposition}[Equivalence of triviality around vertices and transversality to a triangulation]
 \label{prop:t.a.imply.t.t.G}
    For each \(G\in\TSp\), we have
    \[
        \cMFp = \cMF(G^{\perp}).
    \]
\end{proposition}

To prove this proposition we work in the train track framework and, throughout, adopt the conventions of Penner--Harer for switches, branches and the equivalence relation of train tracks\cite{PH92}.

\begin{definition}[Train track]
    \label{def:train.track}
    A \DEF{\em train track} \(\tau\) on \(\Sp\) is a closed \(C^1\) 1--complex in \(\dot{S}\) whose vertices are called \DEF{\em switches} and whose edges are called \DEF{\em branches}, such that at each switch \(s\) there is a well-defined tangent line and a neighbourhood of \(s\) in \(\tau\) is a union of smoothly embedded open arcs.
    Throughout the paper, all train tracks are assumed to be \DEF{\em trivalent}, i.e. each switch has exactly three adjacent branches.

    A measured foliation class \([F]\in\cMF(S)\) is \DEF{\em (weakly) carried} by \(\tau\) if there is a representative of \([F]\) whose leaves can be \(C^1\)-isotoped into an arbitrarily small neighbourhood of \(\tau\) so that all tangent directions are arbitrarily close to those of \(\tau\).

    Two train tracks \(\tau,\tau'\) on \(\Sp\) are \DEF{\em equivalent} if one can be obtained from the other by a finite sequence of shifts, splittings, collapses (inverse splittings) and isotopies of \((S,\up)\) fixing each marked point (see \cite[Sec.~2.1]{PH92}).
\end{definition}

\begin{definition}[Train track dual to a triangulation]
\label{dual.tt.G}
Let \(G\in\TSp\).
A train track \(\tau\) on \((S,\underline{p})\) is said to be \DEF{\em dual} to \(G\) if, for each face \(f\in F(G)\), the intersection \(\tau\cap f\) has exactly three branches in the interior of \(f\) bounding a trigon (a disk with three spikes) and exactly one branch intersecting each edge of \(f\), as illustrated in Figure~\ref{fig:one.junction}.
By \cite[Sec.~2]{PP93} and \cite[Sec.~5]{PP05}, there exists a unique train track dual to \(G\) up to equivalence.
We denote this equivalence class, and any fixed representative in it, by \(\tau_G\).

\begin{figure}[htp]
    \centering
    \includegraphics[width=4cm]{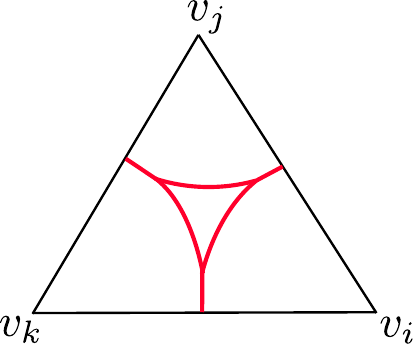}
    \caption{The intersection of the dual train track \(\tau_G\) with a face \(f\) of \(G\).}
    \label{fig:one.junction}
\end{figure}
\end{definition}

Given any \(G\in\TSp\) and any \([F]\in\cMFp\), each peripheral component \([F^{(i)}]\) (\(1\le i\le n\)) around \(p_i\in\up\) is naturally carried by \(\tau_G\).
Moreover, Papadopoulos--Penner show that every measured foliation with compact support in \(\dot{S}\) is carried by the train track \(\tau_G\) dual to \(G\) \cite[Lemma~2.1]{PP93}.
In particular, the interior component \([F^{(0)}]\) of a class \([F]\in\cMFp\) has compact support, hence is carried by \(\tau_G\).
The following is an immediate consequence. 

\begin{lemma}[Measured foliations are carried by the dual train track]
 \label{lem:int.comp.carried by.dual.tt.}
 For each \(G\in\TSp\) and each \([F]\in\cMFp\), \([F]\) is carried by the dual train track \(\tau_G\). In particular, each $[F]\in \cMF(\Delta^{\perp})$ is carried by the dual train track $\tau_{\Delta}$ (as shown in Figure \ref{fig:one.junction}).
\end{lemma}

The following lemma will be used repeatedly to identify measured foliations carried by a given train track from their induced branch weights.

\begin{lemma}[Uniqueness from branch weights]\label{lem:wtt.ml}
Let \(\Sigma\) be either the marked surface \((S,\underline p)\) or a marked triangle \(\Delta\).
Let \(\tau\) be a train track on \(\Sigma\).
If two measured foliation classes \([F_1]\) and \([F_2]\) on \(\Sigma\) are carried by \(\tau\) and induce the same non-negative weight on every branch of \(\tau\), then \([F_1]=[F_2]\).
Equivalently, the branch-weight map is injective on the set of measured foliation classes carried by \(\tau\).
\end{lemma}

\begin{proof}
Via the one-to-one correspondence between measured foliation classes and measured laminations \cite[Thm.~1]{Lev83},
the measured foliation classes \([F_1]\) and \([F_2]\) determine measured laminations carried by \(\tau\) with the same
branch weights. By the uniqueness of a carried measured lamination from its branch weights
\cite[Lem.~1.7.11]{PH92}, these laminations agree, hence \([F_1]=[F_2]\).
\end{proof}

We now express totally transverse foliations on $\Delta$ in terms of their edge intersection numbers.
Let $\Omega_{\Delta}\subset \bR_{>0}^{E(G)}$ be the open cone of triples $(\ell_1,\ell_2,\ell_3)$ satisfying the triangle inequalities.

\begin{lemma}[Triangular coordinates]
\label{lem:MF.triang.ineq.bij}
Let $\Delta$ be a triangle with edges $e_1,e_2,e_3$.
The following map is a bijection:
\[
i^{(\Delta)}_{\cdot} : \cMF(\Delta^{\perp}) \longrightarrow \Omega_{\Delta},
\qquad
[F] \longmapsto \big(i([F],e_j)\big)_{j=1}^3,
\]
where $i([F],e_j)$ denotes the geometric intersection number of $F$ with $e_j$.
\end{lemma}

\begin{proof}
Surjectivity follows by explicitly constructing the standard tripod foliation
for any triple of positive edge-weights satisfying the triangle inequalities;
see Figure~\ref{subfig1:MF.tria.ineq.} for the local model and
Construction~\ref{con:circ.foli.surf.} or \cite[Case~0]{Mos88} for details.
It remains to show the injectivity. Let \([F_1], [F_2]\in\cMF(\Delta^{\perp})\) with the same image under $i^{(\Delta)}_{\cdot}$.  By Lemma \ref{lem:int.comp.carried by.dual.tt.}, both \([F_1]\) and \([F_2]\) are carried by the dual train track $\tau_{\Delta}$.
Denote by \(b_i\) the branch of \(\tau_{\Delta}\) crossing \(e_i\),
and by \(b^{(0)}_i\) the interior branch of \(\tau_{\Delta}\) surrounding the vertex opposite to \(e_i\) (\(i=1,2,3\), indices modulo \(3\)), see e.g. Figure \ref{fig:MF.tota.trans.}.
Let \(w_j(\cdot)\) be the non-negative weight function on branches induced by a representative of \([F_j]\) (\(j=1,2\)).
By definition of the edge intersection numbers for foliation classes transverse to \(\Delta\), the weights on the crossing branches satisfy
\[
w_1(b_i)= i([F_1],e_i) = i([F_2],e_i)= w_2(b_i), \qquad i=1,2,3.
\]
The switch conditions of \(\tau_{\Delta}\) determine uniquely the interior weights \(w_j(b^{(0)}_i)\) from the three crossing weights \(w_j(b_1),w_j(b_2),w_j(b_3)\); concretely, for each \(j=1,2\), \(w_j(b^{(0)}_1),w_j(b^{(0)}_2),w_j(b^{(0)}_3)\) are the unique solution to
\[
w_j(b^{(0)}_i) + w_j(b^{(0)}_{i+1}) = w_j(b_{i+2}), \qquad i=1,2,3.
\]
Since \(w_1(b_i)=w_2(b_i)\) for \(i=1,2,3\), we obtain \(w_1(b^{(0)}_i)=w_2(b^{(0)}_i)\) for \(i=1,2,3\). Therefore,  \([F_1]\) and \([F_2]\)  induce the same branch weights.
Applying Lemma \ref{lem:wtt.ml} yields that $[F_1]=[F_2]$. This concludes the lemma.
\end{proof}

The next lemma is the key combinatorial ingredient: it translates triviality around \(\up\) into the triangle inequalities and total transversality to \(G\).

\begin{lemma}
    \label{lem: suff.card.tt.G}
Let \(G\in\TSp\) and let \([F]\in\cMFp\).
Write
\[
i^{(G)}_{[F]} := \big(i([F],e)\big)_{e\in E(G)} \in \bR^{E(G)}_{\ge 0}
\]
for the vector of geometric intersection numbers with the edges of \(G\).
Then:
\begin{itemize}
\item[(i)] \(i^{(G)}_{[F]} \in \Omega_G\);
\item[(ii)] \([F]\) is totally transverse to \(G\), i.e. \([F]\in \cMF(G^{\perp})\).
\end{itemize}
\end{lemma}

\begin{proof}
We begin with (i).
Let \(f\) be a face of \(G\) and let \(\Delta_f\) be the corresponding triangle.
Label its vertices by \(v_1,v_2,v_3\) in anticlockwise order and denote by \(e_i\) the edge opposite to \(v_i\) (\(i=1,2,3\)).

By Lemma~\ref{lem:int.comp.carried by.dual.tt.}, \([F]\) is carried by the dual train track \(\tau_G\), while the restriction \([F]|_{\Delta_f}\) of $[F]$ to $\Delta_f$ is carried by $\tau_G|_{\Delta_f}=\tau_{\Delta_f}$. Keep the notation of the branches of $\tau_{\Delta_f}$ the same as those for $\tau_{\Delta}$ in the proof of Lemma \ref{lem:MF.triang.ineq.bij}. Then for \(i=1,2,3\), we have
\[
w(b_i)=i([F], e_i), \quad 
w(b^{(0)}_i) + w(b^{(0)}_{i+1}) = w(b_{i+2}),
\]
with indices modulo $3$, and
\begin{align*}
w(b_i) + w(b_{i+1}) 
&= w(b^{(0)}_{i+1})+w(b^{(0)}_{i+2}) +w(b^{(0)}_{i+2})+w(b^{(0)}_{i})\\
&=w(b_{i+2})+2w(b^{(0)}_{i+2})\\
&>w(b_{i+2}),
\end{align*}
where the last inequality follows from \(w(b^{(0)}_i)>0\) for \(i=1,2,3\), since each peripheral component $[F^{(i)}]$ around $v_i$ of $[F]$ has positive radius (Remark~\ref{rem:MFp.convention}) and its restriction to $\Delta_f$ is carried by the branch $b^{(0)}_i$. 
Thus the triple \((w(b_1),w(b_2),w(b_3))\) satisfies the strict triangle inequalities.

As this holds for every face, we conclude that \(i^{(G)}_{[F]} \in \Omega_G\).

\begin{figure}[htp]
    \centering
    \includegraphics[width=7.5cm]{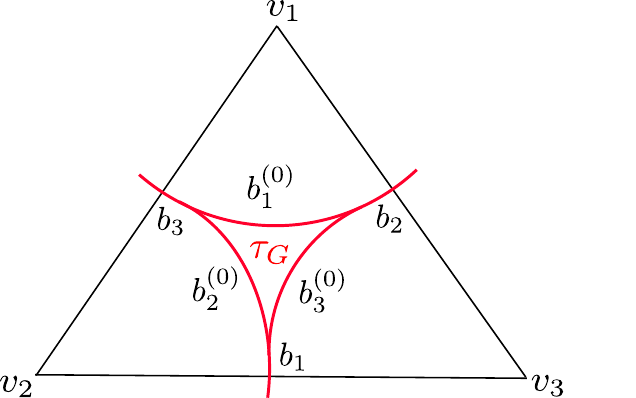}
    \caption{A measured foliation (partially drawn in a face) carried by \(\tau_G\). }
    \label{fig:MF.tota.trans.}
\end{figure}

For (ii), we use the triangular coordinates from Lemma~\ref{lem:MF.triang.ineq.bij}.
By (i), the restriction of \(i^{(G)}_{[F]}\) to the three edges of any face \(f\) of \(G\) defines an element of \(\Omega_{\Delta_f}\).
Hence Lemma~\ref{lem:MF.triang.ineq.bij} yields a unique measured foliation \(F_{\Delta_f}\) on \(\Delta_f\) which is totally transverse to \(\Delta_f\) and satisfies
\[
i(F_{\Delta_f},e) = i([F],e)
\]
for each edge \(e\) of \(\Delta_f\).
Gluing the foliations \(F_{\Delta_f}\) along the edges of \(G\) produces a measured foliation \(F_0\) on \(\Sp\) which is totally transverse to \(G\) by construction.
We denote its class by \([F_0]\).

The foliation \(F_0\) is carried by \(\tau_G\) and induces on each branch of \(\tau_G\) the same weights as \([F]\): this is immediate on the branches crossing the edges of \(G\), and on the interior branches of each face it follows from the switch conditions (as discussed in the proof of Lemma \ref{lem:MF.triang.ineq.bij}). Therefore \([F]\) and \([F_0]\) define the same non-negative weight on every branch of \(\tau_G\).
By Lemma \ref{lem:wtt.ml}, \([F] = [F_0]\).
In particular, \([F]\) is totally transverse to \(G\), proving (ii).
\end{proof}

We are now ready to prove Proposition~\ref{prop:t.a.imply.t.t.G}, which identifies the notion of triviality around \(\up\) and the notion of total transversality to an arbitrary triangulation $G\in\TSp$.

\begin{proof}[Proof of Proposition~\ref{prop:t.a.imply.t.t.G}]
We show the two inclusions separately.

First, the inclusion \(\cMF(G^{\perp}) \subseteq \cMFp\) is exactly Remark~\ref{rem:MF.Gperp.to.MFp}.

For the reverse inclusion, let \([F]\in\cMFp\).
Lemma~\ref{lem: suff.card.tt.G}(ii) applied to \(G\) shows that \([F]\) is totally transverse to \(G\), hence \([F]\in\cMF(G^{\perp})\).

Thus \(\cMFp = \cMF(G^{\perp})\) for every \(G\in\TSp\), as claimed.
\end{proof}

We can now deduce the parametrization of \(\cMFp\) by edge intersection numbers.

\begin{proposition}[Parametrization of measured foliations by edge intersections]
\label{prop:edge.length.MF.biject.}
For any \(G\in\TSp\), the edge intersection map
\[
\iG_{\cdot} : \cMFp \longrightarrow \OG \subset \bR^{E(G)}_{\geq 0},
\]
defined in Definition~\ref{def:MF.edge.inter.} by
\[
\iG_{\cdot}([F]) := \big(i([F],e)\big)_{e\in E(G)},
\]
is a bijection.
\end{proposition}

\begin{proof}
By Proposition~\ref{prop:t.a.imply.t.t.G} we have \(\cMFp=\cMF(G^{\perp})\).
Thus we may work throughout with representatives that are totally transverse to \(G\).
We prove that \(\iG_{\cdot}\) is well-defined, surjective and injective.

\emph{(1) Well-definedness.}
Let \([F]\in\cMFp=\cMF(G^\perp)\).
By Lemma~\ref{lem:int.comp.carried by.dual.tt.}, \([F]\) is carried by \(\tau_G\).
Moreover, Lemma~\ref{lem: suff.card.tt.G}(i) implies that the edge-intersection vector
\(
\big(i([F],e)\big)_{e\in E(G)}
\)
satisfies the required inequalities of \(\Omega_G\), i.e.
\[
\iG_{\cdot}([F]) = i^{(G)}_{[F]} \in \Omega_G.
\]
Hence \(\iG_{\cdot}\) is well-defined as a map \(\cMFp\to\Omega_G\).

\emph{(2) Surjectivity.}
Let \((a_e)_{e\in E(G)}\in\Omega_G\).
For each face \(f\in F(G)\), the restriction of \((a_e)\) to the three edges of \(f\) lies in \(\Omega_{\Delta_f}\).
By Lemma~\ref{lem:MF.triang.ineq.bij} there exists a unique measured foliation class \([F_{\Delta_f}]\) on \(\Delta_f\), totally transverse to \(\Delta_f\), such that
\[
i([F_{\Delta_f}],e)=a_e \qquad\text{for each edge } e\subset \Delta_f.
\]
These local foliations \(F_{\Delta_f}\) glue consistently along the edges of \(G\), producing a measured foliation class
\(
[F]\in \cMF(G^\perp).
\)
By Proposition~\ref{prop:t.a.imply.t.t.G}, we have \([F]\in\cMFp\), and by construction
\(\iG_{\cdot}([F])=(a_e)_{e\in E(G)}\).
Thus \(\iG_{\cdot}\) is surjective onto \(\Omega_G\).

\emph{(3) Injectivity.}
Suppose \([F_1],[F_2]\in\cMFp\) satisfy
\(\iG_{\cdot}([F_1])=\iG_{\cdot}([F_2])\), i.e.
\[
i([F_1],e)=i([F_2],e)\ \ \text{for all } e\in E(G).
\]

By Lemma~\ref{lem:int.comp.carried by.dual.tt.}, both classes are carried by the dual train track \(\tau_G\). Using the switch conditions on $\tau_G$ restricted to each face of $G$ (as discussed in the proof of Lemma \ref{lem:MF.triang.ineq.bij}), one can show that \([F_1]\) and \([F_2]\) induce the same branch weights on $\tau_G$. Combined with Lemma~\ref{lem:wtt.ml}, \([F_1]=[F_2]\).
Hence \(\iG_{\cdot}\) is injective.

Combining (1)--(3), \(\iG_{\cdot}:\cMFp\to\Omega_G\) is a bijection.

\end{proof}

\subsection{Shear-radius coordinates of \texorpdfstring{\(\cMFp\)}{MF(S,\underline p)}}
\label{subsec:MF.shear.radius}

While the edge intersection map \(\iG_{\cdot}\) (Proposition~\ref{prop:edge.length.MF.biject.}) provides global coordinates on \(\cMFp\), it mixes the contribution of the interior component and of the peripheral annuli.
For our purposes---in particular for defining deformations that change the cone angles---it is more convenient to split the data into \emph{shear} parameters along the edges and \emph{radius} parameters at the marked points.

In this subsection we introduce the \emph{shear--radius map}
\[
\mathbf{sr}_G = (\mathbf{s}_G,\mathbf{r}_G): \cMFp \longrightarrow \bR^{E(G)} \times \bR_{>0}^n
\]
associated to a triangulation \(G\in\TSp\).
We first put measured foliations in a normal form \(F_G\) totally transverse to \(G\) and define the shear parameters \(s_e([F])\) along each edge \(e\).
We then define the radii \(r_i([F])\) at the marked points and the vertex-balanced cone \(\Lambda_G\subset\bR^{E(G)}\times\bR_{>0}^n\).
Finally we construct an explicit map \(\mathcal{L}^{(G)}_{sr}:\Lambda_G\to\Omega_G\) and use the edge intersection parametrization to prove that \(\mathbf{sr}_G\) is a bijection (Proposition~\ref{prop:bijection.srG}), with continuity deferred to Subsection~\ref{subsec:invers.contin.}.
This will be the combinatorial model for the circular foliation coordinates and for the partial stretch/anti-stretch deformations in Section~\ref{subsec:application}.

\subsubsection{Normal forms and singular feet}

The shear parameters along the edges of \(G\) are defined by comparing canonical
foot points on each edge coming from the two adjacent faces.  This is analogous to
Thurston's shear (shift) coordinates defined using the horocyclic foliation on ideal
triangles, where the non-foliated core is collapsed onto a tripod and each edge carries
a distinguished point. 

\begin{definition}[Tripod normal form on a triangle]
\label{def:normal.form.triangle}
Let \(\Delta\) be a marked triangle and let \([F_\Delta]\in \cMF(\Delta^\perp)\).
A representative \(F_\Delta\) is in \DEF{\em (tripod) normal form} if its singular locus in \(\Delta\) is an embedded tripod with a unique \(3\)-prong singularity in the interior and three singular leaves meeting the three edges transversely at three points.  
These points are called the \DEF{\em singular feet} of \(F_\Delta\) on \(\partial\Delta\). 
(e.g. Figure~\ref{subfig1:MF.tria.ineq.})
\end{definition}

\begin{definition}[\(G\)-normal form]
\label{def:G.normal.form}
Let \(G\) be a triangulation of \((S,\underline p)\) and let \([F]\in \cMFp\).
A representative \(F\) is in \DEF{\em \(G\)-normal form} if for every face \(f\in F(G)\), the restriction \(F|_f\) is in tripod normal form.
\end{definition}

\begin{lemma}[Existence of the \(G\)-normal form]
\label{lem:G-normal-form-existence}
Fix a triangulation \(G\) of \((S,\underline p)\).
Any $[F] \in \cMFp$ admits a $G$-normal form $F_G$. 
\end{lemma}

\begin{proof}
Let \(G\) be fixed and let \(\tau_G\) be the dual train track.
By Lemma~\ref{lem: suff.card.tt.G}, the class \([F]\in\cMFp\) determines a branch-weight system
\(w=w([F])\) on \(\tau_G\) whose weights are positive and satisfy the switch conditions.
Equivalently, for each face \(f\in F(G)\), the induced triple of edge-intersection numbers
\(i^{(G)}_{[F|_f]}\) lies in the open cone \(\Omega_\Delta\), hence satisfies the strict triangle inequalities.

Realize \((\tau_G,w)\) by a measured foliation on a foliated neighborhood \(N(\tau_G)\)
whose leaves run parallel to the branches and whose transverse measures are given by \(w\).
Since \(\tau_G\) is maximal, each complementary component of \(S\setminus N(\tau_G)\)
is either a topological triangle (a trigon) or a disk with a marked point in its interior. Collapse each complementary triangle to a tripod.
The resulting measured foliation on \((S,\underline p)\) has, in each face \(f\),
a unique interior \(3\)-prong singularity coming from the collapsed trigon; denote this foliation by \(F_G\).
By construction, \((F_G)|_f\) is a tripod normal form for every face \(f\), i.e.\ \(F_G\) is in \(G\)-normal form.

Moreover, \(F_G\) is carried by \(\tau_G\) and induces exactly the same branch weights \(w\).
Therefore, by Lemma~\ref{lem:wtt.ml} (injectivity of branch weights on carried classes),
we have \([F_G]=[F]\). Hence \(F_G\) is a \(G\)-normal form representative of \([F]\).
\end{proof}

\begin{definition}[Singular feet and dual arc]
    \label{def:feet.dual.arc}
    Let \(F_G\) be the $G$-normal form of \([F]\in\cMFp\) for a triangulation \(G\in\TSp\).
    For each face \(f\in F(G)\), let \(q_f\) be the unique interior singularity of \(F_G\) in \(f\); the singular locus in \(f\) is a tripod with center \(q_f\).

    For each edge \(e\in E(G)\), choose an orientation of \(e\) and denote by \(f^{(\mathrm{l})}_e\) and \(f^{(\mathrm{r})}_e\) the faces to the left and to the right of \(e\) (they may coincide, see Figure~\ref{subfig:shear_same_face}).
    We define:
    \begin{itemize}
        \item the \DEF{\em left} and \DEF{\em right singular feet} \(c^{(\mathrm{l})}_e, c^{(\mathrm{r})}_e \in e\) as the intersection points of \(e\) with the singular leaves issuing from \(q_{f^{(\mathrm{l})}_e}\) and \(q_{f^{(\mathrm{r})}_e}\), respectively; 
        \item the \DEF{\em dual arc} \(e^*\) of \(e\) as the unique (up to isotopy) arc contained in \(f^{(\mathrm{l})}_e \cup f^{(\mathrm{r})}_e\) which joins \(q_{f^{(\mathrm{l})}_e}\) to \(q_{f^{(\mathrm{r})}_e}\) and crosses \(e\) exactly once (and meets no other edge of \(G\)).
    \end{itemize}
\end{definition}

This is the usual dual picture for shear coordinates: the dual arc \(e^*\) records the relative position of the singular tripod in the two adjacent faces.

\subsubsection{Shear along edges and vertex balance}

We now define the shear of a measured foliation along an edge in terms of the transverse measure on the dual arc.
The definition uses a coordinate function on \(e\), but the resulting shear is independent of this choice and the orientation of \(e\).

\begin{definition}[Shear along an edge]
\label{def:MF.shear}
\label{def:MF.shear.map}
Let \(G\in\TSp\), \([F]\in\cMFp\) and \(F_G\) be the \(G\)-normal form of \([F]\).
For an edge \(e\in E(G)\), fix an orientation of \(e\) and a base point \(b_e\in e\).
We define a coordinate map
\[
\iota_e : e \longrightarrow \bR
\]
as follows: for any \(p\in e\), the value \(\iota_e(p)\) is the \DEF{\em signed distance} from \(b_e\) to \(p\), where the distance is the intersection number \(i(F_G,[b_e,p])\) (with \([b_e,p]\) the homotopy class of the segment in \(e\) joining \(b_e\) to \(p\)), taken with positive sign if the segment is oriented in the same direction as \(e\) and negative otherwise.

The \DEF{\em shear of \(F_G\) along \(e\)} is
\begin{equation}
    \label{eq:shear}
    s_e(F_G) := \iota_e\bigl(c^{(\mathrm{l})}_e\bigr) - \iota_e\bigl(c^{(\mathrm{r})}_e\bigr).
\end{equation}
This quantity is independent of the choice of base point \(b_e\) and of the choice of orientation of \(e\); its absolute value coincides with the intersection number \(i(F_G,e^*)\).
The \DEF{\em shear of \([F]\) along \(e\)} is defined by
\[
s_e([F]) := s_e(F_G).
\]
\end{definition}

Note that \(s_e([F]) = 0\) if and only if the dual arc \(e^*\) is isotopic to a saddle connection of $F_G$ joining the singularities \(q_{f^{(\mathrm{l})}_e}\) and \(q_{f^{(\mathrm{r})}_e}\) (Figure~\ref{subfig:shear_same_face}).

\begin{figure}[htp]
\begin{subfigure}[b]{0.48\textwidth}
  \centering
  \includegraphics[width=0.6\linewidth]{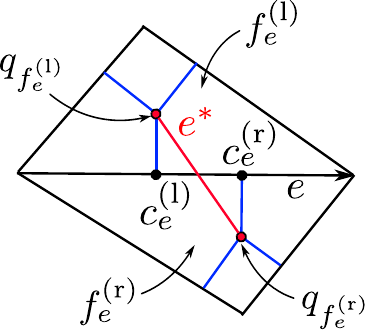}
  \caption{\small\(f^{(\mathrm{l})}_e \neq f^{(\mathrm{r})}_e\), \(s_e<0\).}
  \label{subfig:def_shear_diff_face}
\end{subfigure}
\hspace{0.3cm}
\begin{subfigure}[b]{0.48\textwidth}
  \centering
  \includegraphics[width=0.69\linewidth]{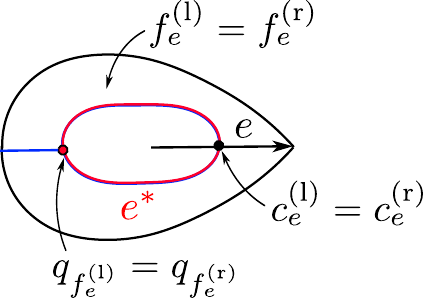}
  \caption{\small\(f^{(\mathrm{l})}_e=f^{(\mathrm{r})}_e\), \(s_e=0\).}
  \label{subfig:shear_same_face}
\end{subfigure}
\caption{Examples of shears along an edge \(e\).}
\label{fig:def_shear}
\end{figure}

\begin{definition}[Shear map]
    \label{def:MF.shear.map.global}
    For each triangulation \(G\in\TSp\), the \DEF{\em shear map} is the function
    \[
    \mathbf{s}_G : \cMFp \longrightarrow \bR^{E(G)}, \qquad
    [F] \longmapsto \bigl(s_e([F])\bigr)_{e\in E(G)}.
    \]
    Upon choosing an ordering \(E(G)=\{e_1,\dots,e_{6g-6+3n}\}\), we may view \(\mathbf{s}_G\) as a map into \(\bR^{6g-6+3n}\).
\end{definition}

\begin{remark}[Comparison with classical shear coordinates]
\label{rem:MF.shear.classical}
Our sign convention for \(s_e([F])\) is chosen so that, the collection \(\bigl(s_e([F])\bigr)_{e\in E(G)}\) agrees with Bonahon's shear cocycle \cite{Bon96} and with Thurston's shear coordinates for hyperbolic structures \cite{Thu98}.
Related shear-type coordinate systems have also been developed from different viewpoints: Pan--Wolf introduce shear--type coordinates associated to a fixed geodesic lamination via harmonic maps \cite{PW22}, and Calderon--Farre define shear--shape cocycles that extend the Bonahon--Thurston shear formalism to arbitrary measured laminations \cite{CF24}.
\end{remark}

In particular, it satisfies the following linear balance condition at each vertex.

\begin{lemma}[Vertex balance condition]
    \label{lem:balanc.cond.}
    For any triangulation \(G \in \TSp\) and any \([F] \in \cMFp\), the sum of shears around any marked point \(p \in \up\) vanishes:
    \begin{equation}\label{eq:shear-cond}
    \sum_{e \in E_p(G)} s_{e}([F]) = 0,
    \end{equation}
    where \(E_p(G)\) is the star at \(p\) in \(G\) (see Definition~\ref{def:face.star}).
\end{lemma}

\begin{proof}
Let \(F_G\) be the \(G\)-normal form of \([F]\).
Fix a vertex \(p\in\up\) and enumerate the edges incident to \(p\) as \(e_1,\dots,e_m\) in counter-clockwise order around \(p\).
For each \(j\), let \(f_j\) be the face between \(e_j\) and \(e_{j+1}\) (indices taken modulo \(m\)).
Orient each \(e_j\) outward from \(p\).

For each \(j\), let \(c_j^{(\mathrm{l})}\) and \(c_j^{(\mathrm{r})}\) be the left and right singular feet of \(F_G\) on \(e_j\), and choose a simple closed leaf \(C\) of the peripheral component around \(p\).
Let \(b_j:=C\cap e_j\) and let \(\iota_{e_j}\) be the coordinate map on \(e_j\) with base point \(b_j\) as in Definition~\ref{def:MF.shear}.
Then
\begin{equation}
    \label{eq:balanc.cond.shear}
    s_{e_j}([F]) = s_{e_j}(F_G) = \iota_{e_j}\bigl(c_j^{(\mathrm{l})}\bigr) - \iota_{e_j}\bigl(c_j^{(\mathrm{r})}\bigr).
\end{equation}

The segments \([c_j^{(\mathrm{l})},b_j]\) and \([c_{j+1}^{(\mathrm{r})},b_{j+1}]\) are homotopic through leaves of \(F_G\) contained in the union of the two adjacent faces (see Figure~\ref{fig:shear}), hence they have the same transverse measure.
This implies
\begin{equation}
    \label{eq:balanc.cond.l2r}
    \iota_{e_{j}} \bigl(c_j^{(\mathrm{l})}\bigr) - \iota_{e_{j}} (b_j)
    = \iota_{e_{j+1}}\bigl(c^{(\mathrm{r})}_{j+1}\bigr) - \iota_{e_{j+1}} (b_{j+1}).
\end{equation}

Summing \eqref{eq:balanc.cond.shear} over \(j=1,\dots,m\) and inserting \eqref{eq:balanc.cond.l2r} we obtain
\begin{align*}
    \sum_{j=1}^{m} s_{e_j}([F])
    &= \sum_{j=1}^{m} \bigl(\iota_{e_j}(c_j^{(\mathrm{l})}) - \iota_{e_j}(c_j^{(\mathrm{r})})\bigr) \\
    &= \sum_{j=1}^{m} \bigl(\iota_{e_j}(c_j^{(\mathrm{l})}) - \iota_{e_j}(b_j)\bigr)
       - \sum_{j=1}^{m} \bigl(\iota_{e_j}(c_j^{(\mathrm{r})}) - \iota_{e_j}(b_j)\bigr) \\
    &= \sum_{j=1}^{m} \bigl(\iota_{e_{j+1}}(c_{j+1}^{(\mathrm{r})}) - \iota_{e_{j+1}}(b_{j+1})\bigr)
       - \sum_{j=1}^{m} \bigl(\iota_{e_j}(c_j^{(\mathrm{r})}) - \iota_{e_j}(b_j)\bigr) \\
    &= 0,
\end{align*}
as required.
\end{proof}

\begin{figure}[htp]
    \centering
    \includegraphics[width=8cm]{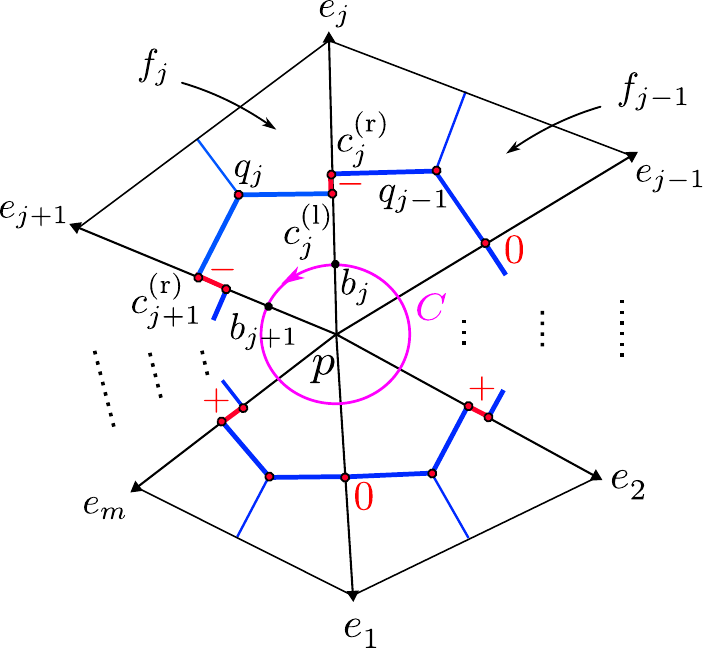}
    \caption{Shears along the edges incident to a vertex \(p\), with signs determined by the relative positions of the singular feet.}
    \label{fig:shear}
\end{figure}

 \begin{remark}
The vertex balance condition \eqref{eq:shear-cond} is the measured--foliation analogue of the cusp completeness relation in classical shear coordinates (see \cite[Sec.~3]{PT07}).
It also matches the cocycle (closedness) condition in Bonahon's shear theory \cite[Sec.~12]{Bon96} and the vanishing of the metric residue in the shear--shape formalism \cite[Lem.~7.9]{CF24}.
\end{remark}

\subsubsection{Radius parameters and the shear--radius map}

\begin{definition}[Radius and radius map]
    \label{def:MF.radius}
    Let \(G\in\TSp\) and \([F]\in\cMFp\), and let \(F_G\) be the \(G\)-normal form of \([F]\).
    For each vertex \(p_i\in\up\) and each face \(f\) adjacent to \(p_i\), let \([p_i,q_f]\) denote the homotopy class of a simple arc joining \(p_i\) to the singularity \(q_f\) of \(F_G\) in \(f\).
    The \DEF{\em radius} of \([F]\) at \(p_i\) is
    \[
    r_i([F]) := \min_{f\in F_{p_i}(G)} i\bigl(F_G,[p_i,q_f]\bigr),
    \]
    where \(F_{p_i}(G)\) is the set of faces of $G$ adjacent to \(p_i\).
    The \DEF{\em radius map} associated to \(G\) is
    \[
    \mathbf{r}_G : \cMFp \longrightarrow \bR_{>0}^n, \qquad
    [F] \longmapsto \bigl(r_j([F])\bigr)_{j=1}^n.
    \]
\end{definition}

\begin{remark}[Radius and peripheral components]
\label{rem:MF.radius.positive}
By Remark~\ref{rem:MFp.convention} we restrict to classes \([F]\in\cMFp\) for which each \(r_i([F])\) is strictly positive, ruling out the degenerate case where the peripheral annulus about \(p_i\) has zero transverse width. 
For such \([F]\), the value \(r_i([F])\) is intrinsic (it depends only on \([F]\) and \(p_i\), not on the choice of triangulation \(G\)) and agrees with the radius of the peripheral component \([F^{(i)}]\) around \(p_i\) in the decomposition of Definition~\ref{def:MF.surf.trivial.}, i.e.\ the transverse measure of any arc crossing that annulus once. 
All measured foliations arising later in the paper—in particular those coming from circular foliations and from the partial (anti)-stretch deformations—satisfy this positivity condition.
\end{remark}

\begin{definition}[Shear--radius map]
\label{def:MF-shear-radius}
For each triangulation \(G\in\TSp\), the \DEF{\em shear--radius map} is
\[
\mathbf{sr}_G : \cMFp \longrightarrow \bR^{E(G)} \times \bR_{>0}^n, \qquad
[F] \longmapsto \bigl(\mathbf{s}_G([F]),\mathbf{r}_G([F])\bigr).
\]
\end{definition}

\begin{remark}[Shear--radius coordinates]
\label{rem:MF.sr.coordinates}
The map \(\mathbf{sr}_G = (\mathbf{s}_G,\mathbf{r}_G)\) provides a shear-type parametrization of \(\cMFp\) adapted to our circular foliation picture.
It is formally analogous to the classical shear coordinates for hyperbolic metrics on a punctured surface with respect to an ideal triangulation, where the shears describe the interior geometry and the horocyclic radii control the behaviour at the punctures; see \cite{PP93,PP05,Bon96,CF24}.
Here the roles are played by the interior component and the peripheral radii of the measured foliation.
\end{remark}

\begin{definition}[Vertex-balanced cone]
    \label{def:vertex.balanc.part}
    Let \(\Lambda_G\) denote the subset of \(\bR^{E(G)} \times \bR^n_{>0}\) consisting of pairs \((\mathbf{s},\mathbf{r})\) whose first component \(\mathbf{s}\) satisfies the vertex balance condition \eqref{eq:shear-cond} at every marked point:
    \[
    \Lambda_G := \Bigl\{ (\mathbf{s},\mathbf{r}) \in \bR^{E(G)} \times \bR^n_{>0} \;\Big|\;
    \sum_{e\in E_p(G)} s_e = 0 \text{ for all } p\in\up \Bigr\}.
    \]
    We call \(\Lambda_G\) the \DEF{\em vertex-balanced cone}.
\end{definition}

\subsubsection{Recovering edge intersections from shear–radius data}

We now construct a map \(\mathcal{L}^{(G)}_{sr}:\Lambda_G\to\Omega_G\) which recovers the edge intersection numbers from the shear--radius data.
Combined with the bijectivity of the edge intersection map \(\iG_{\cdot}\) (Proposition~\ref{prop:edge.length.MF.biject.}), this will yield the bijectivity of \(\mathbf{sr}_G\).

\begin{construction}[Construction of \(\mathcal{L}^{(G)}_{sr}\)]
\label{const:LG.map}
Let \((\mathbf{s},\mathbf{r})\in\Lambda_G\).
We construct \(\mathcal{L}^{(G)}_{sr}(\mathbf{s},\mathbf{r})\in\Omega_G\) in four steps.

\medskip\noindent
\emph{Step 1: The set \(E_v^*\).}
For each vertex \(v\in\up\), let \(E_v(G)\) be the star at \(v\) in $G$ and write its edges in counter-clockwise order as \((e_1,\dots,e_m)\), with indices taken modulo \(m\).
For \(k\in \mathbb{N}^+\), set
\[
S_v(k) := \sum_{0<j\le k} \mathbf{s}(e_j).
\]
The vertex balance condition \(\sum_{j=1}^m \mathbf{s}(e_j)=0\) implies that the minimum of \(\{S_v(k)\}_{k\in\mathbb{N}^+}\) is attained.
We define the non-empty subset
\[
E_v^* := \bigl\{ e_k\in E_v(G) \;\big|\; S_v(k) = \min_{1\le l\le m} S_v(l) \bigr\}.
\]
It is straightforward to check that \(E_v^*\) is independent of the choice of the starting edge in the cyclic ordering of \(E_v(G)\).

\medskip\noindent
\emph{Step 2: Corner weights \(a_f(v)\).}
For each \(v\in\up\) and each face \(f\) adjacent to \(v\), we assign a \emph{corner weight} \(a_f(v)\) as follows.
If \(f=f_k\) is the face bounded by \(e_k\) and \(e_{k+1}\), choose any \(e_{k_0}\in E_v^*\) and set
\begin{equation}
    \label{eq:def.a_f}
    a_{f_k}(v) := \mathbf{r}(v) + \sum_{k_0<j\le k} \mathbf{s}(e_j),
\end{equation}
where \(\mathbf{r}(v):=\mathbf{r}_i\) if \(v=p_i\), and the sum is taken along the cyclic order of edges.
By construction, the difference \(a_{f_k}(v)-\mathbf{r}(v)\) is a partial sum of the sequence \(\{\mathbf{s}(e_j)\}\) shifted so that its minimum value is \(0\); in particular,
\[
a_{f_k}(v) \ge \mathbf{r}(v) > 0
\]
for all faces \(f_k\) adjacent to \(v\), and the value \(a_{f_k}(v)\) is independent of the choice of \(e_{k_0}\in E_v^*\).

\medskip\noindent
\emph{Step 3: Edge lengths \(l_e\).}
For each edge \(e\in E(G)\) with endpoints \(v\) and \(v'\), choose one of the two incident faces, say \(f\), and define
\begin{equation}
    \label{eq:def.edge.sr}
    l_e := a_f(v) + a_f(v').
\end{equation}
One checks directly from the definition of the \(a_f(\cdot)\) that this value does not depend on the choice of the incident face \(f\): if \(f\) and \(f'\) are the two faces adjacent to \(e\), then the differences \(a_{f'}(v)-a_{f}(v)\) and \(a_{f'}(v')-a_f(v')\) have the same absolute value \(|\mathbf{s}(e)|\) but with opposite signs. This implies that $l_e$ is well-defined.

\medskip\noindent
\emph{Step 4: Definition of \(\mathcal{L}^{(G)}_{sr}\).}
We define
\[
\mathcal{L}^{(G)}_{sr} : \Lambda_G \longrightarrow \bR_{\ge0}^{E(G)}, \qquad
(\mathbf{s},\mathbf{r}) \longmapsto (l_e)_{e\in E(G)}.
\]
By construction, the restriction of \((l_e)\) to the three edges of any face of \(G\) satisfies the strict triangle inequalities: this follows from the positivity of the corner weights \(a_f(v)\) and from the formula \eqref{eq:def.edge.sr}.
Hence \(\mathcal{L}^{(G)}_{sr}(\mathbf{s},\mathbf{r})\in\Omega_G\).
\end{construction}

\subsubsection{Bijectivity of the shear--radius map}

We can now relate the shear--radius data to the edge intersection coordinates.

\begin{proposition}[Bijectivity of the shear--radius map]
\label{prop:bijection.srG}
For each triangulation \(G\in\TSp\), the shear--radius map
\[
\mathbf{sr}_G : \cMFp \longrightarrow \Lambda_G
\]
is a bijection.
Moreover its inverse is the composition
\[
\mathbf{sr}_G^{-1} = \bigl(i^{(G)}_{\cdot}\bigr)^{-1} \circ \mathcal{L}^{(G)}_{sr},
\]
where \(i^{(G)}_{\cdot}:\cMFp\to\Omega_G\) is the edge intersection map and \(\mathcal{L}^{(G)}_{sr}:\Lambda_G\to\Omega_G\) is given by Construction~\ref{const:LG.map}.
\end{proposition}

\begin{proof}
We prove the two identities
\[
\mathbf{sr}_G \circ \bigl(i^{(G)}_{\cdot}\bigr)^{-1} \circ \mathcal{L}^{(G)}_{sr} = \mathrm{id}_{\Lambda_G}
\quad\text{and}\quad
\bigl(i^{(G)}_{\cdot}\bigr)^{-1} \circ \mathcal{L}^{(G)}_{sr} \circ \mathbf{sr}_G = \mathrm{id}_{\cMFp}.
\]

\emph{First identity.}
Let \((\mathbf{s},\mathbf{r})\in\Lambda_G\) and set \([F]:=\bigl(i^{(G)}_{\cdot}\bigr)^{-1}\circ\mathcal{L}^{(G)}_{sr}(\mathbf{s},\mathbf{r})\).
Let \(F_G\) be the \(G\)-normal form of \([F]\) and, for each face \(f\), let \(q_f\) be its singularity.
By construction of \(i^{(G)}_{\cdot}\) and \(\mathcal{L}^{(G)}_{sr}\), the corner weights satisfy
\begin{equation}
    \label{eq:inter.num.vs.a_f}
    i\bigl(F_G,[p,q_f]\bigr) = a_f(p)
\end{equation}
for every vertex \(p\) of \(f\).
Taking the minimum over the faces adjacent to \(p_i\) gives
\[
r_i([F])
= \min_{f\in F_{p_i}(G)} i\bigl(F_G,[p_i,q_f]\bigr)
= \min_{f\in F_{p_i}(G)} a_f(p_i)
= \mathbf{r}_i,
\]
using \eqref{eq:def.a_f} and the normalization by the minimum of the partial sums.

Similarly, for an edge \(e\) joining vertices \(v\) and \(v'\), with left and right faces \(f^{(\mathrm{l})}_e\) and \(f^{(\mathrm{r})}_e\), we may compute the shear using \(v\) as base point:
\begin{align*}
    s_e([F])
    &= i\bigl(F_G,[v,q_{f^{(\mathrm{l})}_e}]\bigr) - i\bigl(F_G,[v,q_{f^{(\mathrm{r})}_e}]\bigr)
      &&\text{by \eqref{eq:shear}} \\
    &= a_{f^{(\mathrm{l})}_e}(v) - a_{f^{(\mathrm{r})}_e}(v)
      &&\text{by \eqref{eq:inter.num.vs.a_f}} \\
    &= \mathbf{s}(e),
\end{align*}
where the last equality follows from the definition \eqref{eq:def.a_f} and the fact that the faces around \(v\) are ordered in accordance with the cyclic order of edges.
Thus \(\mathbf{sr}_G([F]) = (\mathbf{s},\mathbf{r})\), proving the first identity.

\emph{Second identity.}
We must show that
\[
\mathcal{L}^{(G)}_{sr} \circ \mathbf{sr}_G = i^{(G)}_{\cdot}.
\]
Let \([F]\in\cMFp\), and let \(F_G\) and \(q_f\) be as above.
For an edge \(e\) with endpoints \(v,v'\) and an adjacent face \(f\), we have
\begin{align*}
\bigl(\mathcal{L}^{(G)}_{sr}\circ\mathbf{sr}_G\bigr)([F])(e)
&= l_e = a_f(v) + a_f(v') &&\text{by \eqref{eq:def.edge.sr}} \\
&= i\bigl(F_G,[v,q_f]\bigr) + i\bigl(F_G,[v',q_f]\bigr)
   &&\text{by \eqref{eq:inter.num.vs.a_f}} \\
&= i(F_G,e) \\
&= i^{(G)}_{[F]}(e),
\end{align*}
since the concatenation of \([v,q_f]\) and \([q_f,v']\) is homotopic to the edge \(e\).
This proves the second identity.

The two identities imply that \(\mathbf{sr}_G\) is bijective, with inverse \(\bigl(i^{(G)}_{\cdot}\bigr)^{-1}\circ\mathcal{L}^{(G)}_{sr}\).
\end{proof}

 \section{Teichm\"uller space}
\label{sec:Teich.}

Having defined our combinatorial model space $\cMFp$ in the previous section, we now
introduce its geometric counterpart: the Teichm\"uller space of hyperbolic
cone-metrics $\cTp$.  We begin with the local models, and then pass to the global Teichm\"uller space and the triangulable loci
$\cT(G)$ that will serve as coordinate charts in Section~\ref{sec:circ.foli.coor.app.}.

\subsection{Teichm\"uller space of hyperbolic cone metrics}
\label{subsubsec:Teich.space.surf.}

This subsection recalls the basic objects on the geometric side: hyperbolic
triangles, hyperbolic cone-metrics on $(S,\up)$, and the induced Teichm\"uller
space together with the loci $\cT(G)$.

\begin{definition}[Teichm\"uller space of a triangle]
\label{def:Teich.of.hyp.triang.}
Let $\Delta$ be a topological triangle (Definition~\ref{def:top.triang.}). 
The \DEF{\em Teichm\"uller space} of $\Delta$, denoted by $\mathcal{T}(\Delta)$, is
the set of hyperbolic metrics on $\Delta$ for which all edges of $\Delta$ are
geodesic segments of finite length, taken up to isotopy fixing each vertex.
\end{definition}

\begin{definition}[Hyperbolic cone-metrics]
\label{def:tri-admis}
A \DEF{\em hyperbolic cone-metric} on $\Sp$ is a Riemannian metric which has
constant curvature $-1$ on $\dot S$ and such that each marked point $p_i$ has a
neighbourhood isometric to a neighbourhood of the singular point in the model
$\bH^2_{\theta_i}$ for some $\theta_i>0$, where $\bH^2_{\theta_i}$ denotes the
singular disk (with singular point at the center) endowed with the metric
\[
g_{\theta_i} = dr^2 + \sinh^2 r\,d\alpha^2,
\]
with polar coordinates $(r,\alpha) \in \bR_{\geq 0} \times \bR/(\theta_i\mathbb{Z})$.
We call the $\theta_i$ the \DEF{\em cone angle} at $p_i$.
\end{definition}

\begin{remark}[Angles and classical Teichm\"uller space]
\label{rmk:ang.class.Teich.space}
The cone-metrics considered in this paper always have positive cone angles at all
marked points.  Under the condition
\[
\chi(\dot S) + \sum_{i=1}^n\frac{\theta_i}{2\pi} < 0,
\]
there exists a unique hyperbolic cone-metric in each conformal class on $S$
with prescribed cone angles $\theta_i$ at $p_i$
(see \cite{Hei62,McO88,Tro91}).  In particular, for each admissible angle vector
$\underline{\theta} = (\theta_1,\dots,\theta_n)$, the Teichm\"uller space
$\cT(S,\up;\underline{\theta})$ is canonically identified with the classical
Teichm\"uller space $\cT(\dot S)$ of marked conformal structures on the punctured
surface $\dot S$.  An alternative proof of this uniformization statement, via
rank-two stable parabolic Higgs bundles (also allowing cusps), is given in
\cite{FX25}.
\end{remark}

\begin{definition}[$G$-triangulable]
Let $G\in \TSp$. A hyperbolic cone-surface $X$ (that is, the surface $\Sp$
equipped with a hyperbolic cone-metric) is said to be \DEF{\em $G$-triangulable} if
the geodesic representatives on $X$ of the edges of $G$ are smooth with mutually
disjoint interiors and decompose $X$ into non-degenerate hyperbolic geodesic
triangles (cf.~\cite[Sec.~2]{Pro25}).  
\end{definition}

\begin{remark}\label{rem:GB.and.triang}
Izmestiev~\cite[Prop.~4]{Izm15} proves the following: if \((M,g)\) is a euclidean or hyperbolic cone-surface (possibly with boundary) and \(V\subset M\) is a finite non-empty set that contains the singular locus of $g$ and contains at least one point on each boundary component of $M$, then \((M,g)\) admits a geodesic triangulation whose vertex set is exactly \(V\).
In particular, in our setting one may take \(V=\underline p\), so every hyperbolic cone-surface \(X\) is \(G\)-triangulable for some \(G\in\TSp\).
Equivalently, \(X\) can be obtained by gluing finitely many hyperbolic geodesic triangles with vertices in \(\underline p\) along their edges.
\end{remark}

\begin{definition}[Teichm\"uller space of hyperbolic cone-metrics]
\label{def:Teich.cone.surf.}
Let ${\mathcal{M}}_{-1}(S,\up)$ denote the space of hyperbolic cone-metrics on $(S,\up)$ and let ${\rm Homeo}_0(S,\up)$ denote the group of orientation-preserving homeomorphisms of $S$ isotopic to the identity and fixing each marked point in $\up$, acting on ${\mathcal{M}}_{-1}(S,\up)$ by pull-back.
\begin{itemize}
\item The \DEF{\em Teichm\"uller space of hyperbolic cone-metrics on $\Sp$} is the
quotient
\[
\cTp := {\mathcal{M}}_{-1}(S,\up)/{\rm Homeo}_0(S,\up).
\]
\item For $G\in\TSp$, the \DEF{\em $G$-triangulable locus} is the subset
\[
\cT(G) \subset \cTp
\]
consisting of $G$-triangulable cone-metrics.
\item If $\underline{G}$ is a non-empty subset of $\TSp$, the \DEF{\em
$\underline{G}$-triangulable locus} is
\[
\cT(\underline{G}) := \bigcap_{G\in\underline{G}}\cT(G).
\]
In particular, for $\underline{G}=\TSp$ we call
\[
\cT_{u.t.}\Sp := \cT(\TSp)
\]
the \DEF{\em universally triangulable locus}.
\item For $0\le a<b$, let $\cT_{(a,b)}\Sp$ denote the subspace of $\cTp$ consisting
of cone-metrics for which each cone angle lies in $(a,b)$.  In particular, we call
$\cT_{(0,\pi)}\Sp$ the \DEF{\em admissible locus} (see e.g.~\cite{DP07}).
\end{itemize}
\end{definition}

\begin{remark}[Covering $\cTp$ by triangulable loci]
\label{rem:charts_cover_space}
The $G$-triangulable loci $\cT(G)$ will be the basic coordinate domains in our
parametrization of $\cTp$.  By Izmestiev's result mentioned in
Remark~\ref{rem:GB.and.triang}, every $X\in\cTp$ is $G$-triangulable for some
$G\in\TSp$, so the collection $\{\cT(G)\}_{G\in\TSp}$ covers $\cTp$.  We shall see
in Proposition~\ref{prop:edge.leng.homeo.} that each $\cT(G)$ is open in $\cTp$,
so this is in fact an open cover.
\end{remark}

While the definition of being \emph{universally triangulable} requires, a priori, checking all combinatorial triangulations, this condition can be verified more easily in certain regimes.  
The following lemma records a simple sufficient criterion.

\begin{lemma}[A sufficient condition for being universally triangulable]
\label{lem:adm_universal_triang}
Let \(X\in\cT_{(0,\pi)}\Sp\) be a hyperbolic cone-surface whose cone angles at all
marked points are strictly less than \(\pi\).
Then for every combinatorial triangulation \(G\in\TSp\) there exists a geodesic
triangulation \(G_X\) of \(X\) which is combinatorially isomorphic to \(G\).
In particular,
\[
\cT_{(0,\pi)}\Sp\subset\cT_{u.t.}\Sp.
\]
\end{lemma}

\begin{proof}[Sketch of proof]
The argument is standard in the \(\mathrm{CAT}(-1)\) setting; we include it for
completeness.
It is well known that a hyperbolic cone-surface all of whose cone angles are at
most \(\pi\) is a locally \(\mathrm{CAT}(-1)\) geodesic metric space (see, for instance,
\cite{TWZ06} together with the general \(\mathrm{CAT}(-1)\) theory in
Bridson--Haefliger~\cite{BH99}).
In particular, every homotopy class of essential arcs and curves on \(X\)
contains a unique length-minimizing geodesic representative, and geodesic
bigons do not occur.

The local model of a cone neighbourhood of angle \(\theta<\pi\) shows that
a minimizing geodesic arc with endpoints outside a sufficiently small
cone neighbourhood cannot pass through the cone point in its interior.
A detailed discussion of this phenomenon for hyperbolic cone-surfaces can be
found, for example, in \cite[Section~2]{Nai25} (in particular Prop.~2.5,
2.6 and Lem.~2.8 there).
Applying this to the cone points of \(X\), we see that if \(\alpha\) is
an essential arc on \(X\) with endpoints at marked points and whose interior is disjoint from the other cone points, then its unique geodesic representative \(\alpha_X\) is a simple arc whose interior remains disjoint from the cone locus.

Now fix \(G\in\TSp\), and choose a topological realization of \(G\) on \(X\)
whose vertex set is \(\underline p\) and whose edges are simple
arcs with pairwise disjoint interiors.
For each edge \(e\) of \(G\), let \(e_X\) denote its geodesic representative in the corresponding rel-endpoint homotopy class.
By the absence of geodesic bigons in a \(\mathrm{CAT}(-1)\) surface,
distinct edges \(e,e'\) of \(G\) give rise to geodesics
\(e_X, e'_X\) whose interiors are still disjoint.
Hence the collection \(\{e_X\}_{e\in E(G)}\) forms a geodesic cell
decomposition of \(X\).

Finally, each two-cell of this decomposition is bounded by three simple
geodesic edges, so it is a non-degenerate geodesic triangle.
Thus we obtain a geodesic triangulation \(G_X\) combinatorially
isomorphic to \(G\), proving the lemma.
\end{proof}

\subsection{Length functions on Teichm\"uller space}
\label{subsubsec:length.func.}

We now introduce the length functions on \(\cT(S,\underline p)\) and on the triangulable loci \(\cT(G)\). These will be used repeatedly in the construction of our Teichm\"uller charts and in the comparison with measured foliations.

\begin{definition}[Arc--curve length function]
\label{def:arc.curve.leng.Teich.}
The \DEF{\em arc--curve length function} on \(\cT(S,\underline p)\) is the map
\begin{equation}\label{eq:edge.length.Teich}
\lAC_{\cdot}: \cTp \to \bR^{\cA\cup\cC}_{>0},\quad
X \mapsto \bigl(\ell_X([\beta])\bigr)_{[\beta]\in\cAC}.
\end{equation}
where \(\ell_X([\beta])\) is the infimum of the lengths (with respect to \(X\)) of all curves or arcs in the homotopy class \([\beta]\) relative to \(\underline p\).
Any representative realizing this infimum is called the \DEF{\em length–minimizing} curve or arc for \([\beta]\) with respect to \(X\), and is denoted by \(\beta_X\). 
Note that this length-minimizer might fail to be homotopic to $\beta$ when its interior contains a marked point, but we still use the name `representative' for simplicity. 
\end{definition}

It is standard that every homotopy class \([\beta]\in\cAC\) admits a length–minimizing representative, and that whenever this minimizer is (piecewise) geodesic it is unique in its homotopy class; see for instance \cite[Lem.~7.3]{Mon10}. In particular, for the cone metrics considered here we always have a well-defined length \(\ell_X([\beta])\).

\begin{definition}[Edge length function]
\label{def:edge.leng.Teich.}
Let \(G\in\TSp\).
The \DEF{\em edge length function} associated to \(G\) is the map
\[
\ell^{(G)}_{\cdot} : \cT(S,\underline p) \to \bR^{E(G)}_{>0},\quad
X \mapsto \bigl(\ell_X(e)\bigr)_{e\in E(G)},
\]
obtained by composing the arc--curve length function \(\ell^{(\cAC)}_{\cdot}\) with the restriction map
\(\mathrm{res} : \bR^{\cAC}_{>0}\to\bR^{E(G)}_{>0}\).
Here \(\ell_X(e)\) denotes the length of the length–minimizing representative of the edge class \(e\) on \(X\).
When we view \(\ell^{(G)}_{\cdot}\) as a coordinate map, we will restrict it to the triangulable locus \(\cT(G)\subset\cT(S,\underline p)\).
\end{definition}

It might be well known that $\cT(G)$ admits a parametrization by edge lengths.
We record the statement and a self-contained proof for later use.

\begin{lemma}[Edge length coordinates on \(\cT(G)\)]
\label{lem:edge.length.Teich.biject.}
Let \(G\in\TSp\). 
Then the edge length map
\[
\lG_{\cdot}:\cT(G)\longrightarrow\Omega_G\subset\bR^{E(G)}_{>0},\qquad
X\longmapsto\big(\ell_X(e)\big)_{e\in E(G)},
\]
is a bijection.
\end{lemma}

\begin{remark}
Lemma~\ref{lem:edge.length.Teich.biject.} is a special case of the general
theory of hyperbolic polyhedral surfaces.
Given a triangulated surface \((S,G)\) and an edge length function
\(\ell:E(G)\to\bR_{>0}\) such that on each face the three edge lengths satisfy
the strict triangle inequalities, one can glue hyperbolic triangles with
these side lengths along isometric edges to obtain a hyperbolic polyhedral
metric on \(S\).
This construction is standard; see for instance Luo's work on the rigidity and
deformation theory of polyhedral surfaces~\cite{Luo14,Luo11}.
We include a short proof below for completeness and to keep the notation
self-contained.
\end{remark}

\begin{proof}
The map is well defined because the edge lengths of any hyperbolic geodesic
triangle satisfy the strict triangle inequalities, hence the image of
\(\lG_{\cdot}\) is contained in \(\Omega_G\).

It remains to prove surjectivity and injectivity.

\begin{enumerate}
\item \emph{Surjectivity.}
Given \(l\in\Omega_G\), for each face \(f\in F(G)\) the three values of \(l\) on
its edges satisfy the strict triangle inequalities.
Thus there exists a hyperbolic geodesic triangle with these side lengths.
Gluing these triangles along isometric edges according to the combinatorics of
\(G\) produces a hyperbolic cone-surface \(X\) together with a geodesic
triangulation isomorphic to \(G\), and by construction \(\lG_X=l\).

\item \emph{Injectivity.}
For a given triple of edge lengths satisfying the strict triangle inequalities,
the corresponding hyperbolic triangle is unique up to isometry.
Since the gluing pattern is prescribed by \(G\), the resulting cone-metric on
\(\Sp\) is uniquely determined by its edge lengths.
Hence two points of \(\cT(G)\) with the same edge length data must coincide.
\end{enumerate}
\end{proof}

By Lemma~\ref{lem:edge.length.Teich.biject.}, it is natural to define topologies on \(\cTp\)
via length data, by pulling back the product/weak topology on appropriate function spaces
(Definition~\ref{def:weak.top.func.space}).

\begin{definition}[Length-spectrum topologies on Teichm\"uller space]
\label{def:topology.on.T}
The \DEF{\em arc--curve length-spectrum topology} on \(\cTp\) is defined as the pullback of the product topology on
\(\bR^{\cAC}_{>0}\) by the arc--curve length function \(\lAC_{\cdot}: \cTp\longrightarrow \bR^{\cAC}_{>0}\).
The subspaces \(\cT(G)\), \(\cT_{(a,b)}\Sp\), and \(\cT_{u.t.}\Sp\) are equipped with the corresponding subspace topology.

The \DEF{\em simple-curve length-spectrum topology} on \(\cTp\) is defined as the pullback of the weak topology on
\(\bR^{\cS}_{>0}\) via the length-spectrum map
\[
  \ell^{(\cS)}_{\cdot}:\cTp\longrightarrow \bR^{\cS}_{>0},\qquad
  X\longmapsto \bigl(\ell_\gamma(X)\bigr)_{\gamma\in\cS}.
\]
\end{definition}

\smallskip
\noindent\textbf{Convention.}
Unless explicitly stated otherwise, we equip \(\cTp\) with the \emph{arc--curve length-spectrum topology}.
To study the convergence of peripheral stretch rays, we instead use the \emph{simple-curve length-spectrum topology}, since peripheral stretch rays diverge in the arc--curve length-spectrum topology (Remark~\ref{rem:topology.convergence.light}).
In particular, the arc--curve length-spectrum topology is finer than the simple-curve length-spectrum topology.

\begin{remark}[Length-spectrum viewpoint]
Our choice of topology on \(\cTp\), defined via the arc--curve length map
\(\lAC_{\cdot}\), follows the standard length-spectrum approach for bordered or
marked hyperbolic surfaces (see for instance \cite{LPST10}) and is compatible
with the classical weak-topology viewpoint on spaces of measured foliations and
laminations \cite{FLP12}.  
On the metric side, Liu--Papadopoulos--Su--Th\'eret
introduced and studied length-spectrum and arc-length-spectrum metrics on
Teichm\"uller spaces of surfaces with geodesic boundary \cite{LPST10},
showing in particular that these metrics are almost isometric to the
Teichm\"uller metric on appropriate thick parts; related comparisons between
the Lipschitz metric and the Teichm\"uller metric in the classical punctured
case were obtained by Choi--Rafi \cite{CR07}.  
Although we will not use these metric comparison results directly, this
viewpoint provides background and motivation for treating arc and curve length
functions as global coordinate data on~\(\cTp\).
\end{remark}

 \section{Circular foliation coordinates}
\label{sec:circ.foli.coor.app.}

In this section we connect the combinatorial model space $\cMFp$ from
Section~\ref{sec:MF.} with the Teichm\"uller space $\cTp$ of hyperbolic
cone-metrics.  Fix a triangulation $G\in\TSp$.  On the foliation side we
have the edge intersection map (Definition~\ref{def:MF.edge.inter.})
\[
\iG_{\cdot} : \cMFp \longrightarrow \Omega_G,
\]
while on the geometric
side we have the edge length map (Definition~\ref{def:edge.leng.Teich.})
\[
\lG_{\cdot} : \cT(G) \longrightarrow \bR_{>0}^{E(G)}.
\] 
Both maps are bijections (Proposition~\ref{prop:edge.length.MF.biject.} and Lemma \ref{lem:edge.length.Teich.biject.}). Furthermore, they will be shown to be homeomorphisms onto the same admissible cone $\Omega_G$
(Proposition~\ref{prop:edge.leng.homeo.}).  This allows us to define the
\emph{circular foliation map}
\[
\cC_G
:= \big(\iG_{\cdot}\big)^{-1}\circ \lG_{\cdot}
: \cT(G) \longrightarrow \cMFp,
\]
and then to extract shear--radius coordinates for $\cT(G)$ by composing
with the parametrization $\mathbf{sr}_G$ from
Theorem~\ref{thm:homeo.MFp.shear.radius}.

The main technical task of this section is to show that edge data control the geometry of all arcs inside each face of the triangulation.  For measured
foliations, edge intersection numbers determine the
intersection with any interior arc (Lemma \ref{clm:arc.inter.triang.}); for hyperbolic metrics, edge lengths determine the length of any interior geodesic arc (Lemma \ref{clm:express.length}).  Once this
\emph{local rigidity on triangles} is established, we can promote the edge
parametrizations $\iG_{\cdot}$ and $\lG_{\cdot}$ to global coordinate
charts and construct the circular foliation map.

We briefly outline the structure of the section.
\begin{itemize}
  \item In Section~\ref{subsec:arc.triang.} we set up the notation for
  marked arcs in a topological triangle and for boundary markings induced
  by measures.
  \item In Section~\ref{subsec:MF.triang.} we analyse measured foliations on
  a triangle and show that edge intersection data control intersection
  numbers with arbitrary interior arcs
  (Proposition~\ref{lem:edge.control.arc.MF.triang.}).
  \item In Section~\ref{subsec:hyp.triang.} we prove the analogous statement
  for hyperbolic triangles: edge lengths control the lengths of interior
  geodesic arcs (Proposition~\ref{prop:edge.control.arc.hyp.triang.}).
  \item In Section~\ref{subsec:invers.contin.} we use these local results to
  establish the continuity of the inverse maps
  $\Omega_G\to\cMFp$ and $\Omega_G\to\cT(G)$
  (Proposition~\ref{prop:MF.Teich.invers.continuous}), and hence prove that
  $\iG_{\cdot}$ and $\lG_{\cdot}$ are homeomorphisms
  (Proposition~\ref{prop:edge.leng.homeo.}).
  \item Finally, in Section~\ref{subsec:circ.foli.surf.coor.} we define the
  circular foliation map $\cC_G$ and the associated shear--radius
  coordinates on $\cT(G)$, by realizing $\cC_G(X)$ as a circular foliation
  on each hyperbolic triangle (Construction~\ref{con:circ.foli.surf.}) and
  then passing to the global surface.
\end{itemize}

\subsection{Markings on triangle boundaries}
\label{subsec:arc.triang.}

We begin by setting up a convenient way to compare boundary measures on different
triangles.  Given a Radon measure on the boundary of a topological triangle
$\Delta$ and a fixed ``model'' triangle $\Delta_0$ with parametrized boundary,
we obtain a canonical identification $f_\mu:\partial\Delta_0\to\partial\Delta$
(Definition~\ref{def:edge.marking.triang.}).  This marking will be used in the
continuity estimates for arcs and foliations in the following subsections.

\begin{definition}[Radon measures on the boundary of a triangle]
\label{def:R(Delta)}
Let $\Delta$ be a topological triangle with consecutive vertices
$v_1,v_2,v_3$ in counterclockwise order and opposite edges
$e_1,e_2,e_3\subset\partial\Delta$.  We denote by $\cR(\partial\Delta)$ the set
of Radon measures $\mu$ on $\partial\Delta$ such that:
\begin{enumerate}
\item[(a)] (\emph{full support on edges}) for every nonempty open sub-arc
  $I\subset e_i$ we have $\mu(I)>0$ for $i=1,2,3$;
\item[(b)] (\emph{non-atomic on edges}) $\mu(\{x\})=0$ for every
  $x\in\partial\Delta$.
\end{enumerate}
Equivalently, for each oriented edge $e_i$ parametrized by a path
$e_i(t)$, \(t\in[0,1]\), with initial point $e_i(0)$, the cumulative function
\[
F_i(t):=\mu\big([\,e_i(0),e_i(t)\,]\big), \qquad t\in[0,1],
\]
is continuous and strictly increasing; in particular $F_i$ admits a continuous
inverse $F_i^{-1}$ on its range.
\end{definition}

\begin{definition}[Based topological triangle]
\label{def:based.triang}
Let $\Delta_0$ be a fixed topological triangle with vertex set
$\{v_1,v_2,v_3\}$ and edges $e_i$ opposite to $v_i$ for $i=1,2,3$.  Each edge
$e_i$ is parametrized by a path $e_i(t)$, \(t\in[0,1]\), such that
$e_i(0)=v_{i+1}$ and $e_i(1)=v_{i+2}$, where the indices are taken cyclically
in $\{1,2,3\}$.  We call $\Delta_0$ the \DEF{\em based triangle with
parameterized boundary}.
\end{definition}

\begin{definition}[Marking induced by a measure]
\label{def:edge.marking.triang.}
Let $\mu\in\cR(\partial\Delta)$.  A map
$f_{\mu}:\partial\Delta_0\rightarrow\partial\Delta$ is called a
\DEF{\em $\mu$-marking of $\partial\Delta$} if
\begin{itemize}
\item[(i)] $f_{\mu}$ is an orientation-preserving homeomorphism sending the
  $i$-th vertex of $\Delta_0$ to the $i$-th vertex of $\Delta$ for $i=1,2,3$;
\item[(ii)] for each edge $e_i\subset\partial\Delta_0$ and each $t\in[0,1]$ we have
\begin{equation}\label{eq:def.mark.triang.}
\frac{\mu\big(f_{\mu}([e_i(0),e_i(t)])\big)}
     {\mu\big(f_{\mu}([e_i(0),e_i(1)])\big)}
= t,
\end{equation}
where $[a,b]$ denotes the closed sub-arc of $e_i$ joining $a$ to $b$.
\end{itemize}
\end{definition}

\subsection{From edge intersections to arc intersections on triangles}
\label{subsec:MF.triang.}

The goal of this subsection is to prove Proposition~\ref{lem:edge.control.arc.MF.triang.}, which states that the intersection number of an \emph{arc on a topological triangle} (Definition~\ref{def:arc.triang.}) varies continuously with the intersection numbers of the three edges.
To organize the argument we introduce an \(i_{[F]}\)-marking of the boundary (Definition~\ref{def:i-marking}) and associated boundary \emph{coordinates} (Definition~\ref{def:coord.pt.marking}), together with a signed intersection quantity \(I_p([F])\) (Definition~\ref{def:I(p)}).
We then express the total intersection number \(i_{[F]}([\alpha])\) of an arc \(\alpha\) in terms of the data at its endpoints (Lemma~\ref{clm:arc.inter.triang.}), which yields the desired continuity.

\begin{definition}[Arcs on a triangle]
\label{def:arc.triang.}
An \DEF{\em (oriented) arc} on a topological triangle \(\Delta\) is a continuous map \(\alpha:[0,1]\to\Delta\) with \(\alpha((0,1))\) contained in the interior of $\Delta$ and \(\alpha(0),\alpha(1)\) lying on \(\partial\Delta\).
We denote by \(\cA(\Delta)\) the set of homotopy classes of arcs in \(\Delta\) relative to their endpoints, and by \(\cA_0(\Delta)\subset\cA(\Delta)\) the subset of classes of simple arcs.
\end{definition}

\begin{definition}[Marking induced by a measured foliation]
\label{def:i-marking}
Let \([F]\in\mathcal{MF}(\Delta^{\perp})\) and let \(\mu:=i_{[F]}(\cdot)\vert_{\partial\Delta}\in\cR(\partial\Delta)\) be the restriction of the intersection function to the boundary.
The \DEF{\em \(i_{[F]}\)-marking} of \(\partial\Delta\) is the \(\mu\)-marking
\[
f_{[F]}:\partial\Delta_0\longrightarrow\partial\Delta
\]
given by Definition~\ref{def:edge.marking.triang.}.
\end{definition}

\begin{definition}[Coordinates on \(\partial\Delta\)]
\label{def:coord.pt.marking}
Let \([F]\in\mathcal{MF}(\Delta^{\perp})\) and let \(p\in\partial\Delta\) lie on the edge \(e_j\) of \(\Delta\), \(j\in\{1,2,3\}\).
\begin{enumerate}
\item The \DEF{\em coordinate} of \(p\) with respect to the \(i_{[F]}\)-marking is the unique number \(t_p^{[F]}\in[0,1]\) such that
\[
f_{[F]}(e_j(t_p^{[F]})) = p,
\]
where \(e_j(t)\) is the fixed parametrization of the \(j\)-th edge of the based triangle \(\Delta_0\) (Definition~\ref{def:based.triang}).
\item If \(p\) is the \DEF{\em singular foot} of \([F]\) on \(e_j\), that is, the intersection of the singular locus of the tripod normal form of \([F]\) with \(e_j\), we denote it by \(c_j^{[F]}\) and write its coordinate as \(t_j^{[F]}:=t_{c_j^{[F]}}^{[F]}\).
\end{enumerate}
\end{definition}

The coordinate of a singular foot is given by the formula
\begin{equation}\label{eq:center}
    t^{[F]}_{j}
    = \frac{i_{[F]}(e_{j+2}) + i_{[F]}(e_j) - i_{[F]}(e_{j+1})}
           {2 \, i_{[F]}(e_j)},
\end{equation}
for $j \in \{1,2,3\}$, where the indices are taken cyclically.
In particular, $t^{[F]}_j$ depends continuously on the triple of edge
intersection numbers $\big(i_{[F]}(e_1),i_{[F]}(e_2),i_{[F]}(e_3)\big)$.

From now on we fix, for each $[F] \in \cMF(\Delta^{\perp})$, the tripod normal form $F_\Delta$ of Definition~\ref{def:normal.form.triangle} and its unique interior singularity $q_f$.

\begin{definition}[Signed intersection number]
\label{def:I(p)}
Let $[F]\in\cMF(\Delta^{\perp})$ and let $p\in\partial\Delta$.
Assume that $p$ lies on the $j$-th edge of $\Delta$.
Let $t^{[F]}_p\in [0,1]$ be the coordinate of $p$ with respect to the
$i_{[F]}$-marking of $\partial\Delta$. The \DEF{\em (signed) intersection number} of $[F]$ at $p$ is defined by
\[
    I_p([F]) := (t^{[F]}_p - t^{[F]}_{j}) \, i_{[F]}(e_j),
\]
where $t_j^{[F]}$ is the coordinate of the singular foot $c_j^{[F]}$.
\end{definition}

By construction, $I_p([F])$ is linear in the edge intersection numbers and
therefore varies continuously with these data.

\begin{lemma}[Explicit formula for the signed intersection number]
\label{clm:express.I(p)}
Let $[F]\in\cMF(\Delta^{\perp})$ and let $p$ be a point on the $j$-th edge of
$\Delta$. Then the signed intersection number $I_p([F])$ is given by
\begin{equation}
    I_p([F])
    = \frac{1}{2}\Big( (2t^{[F]}_p-1) \, i_{[F]}(e_j)
                      + i_{[F]}(e_{j+1}) - i_{[F]}(e_{j+2}) \Big),
\end{equation}
where the indices are taken cyclically in $\{1,2,3\}$.
\end{lemma}

\begin{proof}
By Definition~\ref{def:I(p)} we have
\[
    I_p([F]) = (t_p^{[F]} - t_j^{[F]}) \, i_{[F]}(e_j).
\]
Substituting the expression for $t_j^{[F]}$ from~\eqref{eq:center} yields
\begin{align*}
    I_p([F])
    &= \Big( t_p^{[F]}
       - \frac{i_{[F]}(e_{j+2}) + i_{[F]}(e_j) - i_{[F]}(e_{j+1})}
              {2 \, i_{[F]}(e_j)} \Big) \, i_{[F]}(e_j) \\
    &= \frac{1}{2}\Big( (2t_p^{[F]} - 1) \, i_{[F]}(e_j)
                      + i_{[F]}(e_{j+1}) - i_{[F]}(e_{j+2}) \Big),
\end{align*}
as required.
\end{proof}

\begin{lemma}[Explicit formula for the arc intersection number]
\label{clm:arc.inter.triang.}
Let $[F]\in\cMF(\Delta^{\perp})$.
For any $[\alpha]\in\cA_0(\Delta)$ with endpoints $p\in e_i$ and $q\in e_j$,
we have:
\begin{itemize}
\item
if $p$ and $q$ lie on the same side of the singular locus of $[F]$, then
\[
    i_{[F]}([\alpha]) = \big||I_p([F])| - |I_q([F])|\big|;
\]
\item
if $p$ and $q$ lie on different sides of the singular locus of $[F]$, then
\[
    i_{[F]}([\alpha]) = |I_p([F])| + |I_q([F])|.
\]
\end{itemize}
\end{lemma}

\begin{proof}
The formulas follow directly from the geometric meaning of the signed
intersection number $I_p([F])$.
The quantity $|I_p([F])|$ is the transverse measure of the set
$\mathcal{F}_p$ of leaves whose endpoints on the edge $e_i$ containing $p$
lie between $p$ and the singular foot $c_i^{[F]}$.

\emph{Case~1: $p$ and $q$ lie on the same side of the singular locus.}
In this case, one of the sets $\mathcal{F}_p$ and $\mathcal{F}_q$ is contained
in the other; assume $\mathcal{F}_p \subseteq \mathcal{F}_q$.
An arc $[\alpha]$ connecting $p$ and $q$ crosses exactly the leaves in the set
difference $\mathcal{F}_q \setminus \mathcal{F}_p$
(Figure~\ref{fig:MF.inter.edge.case1}),
whose transverse measure is the absolute difference of the measures of
$\mathcal{F}_p$ and $\mathcal{F}_q$.
Thus
\[
    i_{[F]}([\alpha])
    = \big|\mathrm{measure}(\mathcal{F}_p)
             - \mathrm{measure}(\mathcal{F}_q)\big|
    = \big||I_p([F])| - |I_q([F])|\big|.
\]

\emph{Case~2: $p$ and $q$ lie on different sides of the singular locus.}
Here the sets $\mathcal{F}_p$ and $\mathcal{F}_q$ are disjoint, and any arc
$[\alpha]$ connecting $p$ and $q$ crosses all leaves in both sets
(Figure~\ref{fig:MF.inter.edge.case2}).
Therefore,
\[
    i_{[F]}([\alpha])
    = \mathrm{measure}(\mathcal{F}_p) + \mathrm{measure}(\mathcal{F}_q)
    = |I_p([F])| + |I_q([F])|.
\]

\begin{figure}[htp]
\begin{subfigure}[b]{0.49\textwidth}
\centering
\includegraphics[width=0.49\linewidth]{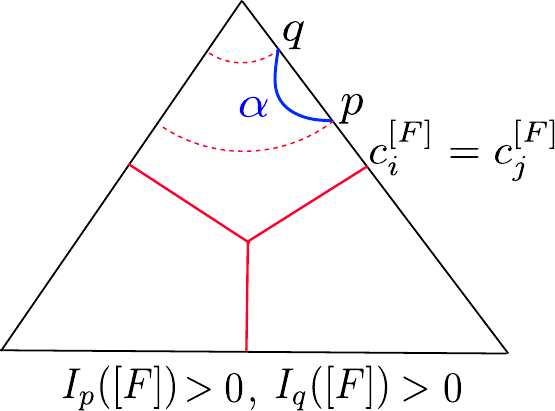}
\includegraphics[width=0.49\linewidth]{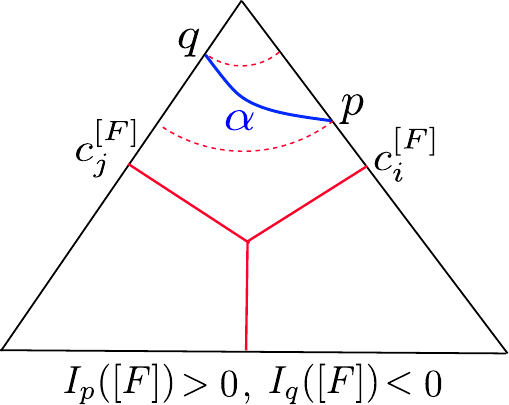}
\caption{Endpoints on the same side.}
\label{fig:MF.inter.edge.case1}
\end{subfigure}
\begin{subfigure}[b]{0.49\textwidth}
\centering
\includegraphics[width=0.495\linewidth]{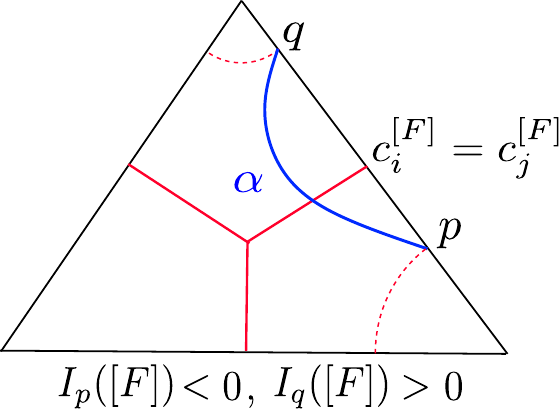}
\includegraphics[width=0.485\linewidth]{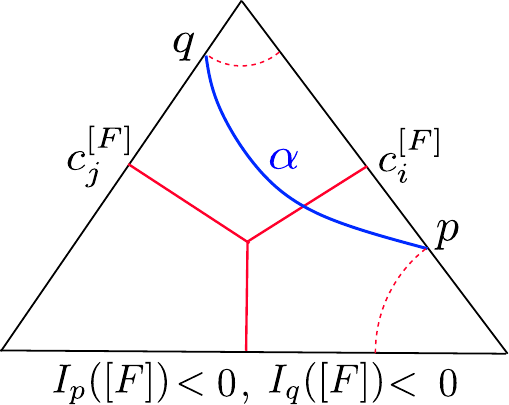}
\caption{Endpoints on different sides.}
\label{fig:MF.inter.edge.case2}
\end{subfigure}
\caption{Relative position of endpoints with respect to the singular locus
of $[F]$.}
\label{fig:MF.inter.edge.cases}
\end{figure}
\end{proof}

\begin{proposition}[Edge intersections control arc intersections]
\label{lem:edge.control.arc.MF.triang.}
Let $\alpha$ be an arc in $\Delta_0$ with distinct endpoints
$p\in e_i$ and $q\in e_j$ for $i,j\in\{1,2,3\}$.
Let $[F_1], [F_2] \in \mathcal{MF} (\Delta^{\perp})$, and let $\alpha_k$ be
a simple arc with endpoints $p_k=f_{[F_k]}(p)$ and $q_k=f_{[F_k]}(q)$
for $k=1,2$.
Then for any $\epsilon > 0$ there exists $\delta = \delta (\epsilon) > 0$
such that if
\[
    \big| i_{[F_1]} (e) - i_{[F_2]} (e) \big| < \delta
\]
for all edges $e$ of $\Delta_0$, then
\[
    \big|i_{[F_1]} (\alpha_1) - i_{[F_2]} (\alpha_2)\big| < \epsilon.
\]
\end{proposition}

\begin{proof}
By Lemmas~\ref{clm:express.I(p)} and~\ref{clm:arc.inter.triang.}, the
intersection number of an arc is determined by the edge intersection numbers
and the coordinates of its endpoints, and this dependence is given by explicit formulas.
We now show that this determines a continuous function of the edge data.

Let $t_p$ and $t_q$ be the parameter values of the endpoints $p$ and $q$ on the
boundary of the base triangle $\Delta_0$
(Definition~\ref{def:based.triang}) respectively.
By construction, these are also the coordinates of the endpoints $p_k$ and $q_k$
of each arc $\alpha_k$ with respect to the $i_{[F_k]}$-marking, for $k=1,2$.

The key observation is that the coordinates $t_i^{[F]}$ of the singular feet
on each edge $e_i$ are continuous functions of the edge intersection numbers
$\big(i_{[F]}(e_1), i_{[F]}(e_2), i_{[F]}(e_3)\big)$,
by the explicit expression~\eqref{eq:center}.

\emph{Case~1: The endpoints are not singular feet for $[F_1]$.}
If $t_p \neq t_i^{[F_1]}$ and $t_q \neq t_j^{[F_1]}$, then by continuity of
the foot coordinates we may choose $\delta>0$ so that for any $[F_2]$ with
$\big|i_{[F_1]}(e) - i_{[F_2]}(e)\big| < \delta$ for all $e$, the relative
ordering of the endpoints and the singular feet on each edge is the same for
$[F_1]$ and $[F_2]$.
In particular, the pairs $\{p_1,q_1\}$ and $\{p_2,q_2\}$ have the same
positional relationship with respect to their singular loci, so the same sign
choice in Lemma~\ref{clm:arc.inter.triang.} applies to both $[F_1]$ and $[F_2]$.
Using Lemma~\ref{clm:express.I(p)} and the fact that $t_p,t_q\in[0,1]$, we obtain
the estimate
\begin{equation}\label{eq:estimate-i}
\begin{split}
\big|i_{[F_1]}([\alpha_1])-i_{[F_2]}([\alpha_2])\big|
&= \Big|\big||I_{p_1}([F_1])|\pm|I_{q_1}([F_1])|\big|
       -\big||I_{p_2}([F_2])|\pm|I_{q_2}([F_2])|\big|\Big| \\
&\leq \big|I_{p_1}([F_1])-I_{p_2}([F_2])\big|
      +\big|I_{q_1}([F_1])-I_{q_2}([F_2])\big| \\
&\leq \sum_{l=1}^3 \big|i_{[F_1]}(e_l)-i_{[F_2]}(e_l)\big|.
\end{split}
\end{equation}

\emph{Case~2: An endpoint is a singular foot for $[F_1]$.}
If $t_p = t_i^{[F_1]}$, then $I_{p_1}([F_1])=0$.
In this situation both formulas in Lemma~\ref{clm:arc.inter.triang.} reduce to
\[
    i_{[F_1]}([\alpha_1])
    = \big||I_{p_1}([F_1])|-|I_{q_1}([F_1])|\big|
    = |I_{p_1}([F_1])|+|I_{q_1}([F_1])|.
\]
On the other hand, Lemma~\ref{clm:arc.inter.triang.} always gives
\[
    i_{[F_2]}([\alpha_2]) = \big||I_{p_2}([F_2])|\pm|I_{q_2}([F_2])|\big|
\]
for some choice of sign~$\pm$.
Thus the estimate~\eqref{eq:estimate-i} still holds.
Since the right-hand side of~\eqref{eq:estimate-i} is linear in the differences
of edge intersections, given any $\epsilon>0$ we may take
$\delta=\epsilon/3$ to conclude the proof.
\end{proof}

\subsection{From edge lengths to arc lengths on triangles}
\label{subsec:hyp.triang.}

In this subsection we show that, inside a fixed hyperbolic triangle, the length of any arc depends continuously on the edge lengths (Proposition~\ref{prop:edge.control.arc.hyp.triang.}).
Our strategy is to express the length of an arc in terms of the edge lengths and the boundary coordinates of its endpoints, using the $\ell_X$–marking (Definition~\ref{def:l-marking}).

\begin{definition}[Arc length]
\label{def:length.arc}
Let $X\in\cT(\Delta)$.
For any $[\alpha]\in\cA(\Delta)$, we define $\ell_X([\alpha])$ to be the length,
with respect to $X$, of the (unique) length-minimizing geodesic representative in the homotopy class $[\alpha]$ relative to its endpoints.
\end{definition}

\begin{definition}[$\ell_X$-marking and coordinates on $\partial \Delta$]
\label{def:l-marking}
Let $X\in\cT(\Delta)$ and let $\mu:=\ell_X(\cdot)\big|_{\partial\Delta}$ be the restriction of the length function to $\partial\Delta$.
We denote by
\[
f_X:\partial\Delta_0\longrightarrow\partial\Delta
\]
the associated $\mu$-marking (Definition~\ref{def:edge.marking.triang.}), called the \DEF{\em $\ell_X$-marking} of $\partial\Delta$.
If $p$ lies on the $i$-th edge of $\Delta$, we define its \DEF{\em coordinate} with respect to the $\ell_X$-marking to be the number $t_p^X\in[0,1]$ such that
\[
f_X(e_i(t_p^X))=p.
\]
\end{definition}

\begin{lemma}[Explicit formula for arc length]
\label{clm:express.length}
Let $X\in\cT(\Delta)$.
For any $[\alpha]\in\cA_0(\Delta)$ with endpoints $p\in e_i$ and $q\in e_j$, the following holds:
\begin{enumerate}
\item If $j=i$, then
\[
\ell_X([\alpha]) = |t_p^X - t_q^X|\cdot \ell_X(e_i).
\]
\item If $j=i+1 \,(\mathrm{mod}\,3)$, then
\[
\begin{split}
\cosh \ell_X([\alpha])
&= \cosh\big((1-t_p^X)\,\ell_X(e_i)\big)\,
   \cosh\big(t_q^X\,\ell_X(e_j)\big) \\
&\quad - \sinh\big((1-t_p^X)\,\ell_X(e_i)\big)\,
          \sinh\big(t_q^X\,\ell_X(e_j)\big)\,\zeta_{p,q}(X),
\end{split}
\]
\item If $j=i+2 \,(\mathrm{mod}\,3)$, then
\[
\begin{split}
\cosh \ell_X([\alpha])
&= \cosh\big((1-t_p^X)\,\ell_X(e_j)\big)\,
   \cosh\big(t_q^X\,\ell_X(e_i)\big) \\
&\quad - \sinh\big((1-t_p^X)\,\ell_X(e_j)\big)\,
          \sinh\big(t_q^X\,\ell_X(e_i)\big)\,\zeta_{p,q}(X),
\end{split}
\]
\end{enumerate}
where
\[
\zeta_{p,q}(X)
:= \frac{\cosh\big(\ell_X(e_j)\big)\cosh\big(\ell_X(e_i)\big)
        - \cosh\big(\ell_X(e_k)\big)}
       {\sinh\big(\ell_X(e_j)\big)\,\sinh\big(\ell_X(e_i)\big)},
\]
and $k=\{1,2,3\}\setminus\{i,j\}$.
\end{lemma}

\begin{remark}
The formulas above are a convenient explicit form of the dependence of
$\ell_X([\alpha])$ on the edge lengths and the boundary coordinates of $p$ and $q$.
We will only use them to deduce the continuity of arc lengths with respect to the
edge data in Proposition~\ref{prop:edge.control.arc.hyp.triang.}.
\end{remark}

\begin{proof}
If $j=i$, the unique geodesic representative of $[\alpha]$ is the boundary segment
between $p$ and $q$, so (1) follows directly from the definition of the $\ell_X$–marking.

Assume now that $j=i+1 \,(\mathrm{mod}\,3)$, so $e_j$ follows $e_i$ in the
counterclockwise order on $\partial\Delta$.
Consider the geodesic triangle with vertices $p$, $q$ and the vertex between
$e_i$ and $e_j$.
Writing the side lengths of this triangle in terms of $t_p^X$, $t_q^X$ and
$\ell_X(e_i)$, $\ell_X(e_j)$, $\ell_X(e_k)$ and applying the hyperbolic law of
cosines twice (see, for instance, \cite[§2.1]{Bus92}) yields the expression in~(2).
The case $j=i+2 \,(\mathrm{mod}\,3)$ is obtained from (2) by exchanging the roles of
$i$ and $j$, which gives~(3).
\end{proof}

\begin{proposition}[Edge lengths control arc lengths]
\label{prop:edge.control.arc.hyp.triang.}
Let $\alpha$ be a simple arc in $\Delta_0$ with distinct endpoints $p=e_i(t_p)$ and $q=e_j(t_q)$ for some $i,j\in\{1,2,3\}$.
Let $X_1,X_2\in\cT(\Delta)$ and let $\alpha_k$ be a simple arc in $\Delta$ with endpoints
\[
p_k:=f_{X_k}(p),\qquad q_k:=f_{X_k}(q)\qquad (k=1,2).
\]
Then for any $\epsilon>0$ there exists $\delta=\delta(\epsilon)>0$ such that if, for all $l=1,2,3$, 
\[
\big|\ell_{X_1}(e_l)-\ell_{X_2}(e_l)\big|<\delta,
\]
then
\[
\big|\ell_{X_1}([\alpha_1])-\ell_{X_2}([\alpha_2])\big|<\epsilon.
\]
\end{proposition}

\begin{proof}
By Definition~\ref{def:l-marking}, the coordinates of $p_k$ and $q_k$ with respect to the $\ell_{X_k}$-marking are independent of $k$, namely
\[
t_{p_1}^{X_1}=t_{p_2}^{X_2}=t_p,\qquad
t_{q_1}^{X_1}=t_{q_2}^{X_2}=t_q.
\]
Combining this with Lemma~\ref{clm:express.length}, we see that $\ell_{X_k}([\alpha_k])$ can be written as a function
\[
\ell_{X_k}([\alpha_k])=F\big(\ell_{X_k}(e_1),\ell_{X_k}(e_2),\ell_{X_k}(e_3)\big),
\]
where $F:(0,\infty)^3\to(0,\infty)$ is obtained from the expressions in Lemma~\ref{clm:express.length} by composing finitely many smooth functions (sums, products, $\cosh$, $\sinh$) with the fixed parameters $t_p,t_q$.
In particular, $F$ is continuous.
Hence for any $\epsilon>0$ there exists $\delta>0$ such that
\[
\big\|(\ell_{X_1}(e_1),\ell_{X_1}(e_2),\ell_{X_1}(e_3))
-(\ell_{X_2}(e_1),\ell_{X_2}(e_2),\ell_{X_2}(e_3))\big\|<\delta
\]
implies
\[
\big|\ell_{X_1}([\alpha_1])-\ell_{X_2}([\alpha_2])\big|
<\epsilon.
\]
This is exactly the desired estimate.
\end{proof}

\subsection{Continuity of the inverse maps}
\label{subsec:invers.contin.}

We now apply the local triangle estimates from Subsections~\ref{subsec:MF.triang.} and~\ref{subsec:hyp.triang.} to prove the global continuity of the inverse edge maps.
This shows that our edge-based parametrizations control the geometry of \emph{all} arcs and curves.

\begin{proposition}[Continuity of the inverse edge maps]
\label{prop:MF.Teich.invers.continuous}
For any \(G\in\TSp\), the inverse maps
\[
(\iG_{\cdot})^{-1} : \Omega_G \longrightarrow \cMFp
\quad\text{and}\quad
(\ell^G_{\cdot})^{-1} : \Omega_G \longrightarrow \cT(G)
\]
are continuous.
\end{proposition}

\begin{proof}
The proof for \((\ell^G_{\cdot})^{-1}\) is entirely analogous to that for \((\iG_{\cdot})^{-1}\); we focus on the foliation case.

Fix \([F_0]\in\cMFp\) and a homotopy class \([\beta]\in\cA\Sp\cup\cC\Sp\).
Let \(E([\beta])\subset E(G)\) be the finite set of edges of faces of \(G\) that meet a representative of \([\beta]\).
Choose a representative \(\beta^{[F_0]}\) which realizes the intersection number, i.e.
\[
i_{[F_0]}([\beta]) = i_{[F_0]}(\beta^{[F_0]}).
\]

We claim that for every \(\epsilon>0\) there exists \(\delta>0\) such that
\begin{equation}\label{eq:continuity_inverse}
\bigl|i_{[F]}([\beta]) - i_{[F_0]}([\beta])\bigr| < \epsilon
\end{equation}
whenever \(|i_{[F]}(e)-i_{[F_0]}(e)|<\delta\) for all \(e\in E([\beta])\).
This is precisely the continuity of \((\iG_{\cdot})^{-1}\) with respect to the weak topology.

\smallskip
\noindent
\emph{Step 1: Decomposing \(\beta^{[F_0]}\) into arcs.}
Decompose \(\beta^{[F_0]}\) into a finite union of sub-arcs of two types:
\begin{itemize}
\item
\(\{\beta^{[F_0]}_m\}_{m=1}^{l_1}\): maximal sub-arcs contained in the interior of faces of \(G\);
\item
\(\{\beta^{[F_0]}_{(k)}\}_{k=1}^{l_2}\): maximal sub-arcs contained in edges of \(G\).
\end{itemize}
Thus
\[
i_{[F_0]}(\beta^{[F_0]})
= \sum_{m=1}^{l_1} i_{[F_0]}(\beta^{[F_0]}_m)
  + \sum_{k=1}^{l_2} i_{[F_0]}(\beta^{[F_0]}_{(k)}).
\]

For any other foliation \([F]\in\cMFp\), we construct a comparison curve
\(\beta^{[F]}\in[\beta]\) as follows: for each endpoint of the sub-arcs above,
we transport it along the boundary of the corresponding face via the marking
\(f_{[F]}\circ f_{[F_0]}^{-1}\) (cf.\ Definition~\ref{def:i-marking}),
and join the resulting endpoints by arcs inside the same faces or edges.
This yields sub-arcs \(\beta^{[F]}_m\) and \(\beta^{[F]}_{(k)}\) with the same combinatorics (Figure \ref{fig:curv.inter.minimum}). 

\begin{figure}[htp]
\begin{subfigure}[b]{0.49\textwidth}
\centering
\includegraphics[width=0.75\linewidth]{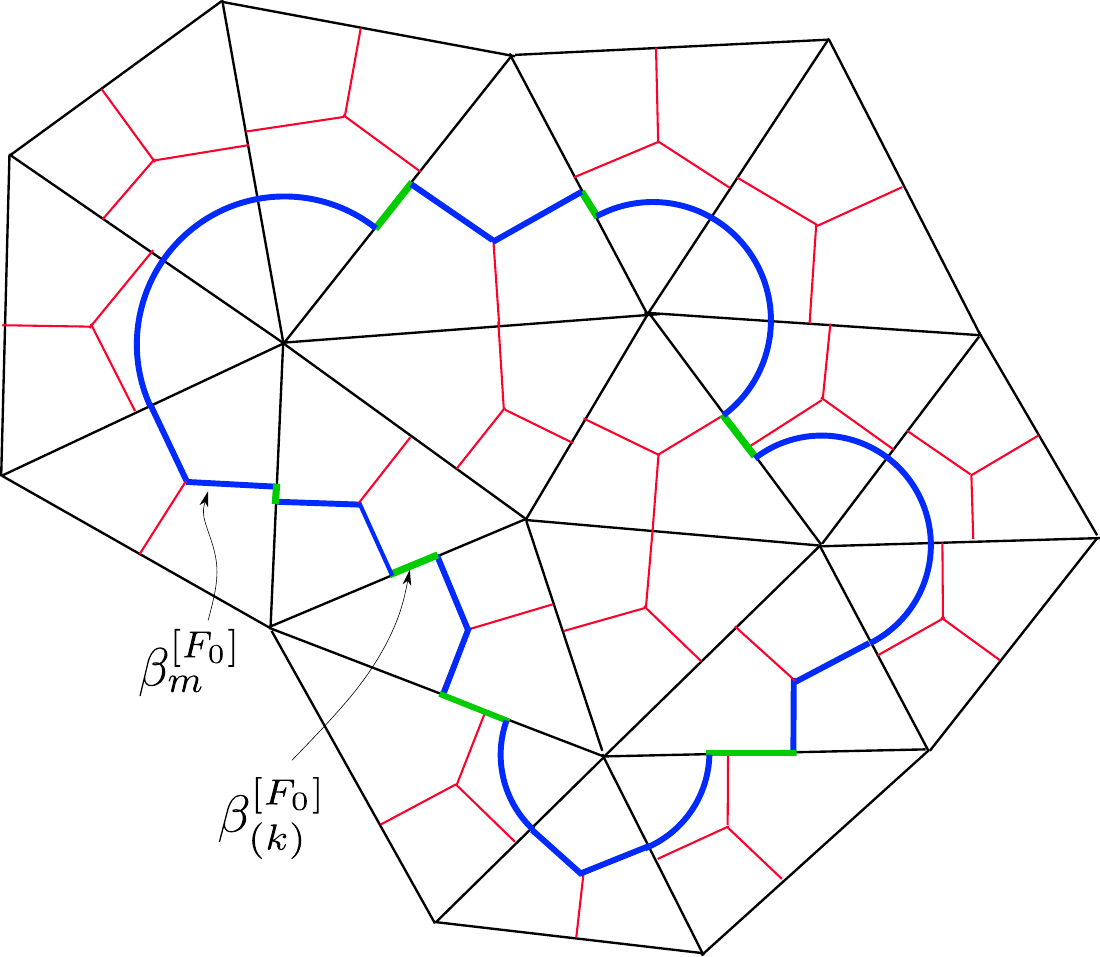}
\caption{$\beta^{[F_0]}: i([F_0],[\beta])=i([F_0],\beta^{[F_0]})$.}
    \label{fig:curv.inter.minimum}
\end{subfigure}
\begin{subfigure}[b]{0.5\textwidth}
\centering
  \includegraphics[width=0.76\linewidth]{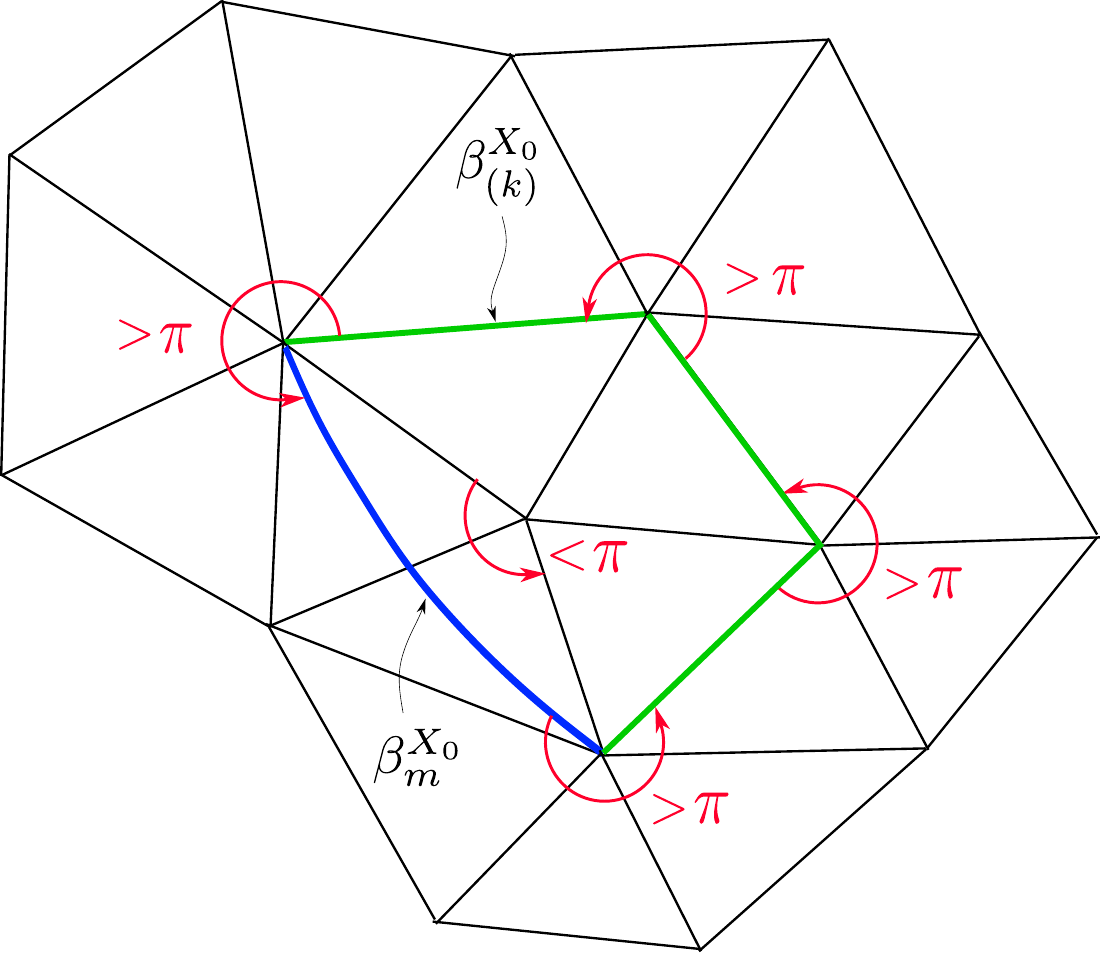}
  \caption{$\beta^{X_0}: \ell_{X_0}([\beta])=\ell_{X_0}(\beta^{X_0})$.}
    \label{fig:curv.length.mimimum}
\end{subfigure}
\caption{Example: a  curve $\beta^{[F_0]}$ (resp. $\beta^{X_0}$) associated to $[\beta]$ for $[F_0]$ (resp. $X_0$).}
\end{figure}
\smallskip
\noindent
\emph{Step 2: Local estimates and the case \(i_{[F]}([\beta])\ge i_{[F_0]}([\beta])\).}
Assume first that \(i_{[F]}([\beta])\ge i_{[F_0]}([\beta])\).
Then
\begin{align*}
0
&\le i_{[F]}([\beta]) - i_{[F_0]}([\beta])\\
&= \inf_{\beta'\in[\beta]} i_{[F]}(\beta') - i_{[F_0]}(\beta^{[F_0]})\\
&\le i_{[F]}(\beta^{[F]}) - i_{[F_0]}(\beta^{[F_0]})\\
&= \sum_{m=1}^{l_1}\bigl(i_{[F]}(\beta^{[F]}_m)-i_{[F_0]}(\beta^{[F_0]}_m)\bigr)
 + \sum_{k=1}^{l_2}\bigl(i_{[F]}(\beta^{[F]}_{(k)})-i_{[F_0]}(\beta^{[F_0]}_{(k)})\bigr)\\
&\le \sum_{m=1}^{l_1}\bigl|i_{[F]}(\beta^{[F]}_m)-i_{[F_0]}(\beta^{[F_0]}_m)\bigr|
 + \sum_{k=1}^{l_2}\bigl|i_{[F]}(\beta^{[F]}_{(k)})-i_{[F_0]}(\beta^{[F_0]}_{(k)})\bigr|.
\end{align*}
Each sub-arc lies in a single face or edge of \(G\), so by
Proposition~\ref{lem:edge.control.arc.MF.triang.} there exists, for any
\(\epsilon_1>0\), a \(\delta_1>0\) such that
\(|i_{[F]}(e)-i_{[F_0]}(e)|<\delta_1\) for all \(e\in E([\beta])\) implies
\[
\bigl|i_{[F]}(\beta^{[F]}_m)-i_{[F_0]}(\beta^{[F_0]}_m)\bigr|<\epsilon_1,
\quad
\bigl|i_{[F]}(\beta^{[F]}_{(k)})-i_{[F_0]}(\beta^{[F_0]}_{(k)})\bigr|<\epsilon_1
\]
for all \(m,k\).
Taking \(\epsilon_1=\epsilon/(l_1+l_2)\) gives
\[
0\le i_{[F]}([\beta]) - i_{[F_0]}([\beta]) < \epsilon.
\]

\smallskip
\noindent
\emph{Step 3: The case \(i_{[F]}([\beta])< i_{[F_0]}([\beta])\).}
Choose instead a representative \(\beta^{[F]}_*\) which realizes
\(i_{[F]}([\beta])\).
Repeating the decomposition and comparison construction with roles of
\([F]\) and \([F_0]\) interchanged, and again applying
Proposition~\ref{lem:edge.control.arc.MF.triang.}, we obtain
\[
0 < i_{[F_0]}([\beta]) - i_{[F]}([\beta]) < \epsilon
\]
under the same smallness condition on
\(|i_{[F]}(e)-i_{[F_0]}(e)|\) for \(e\in E([\beta])\).

\smallskip
Combining the two cases, we see that for any \(\epsilon>0\) there exists
\(\delta>0\) such that \eqref{eq:continuity_inverse} holds whenever
\(|i_{[F]}(e)-i_{[F_0]}(e)|<\delta\) for all \(e\in E([\beta])\).
Since the weak topology on \(\cMFp\) is defined by the intersection functions
\([\beta]\mapsto i_{[F]}([\beta])\), this proves the continuity of
\((\iG_{\cdot})^{-1}\).

The continuity of \((\ell^G_{\cdot})^{-1}\) is proved in exactly the same way,
replacing \([F]\) by \(X\in\cT(G)\), intersection numbers by lengths, and
Proposition~\ref{lem:edge.control.arc.MF.triang.} by
Proposition~\ref{prop:edge.control.arc.hyp.triang.}
(see Figure~\ref{fig:curv.length.mimimum}).
\end{proof}

We are now ready to prove Theorem~\ref{thm:homeo.MFp.shear.radius}.

\begin{proof}[Proof of Theorem~\ref{thm:homeo.MFp.shear.radius}]
By Proposition~\ref{prop:bijection.srG}, the map
\(\mathbf{sr}_G : \cMFp\to\Lambda_G\) is a bijection with inverse
\[
\mathbf{sr}_G^{-1} = (\iG_{\cdot})^{-1}\circ\mathcal{L}^{(G)}_{sr}.
\]

The continuity of \(\mathbf{sr}_G\) follows from the product structure
\eqref{map:prod.str.MF} and the definition of the shear and radius parameters:
the shears \((s_e([F]))_{e\in E(G)}\) depend continuously on the interior
component \([F^{(0)}]\), and the radii \(r_i([F])\) depend continuously on the
peripheral components \([F^{(i)}]\); the decomposition map
\(d:\cMFp\to\cMF_o\Sp\times\prod_{i=1}^n\cMF(p_i)\) is a homeomorphism.

For the inverse, Proposition~\ref{prop:MF.Teich.invers.continuous} gives the
continuity of \((\iG_{\cdot})^{-1}\).
It remains to show that the reconstruction map
\(\mathcal{L}^{(G)}_{sr}:\Lambda_G\to\Omega_G\) is continuous.

Recall from Construction~\ref{const:LG.map} that \(\mathcal{L}^{(G)}_{sr}\)
associates to \((\mathbf{s},\mathbf{r})\in\Lambda_G\) a family of edge weights
\((\ell_e)_{e\in E(G)}\).
For each vertex \(v\), list the incident edges in counterclockwise order
\(e_1,\dots,e_m\) and set
\[
S_v(k) := \sum_{1\le j\le k} \mathbf{s}(e_j),\qquad
S_{v,\min} := \min_{1\le k\le m} S_v(k).
\]
Both \(S_v(k)\) and \(S_{v,\min}\) are continuous in \(\mathbf{s}\).
The corner weights are then given by the explicit formula
\[
a_{f_k}(v) := \mathbf{r}(v) + S_v(k) - S_{v,\min},
\]
which is independent of the choice of an index where the minimum is achieved.
Thus \(a_{f_k}(v)\) depends continuously on \((\mathbf{s},\mathbf{r})\).
Finally, for an edge \(e\) with adjacent faces \(f,f'\) and endpoints \(v,v'\),
we have
\[
\ell_e = a_f(v)+a_f(v'),
\]
so each \(\ell_e\) is a continuous function of \((\mathbf{s},\mathbf{r})\).
Hence \(\mathcal{L}^{(G)}_{sr}\) is continuous.

Since \(\mathbf{sr}_G^{-1}=(\iG_{\cdot})^{-1}\circ\mathcal{L}^{(G)}_{sr}\) is a
composition of continuous maps, it is continuous.
Together with the continuity of \(\mathbf{sr}_G\), this shows that
\(\mathbf{sr}_G\) is a homeomorphism, completing the proof.
\end{proof}

\subsection{Circular foliation charts on the $G$-triangulable locus}
\label{subsec:circ.foli.surf.coor.}

With the continuity of the inverse maps established in
Proposition~\ref{prop:MF.Teich.invers.continuous}, we can now define the
circular foliation map and prove our main coordinate theorems.
In this subsection we first record that the edge parametrizations are
homeomorphisms, then define the circular foliation map and its normal form, and
finally deduce Theorem~\ref{thm:Teich.chart} and the shear--radius coordinates
on~\(\cTp\).

\begin{proposition}[Edge parametrizations are homeomorphisms]
\label{prop:edge.leng.homeo.}
For each \(G \in \TSp\), both the edge intersection map
\(\iG_{\cdot}\) (Definition~\ref{def:MF.edge.inter.})
\[
\iG_{\cdot} : \cMFp \longrightarrow \Omega_G
\]
and the edge length map \(\ell^{(G)}_{\cdot}\) (Definition~\ref{def:edge.leng.Teich.})
\[
\ell^{(G)}_{\cdot}: \cT(G) \longrightarrow \Omega_G
\]
are homeomorphisms.
\end{proposition}

\begin{proof}
By Proposition~\ref{prop:edge.length.MF.biject.} and
Lemma~\ref{lem:edge.length.Teich.biject.}, the maps \(\iG_{\cdot}\) and
\(\ell^G_{\cdot}\) are bijective.
Their continuity follows directly from the definitions of the topologies on
\(\cMFp\) and \(\cT(G)\).
The continuity of the inverses \((\iG_{\cdot})^{-1}\) and
\((\ell^G_{\cdot})^{-1}\) is precisely Proposition~\ref{prop:MF.Teich.invers.continuous}.
\end{proof}

We now define the central transition map of the paper.

\begin{definition}[Circular foliation map]
\label{def:circ.foli.map.}
For \(G \in \TSp\), the \DEF{\em circular foliation map}
\[
\cC_G : \cT(G) \longrightarrow \cMFp
\]
is defined as the composition
\[
\cC_G := (\iG_{\cdot})^{-1} \circ \ell^{(G)}_{\cdot},
\]
where \(\ell^{(G)}_{\cdot}\) is the edge length map and
\((\iG_{\cdot})^{-1}\) is the inverse edge intersection map.
\end{definition}

For each \(X \in \cT(G)\), the image \(\cC_G(X)\) is the unique measured foliation
class whose intersection number with every edge \(e \in E(G)\) equals the length
\(\ell_X(e)\).
We now construct a canonical representative of this class, which we call the
\emph{circular foliation}.
Geometrically, this construction generalises Thurston's construction of horocyclic
foliations in the cusped case \cite[Sec.~4]{Thu98}; here we work with circles
centered at cone points rather than horocycles centered at cusps
(see also \cite[Def.~3.7]{PT07}).

\begin{construction}[Circular foliation on a hyperbolic triangle]
\label{con:circ.foli.surf.}
Let \((\Delta,h) \in \cT(\Delta)\) be a hyperbolic geodesic triangle with vertices
\(p_1,p_2,p_3\) and opposite edges \(e_1,e_2,e_3\) of lengths
\(\ell_1,\ell_2,\ell_3\), respectively.
Define
\[
r_i := \frac{\ell_j + \ell_k - \ell_i}{2},
\qquad \{i,j,k\} = \{1,2,3\},
\]
so that each \(r_i\) is the distance from \(p_i\) to the incircle tangency points on the two edges adjacent to \(p_i\).

The \DEF{\em circular foliation} associated to \((\Delta,h)\) is a measured foliation
\(F\) on \(\Delta\) defined as follows.
\begin{itemize}
\item \emph{Leaves.}
For each vertex \(p_i\), foliate a neighbourhood of \(p_i\) by arcs of hyperbolic
circles centered at \(p_i\), with radii between \(0\) and \(r_i\).
There is a unique triple of circles of radii \(r_1,r_2,r_3\) (centered at $p_1$, $p_2$, $p_3$) that are pairwise
tangent and bound a central region \(R\subset\Delta\); outside \(R\), the leaves
are circular arcs centered at the vertices (see Figure~\ref{fig:tri-circ-foliation}).

\begin{figure}[htp]
    \centering  \includegraphics[width=12cm]{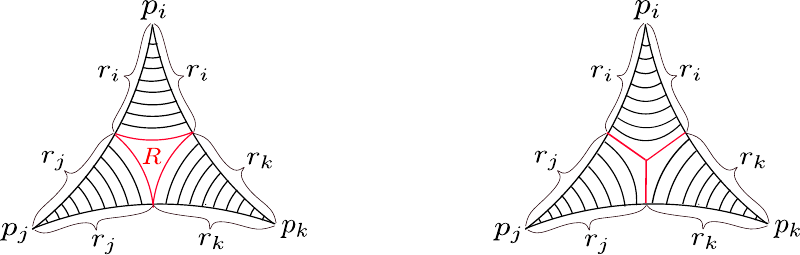}
    \caption{The circular foliation  associated to a hyperbolic geodesic triangle, where the right one is equivalent to the left one, by collapsing the unfoliated region $R$ to a tripod.}
    \label{fig:tri-circ-foliation}
\end{figure}

\item \emph{Transverse measure.}
Let \(\pi:T\Delta\to T\Delta\) denote the projection onto the orthogonal
complement of the leaf directions, and declare \(\pi(v)=0\) whenever the base
point of \(v\) lies in \(R\).
For any \(C^1\) arc \(\alpha:[0,1]\to\Delta\), the transverse measure of
\(\alpha\) is defined by
\[
F(\alpha) := \int_0^1 \|\pi(\alpha'(t))\|_h\,dt,
\]
where \(\|\cdot\|_h\) is the norm induced by the hyperbolic metric \(h\).
If two arcs are isotopic through arcs staying on the same leaves, they have the
same transverse measure.
\end{itemize}

This construction gives a partial measured foliation on \(\Delta\) supported in
\(\Delta\setminus R\).
By modifying the foliation in a small neighbourhood of \(R\) and collapsing
\(R\) to a tripod (see the right-hand picture in
Figure~\ref{fig:tri-circ-foliation}), one obtains a measured foliation with full
support on \(\Delta\) (still denoted \(F\)), which lies in \(\cMF(\Delta^{\perp})\)
and is unique up to equivalence.
\end{construction}

\begin{remark}[Relation to classical horocyclic foliations]
\label{rem:circular.cusp.relation}
In the classical setting of hyperbolic surfaces with cusps, Thurston's stretch
deformations are naturally described using horocyclic foliations centered at the
cusps \cite[Sec.~4]{Thu98}; see also the exposition in \cite[Def.~3.7]{PT07}.
Our circular foliations on hyperbolic triangles should be viewed as cone-surface
analogues of these horocyclic foliations, with cone points replacing cusps and
circle radii determined by the edge length data \(\ell^{(G)}_{\cdot}(X)\).
When cone angles tend to \(0\), the radii \(r_i\) tend to infinity and the
construction recovers the usual horocyclic foliations on cusp neighbourhoods.
\end{remark}

For each triangle \(\Delta_f\) of \(G\), a hyperbolic cone-metric
\(X\in\cT(G)\) restricts to a hyperbolic metric \(h_f := X|_{\Delta_f}\) on
\(\Delta_f\).
Applying Construction~\ref{con:circ.foli.surf.} to each \((\Delta_f,h_f)\) and
gluing the resulting foliations along the edges of \(G\), we obtain a measured
foliation on \(\Sp\) which is totally transverse to \(G\).
This is precisely the $G$-normal form of \(\cC_G(X)\) with respect to \(G\), in the
sense of Definition \ref{def:G.normal.form}.

\medskip

Recall that Theorem~\ref{thm:Teich.chart} asserts that \(\cTp\) is a manifold of
(real) dimension \(6g-6+3n\), with local charts
\(\big(\cT(G),\cC_G:\cT(G)\to\cMFp\big)\) for \(G\in\TSp\).

\begin{proof}[Proof of Theorem \ref{thm:Teich.chart}]
Fix \(G \in \TSp\).
By Proposition~\ref{prop:edge.leng.homeo.}, the map
\(\ell^{(G)}_{\cdot}: \cT(G)\to\Omega_G\) is a homeomorphism onto the open cone
\(\Omega_G\) (Definition~\ref{def:tri.ineq.region}).
In particular, \(\cT(G)\) is open in \(\cTp\).
Using Definition~\ref{def:Teich.cone.surf.} and
Remark~\ref{rem:charts_cover_space}, we see that
\(\{\cT(G)\}_{G\in\TSp}\) is an open cover of~\(\cTp\).

By Definition~\ref{def:circ.foli.map.} and
Proposition~\ref{prop:edge.leng.homeo.}, the circular foliation map
\(\cC_G = (\iG_{\cdot})^{-1}\circ \ell^{(G)}_{\cdot}: \cT(G)\to\cMFp\) is a
homeomorphism, so each pair \((\cT(G),\cC_G)\) is a coordinate chart modelled on
\(\cMFp\cong\bR^{6g-6+3n}\).
Since these charts cover \(\cTp\), the theorem follows.
\end{proof}

By definition, every hyperbolic cone-surface in \(\cT_{u.t.}\Sp\) is
\(G\)-triangulable for all \(G\in\TSp\).
For each \(G\), the circular foliation map
\(\cC_G:\cT_{u.t.}(S,\up)\rightarrow \cMFp\) is therefore an embedding
(see Corollary~\ref{cor:embed}).

\begin{definition}[Shear--radius coordinates of Teichm\"uller space]
\label{def:shear-radius}
For each triangulation \(G\in\TSp\), the \DEF{\em shear--radius map} of \(\cT(G)\) is
\[
\cSR_G := \mathbf{sr}_G \circ \cC_G : \cT(G) \longrightarrow \Lambda_G,
\]
where \(\mathbf{sr}_G\) is the shear--radius map on \(\cMFp\)
(Definition~\ref{def:MF-shear-radius}) and \(\cC_G\) is the circular foliation
map (Definition~\ref{def:circ.foli.map.}).
We also write \(\cS_G := \mathbf{s}_G\circ\cC_G\) and
\(\cR_G := \mathbf{r}_G\circ\cC_G\) for the shear and radius maps on \(\cT(G)\).

By Theorem~\ref{thm:homeo.MFp.shear.radius} and Theorem~\ref{thm:Teich.chart},
the map \(\cSR_G\) is a homeomorphism onto \(\Lambda_G\)
(see Corollary~\ref{cor:homeo.TS.shear.radius}).
The coordinates given by \(\cSR_G\) are called the \DEF{\em shear--radius coordinates}
on~\(\cTp\).
\end{definition}

 \section{Stretch-type deformations and asymptotic limits}
\label{subsec:application}

In this section we use the circular foliation coordinates developed in Section~\ref{sec:circ.foli.coor.app.} to address Question~\ref{Q:deform.Teich.} on natural deformations of hyperbolic cone-surfaces.
Classically, Thurston's stretch deformations provide geodesics for the Thurston (Lipschitz) metric on Teichm\"uller space \cite{Thu98}, and related constructions have been developed on the once-holed torus \cite{HP21} and via harmonic map rays \cite{PW22}; see also Robertson--Rupflin \cite{RR20} for a flow-based perspective in which degeneration of Teichm\"uller harmonic map flow is governed by the stretching of thin collars.

Our aim here is to construct cone-surface analogues of these deformations, formulated in the circular foliation coordinates, and to describe their asymptotic behaviour. More precisely, we introduce two families of deformations on each chart \(\cT(G)\):
\begin{itemize}
  \item \emph{Peripheral stretch deformations}, along which the cone angles decrease and the surfaces converge to a hyperbolic surface with cusps, in the topology induced by length spectrum on simple closed curves (Theorem~\ref{thm:asym.behavior}); and
  \item \emph{Interior anti-stretch deformations}, along which cone angles increase and the surfaces converge to a circle-packed hyperbolic cone-surface (Theorem~\ref{thm:asym.behavior1}).
\end{itemize}

The starting point is the product structure \eqref{map:prod.str.MF} of the model space \(\cMFp\), which allows us to define \emph{partial stretch deformations} (Definition~\ref{def:part.stret.}) by rescaling only the peripheral components or only the interior component of the circular foliation.
Via the circular foliation map, these rescalings induce geometric deformations of cone-metrics on \(\cT(G)\) that provide a controlled way of varying cone angles within our coordinate charts.

The rest of the section is devoted to the asymptotic analysis of these deformations.
The proof of the peripheral stretch theorem (Theorem~\ref{thm:asym.behavior}) is technically more involved: in Subsections~\ref{subsec:asymp.triang.} and~\ref{subsec:bounding-geodesics} we establish a key geometric control estimate (Lemma~\ref{lm:dist-tangent-intersect}) on the behaviour of geodesics under peripheral stretch.
The interior anti-stretch case (Theorem~\ref{thm:asym.behavior1}) is then treated in a parallel, but simpler version.

\subsection{Partial (anti-)stretch deformations of Teichm\"uller space}
\label{subsec:part.stretch.deform.}

Recall that for each $G\in\TSp$, the circular foliation of a hyperbolic cone-surface $X\in\cT(G)$ is decomposed into peripheral components $[F^{(i)}]$ ($1\leq i\leq n$) and an interior component $[F^{(0)}]$, as in the product structure \eqref{map:prod.str.MF}.

\begin{definition}[Partial scaling actions]
\label{def:peripheral.scaling}
For any $t \in \mathbb{R}$, we define two self-homeomorphisms of the space $\cMFp$:
\begin{itemize}
    \item The \DEF{\em peripheral scaling action} ${\rm exp}_{\mathrm{per}}(t): \cMFp \to \cMFp$ scales the transverse measure of each peripheral component $[F^{(i)}]$ by the factor $e^t$, while keeping the interior component $[F^{(0)}]$ fixed.
    \item The \DEF{\em interior scaling action} ${\rm exp}_{\mathrm{int}}(t): \cMFp \to \cMFp$ scales the transverse measure of the interior component $[F^{(0)}]$ by the factor $e^t$ (when $[F^{(0)}]\neq 0$), while keeping all peripheral components $[F^{(i)}]$ fixed; if $[F^{(0)}]=0$, then ${\rm exp}_{\mathrm{int}}(t)$ is the identity.
\end{itemize}
\end{definition}

For each $t\in\bR$, the peripheral scaling action ${\rm exp}_{\mathrm{per}}(t)$ and the interior scaling action ${\rm exp}_{\mathrm{int}}(t)$ are both homeomorphisms of $\cMFp$.

\begin{definition}[Partial (anti-)stretch deformations]
\label{def:part.stret.}
For a given triangulation $G \in \TSp$, we define two one-parameter families of deformations on the space $\cT(G)$, called \DEF{\em partial (anti-)stretch deformations}. 
Both are obtained by pulling back the corresponding scaling actions on $\cMFp$ via the circular foliation map $\cC_G$:
\[
 \xymatrix{
  \cT(G) \ar[d]_{\cC_G} \ar[r]^{{\rm stretch}_G^{\cdot}} & \cT(G) \ar[d]^{\cC_G}  \\
  \cMFp \ar[r]^{{\rm exp}_{\cdot}(t)} &  \cMFp .
 }
\]

\begin{itemize}
    \item The \DEF{\em peripheral (anti-)stretch deformation} is the map
    \[
      {\rm stretch}^{\mathrm{per}}_G : \cT(G) \times \mathbb{R} \longrightarrow \cT(G),\qquad
      {\rm stretch}^{\mathrm{per}}_G(X,t) := \cC_G^{-1} \circ {\rm exp}_{\mathrm{per}}(t) \circ \cC_G(X).
    \]
    
    \item The \DEF{\em interior (anti-)stretch deformation} is the map
    \[
      {\rm stretch}^{\mathrm{int}}_G : \cT(G) \times \mathbb{R} \longrightarrow \cT(G),\qquad
      {\rm stretch}^{\mathrm{int}}_G(X,t) := \cC_G^{-1} \circ {\rm exp}_{\mathrm{int}}(t) \circ \cC_G(X).
    \]
\end{itemize}
For both types of deformation, the case $t > 0$ is called a \DEF{\em stretch}, and the case $t < 0$ is called an \DEF{\em anti-stretch}.
\end{definition}

\begin{remark}[Stretch deformations and our variants]
\label{rem:HP-vs-us}
On the classical Teichm\"uller space of hyperbolic surfaces, Thurston's \emph{stretch deformations} associated to complete geodesic laminations produce \emph{stretch paths}, which are geodesics for the asymmetric Thurston (Lipschitz) metric and are naturally described in terms of horocyclic foliations in the complementary ideal triangles \cite{Thu98}. 
Variants of this picture include the \emph{partial stretch maps} of Huang--Papadopoulos on the once-holed torus, constructed from partial horocyclic foliations and yielding optimal Lipschitz maps and Thurston geodesic rays \cite{HP21}, as well as the harmonic map rays of Pan--Wolf, which select canonical Thurston geodesics and describe ray structures for the Thurston metric \cite{PW22}; for the Thurston metric on Teichm\"uller spaces of hyperbolic cone-surfaces with fixed cone angles and boundary lengths, see also \cite{Pan17}.

The partial stretch deformations introduced here are of a different nature: 
they act on the Teichm\"uller space of hyperbolic cone-surfaces, are defined via the circular foliation coordinates from Theorem~\ref{thm:Teich.chart}, and allow the cone angles to vary along the deformation. 
\end{remark}

\subsection{Asymptotic geometry of hyperbolic triangles}
\label{subsec:asymp.triang.}

The goal of this subsection is to analyse the asymptotic geometry of hyperbolic triangles as their vertices tend to infinity. 
This is a necessary step towards proving Theorem~\ref{thm:asym.behavior}: along a peripheral stretch ray \(X_t\) (as \(t \to +\infty\)), the radii \(r_i(X_t)\) diverge, and the geodesic triangles \(f(X_t)\) on $X_t$ associated to the face $f$ of the triangulation $G$ degenerate towards ideal triangles.
To understand the limit of \(X_t\), we first make precise the limiting object (the ideal triangle), choose canonical representatives for all triangles, and then prove convergence in this canonical model.

\begin{definition}[Extended Teichm\"uller space of a triangle]
\label{def:extend.Teich.of.hyp.triang.}
Let \(\Delta\) be a topological triangle (Definition~\ref{def:top.triang.}). 
The \DEF{\em extended Teichm\"uller space of \(\Delta\)}, denoted \(\bar{\cT}(\Delta)\), is the set of isotopy classes of hyperbolic metrics on \(\Delta\) that realize it as a geodesic hyperbolic triangle with edges of possibly infinite length.

This space contains a distinguished element \(T_{\infty}\), corresponding to the metric that realizes \(\Delta\) as an \DEF{\em ideal hyperbolic triangle}, i.e.\ a triangle whose vertices lie on \(\partial\bH^2\) and whose interior angles are all zero.
\end{definition}

\begin{definition}[Canonical representative of a hyperbolic triangle]
\label{def:canonical-triangles}
For each \(T\in\bar{\cT}(\Delta)\), we associate a \DEF{\em canonical marked hyperbolic triangle} \(\Delta_T\subset\bH^2\) (in the Poincar\'e disk model) satisfying:
\begin{enumerate}
  \item The center of the hyperbolic circle \(C_T\) tangent to all three edges of \(\Delta_T\) locates at the origin \(O\) of \(\bH^2\).
  \item The first marked vertex \(v_1(T)\) lies on the positive \(x\)-axis.
  \item The vertices \(v_1(T),v_2(T),v_3(T)\) occur in anticlockwise order.
\end{enumerate}

In particular, the canonical representative \(\Delta_{T_{\infty}}\) of the ideal triangle \(T_{\infty}\) has vertices \(v^{\infty}_1,v^{\infty}_2,v^{\infty}_3\) on \(\partial\bH^2\) which divide the unit circle into three arcs of equal Euclidean length (see the right-hand panel of Figure~\ref{fig:canonical-triangle}).

For any \(T\in\bar{\cT}(\Delta)\) and \(k=1,2,3\), we adopt the following notation for its canonical representative \(\Delta_T\subset\bH^2\):
\begin{itemize}
  \item \(v_k(T)\) is the \(k\)-th vertex and \(e_k(T)\) is the edge opposite to \(v_k(T)\).
  \item \(l_k^T\) is the geodesic ray from \(O\) to \(v_k(T)\).
  \item \(\theta_k(T)\) is the angle at \(O\) between the rays \(l_k^T\) and \(l_{k+1}^T\) (indices taken modulo \(3\)).
  \item \(d_H\) and \(d_E\) denote the hyperbolic and Euclidean distances on \(\bH^2\), respectively.
\end{itemize}
\end{definition}

\begin{lemma}[Asymptotic geometry of the canonical triangle]
\label{lm:estimate.triangle}
Let \(T\in\bar{\cT}(\Delta)\). Then, as the three vertices of \(\Delta_T\) tend to the ideal boundary of \(\bH^2\), i.e.
\(d_H\bigl(O,v_k(T)\bigr)\to\infty \) for \(k=1,2,3\), we have
\[
d_E\bigl(v_k(T),v_k^{\infty}\bigr)\to 0\quad \text{and} \quad
\theta_k(T)\to \frac{2\pi}{3},
\]
for each \(k=1,2,3\).
\end{lemma}

\begin{proof}
For each \(k\), let \(\bar e_k(T)\) be the complete geodesic in \(\bH^2\) extending the edge \(e_k(T)\), and let \(\xi_k(T)\) be the angle at \(O\) subtended by the endpoints of \(\bar e_k(T)\) on \(\partial\bH^2\).
By Condition (1) in Definition~\ref{def:canonical-triangles} the circle \(C_T\) is centered at \(O\), so the three geodesics \(\bar e_1(T),\bar e_2(T),\bar e_3(T)\) are pairwise symmetric with respect to the rays \(l_k^T\) through the vertices.
Hence
\begin{equation}\label{eq:edge.angle.equal}
  \xi_1(T)=\xi_2(T)=\xi_3(T)\, .
\end{equation}

As \(d_H\bigl(O,v_k(T)\bigr)\to\infty\) for all \(k\), the Euclidean radii \(d_E\bigl(O,v_k(T)\bigr)\) tend to \(1\), so each \(v_k(T)\) converges to a point of the boundary \(\partial\bH^2\) of the unit disk model \(\bH^2\).
By Condition (2) in Definition~\ref{def:canonical-triangles}, the first vertex tends to \(v_1^{\infty}\), and by symmetry we deduce from \eqref{eq:edge.angle.equal} that, in the limit, the three arcs of \(\partial\bH^2\) cut out by the geodesics \(\bar e_k(T)\) all have angle \(2\pi/3\).
Thus
\begin{equation}\label{eq:edge.angle.equal.ideal}
  \xi_1(T)=\xi_2(T)=\xi_3(T)\longrightarrow \frac{2\pi}{3}
\end{equation}
as \(d_H\bigl(O,v_k(T)\bigr)\to\infty\) for all \(k\).

The limiting configuration in \eqref{eq:edge.angle.equal.ideal} is exactly the canonical ideal triangle \(\Delta_{T_{\infty}}\), whose vertices are \(v_k^{\infty}\) and whose central angles are \(2\pi/3\).
Therefore each \(v_k(T)\) converges to \(v_k^{\infty}\) in the Euclidean metric, and the angles \(\theta_k(T)\) between the rays \(l_k^T\) and \(l_{k+1}^T\) converge to \(2\pi/3\).
\end{proof}

\subsection{Bounding geodesics under peripheral stretch}
\label{subsec:bounding-geodesics}

This subsection contains the main technical parts for the proof of Theorem~\ref{thm:asym.behavior}.
To show that the length \(\ell_{X_t}([\gamma])\) converges to \(\ell_{X_{\infty}}([\gamma])\) for each curve class \([\gamma]\in\cS\Sp\), we need uniform control on how the geodesic representative \(\gamma(X_t)\) sits inside the triangles \(f(X_t)\) as they degenerate under peripheral stretch.

Our key tool is Lemma~\ref{lm:dist-tangent-intersect}, which provides a uniform bound: along a peripheral stretch ray, any simple closed geodesic crosses each edge essentially near the incircle tangency points, with a deviation that is uniformly bounded in both time and in the homotopy class of the geodesic.

In order to control the behavior of geodesics during the peripheral stretch deformation, it suffices to consider the hyperbolic cone-surfaces whose cone angles are strictly less than \(\pi\).
In this range a collar lemma holds around each cone point \cite[Thm.~3]{DP07}, providing embedded neighborhoods where simple closed geodesics and geodesic arcs behave in a controlled manner.
When some cone angle reaches or exceeds \(\pi\), such a collar lemma fails and, for suitable hyperbolic metrics, there may even be homotopy classes without a smooth geodesic representative \cite[Lem.~7.3]{Mon10}.
This dichotomy will play a crucial role in the proof of Theorem~\ref{thm:asym.behavior}.

Let
\[
  \cT_{[0,\pi)}(S,\up)
  := \cT_{(0,\pi)}(S,\up) \cup \cT_{0}(S,\up),
\]
where \(\cT_0(S,\up)\) is the Teichm\"uller space of hyperbolic metrics on \((S,\up)\) with cusps at \(\up\).

\begin{lemma}[Collar lemma for hyperbolic cone-surfaces {\cite[Thm.~3]{DP07}}]
\label{lm:Collar-lemma-for-hyperbolic-cone-surfaces}
Let \(X \in \cT_{[0,\pi)}(S,\up)\) and let \(\theta_{\max}(X)\) denote the largest cone angle of \(X\).
Then there exists \(C \in (0,1]\), depending only on \(\theta_{\max}(X)\), such that the disks \(D_i\) centered at \(p_i\) \((1\le i\le n)\) with circumference less than \(C\) are pairwise disjoint and do not meet any simple closed geodesic of \(X\).
\end{lemma}

We now fix some notation that will be used throughout this subsection.
Let \(G \in \TSp\) be a triangulation and \([\gamma] \in \cS(S,\up)\) a simple closed curve.
Choose a representative \(\gamma_G : [0,1] \to \dot{S}\) that intersects \(G\) transversely, and let
\(E([\gamma]) \subset E(G)\) be the set of edges crossed by \(\gamma_G\).

For any hyperbolic cone-surface \(X\), let \(\gamma(X)\) denote the unique shortest geodesic representative of \([\gamma]\) on \(X\).
For each edge \(e \in E(G)\) and face \(f \in F(G)\), write \(e(X)\) and \(f(X)\) for their geodesic realizations on \(X\).
For each \(e \in E(G)\) and each adjacent face \(f \in F(e)\), we denote by
\(\tau_{e,f}(X) \in e(X)\) the tangency point of the incircle of the hyperbolic triangle
\(f(X)\) on the edge \(e(X)\).

\begin{lemma}[Uniformly bounded deviation of geodesics]
\label{lm:dist-tangent-intersect}
Let \(G \in \TSp\) be a triangulation and \(X_0 \in \cT(G)\) an initial surface.
Let \((X_t)_{t \ge 0}\) be the peripheral stretch ray starting from \(X_0\) with respect to \(G\).
Then there exists a constant \(D > 0\), depending only on \(G\) and \(X_0\), such that for every time \(t \ge 0\), every simple closed curve class \([\gamma] \in \cS(S,\up)\), every edge \(e \in E([\gamma])\), every face \(f \in F(e)\), and every intersection point
\(q(X_t) \in e(X_t) \cap \gamma(X_t)\), we have
\[
  d_{X_t}\bigl(\tau_{e,f}(X_t), q(X_t)\bigr) \le D.
\]
In other words, along the peripheral stretch ray, each intersection of \(\gamma(X_t)\) with an edge of \(G\) stays within a uniformly bounded distance of the corresponding incircle tangency points \(\tau_{e,f}(X_t)\).
\end{lemma}

\begin{proof}
The function
\[
  t \;\longmapsto\;
  d_{X_t}\!\bigl(\tau_{e,f}(X_t), q(X_t)\bigr)
\]
is continuous for \(t \ge 0\).
Since at \(t=0\) this distance is bounded above (for all \([\gamma]\) and intersection points) by a constant depending only on \(X_0\), it suffices to obtain a uniform bound for all sufficiently large \(t\).

Assume by contradiction that no such uniform bound exists.
Then for each integer \(n > 0\) there exist:
\[
  t_n > n,\quad
  [\gamma_n] \in \cS(S,\up),\quad
  e_n \in E([\gamma_n]),\quad
  f_n \in F(e_n),
\]
and a point
\[
  q_n(X_{t_n}) \in e_n(X_{t_n}) \cap \gamma_n(X_{t_n})
\]
such that
\begin{equation}
\label{eq:max.number}
  d_{X_{t_n}}\!\bigl(\tau_{e_n,f_n}(X_{t_n}), q_n(X_{t_n})\bigr) > n.
\end{equation}
By reparametrizing the stretch ray if necessary we may assume \(t_n = n\) and write \(X_n := X_{t_n}\).

\medskip\noindent
\emph{Step 1: Passing to a subsequence.}
Since \(G\) has finitely many edges, there exists \(e \in E(G)\) that coincides with \(e_n\) for infinitely many \(n\).
Passing to a subsequence, we may assume \(e_n = e\) for all \(n\).
Every edge \(e\) is adjacent to at most two faces; after possibly passing to another subsequence, we can assume that \(f_n = f\) is the face on a fixed side of \(e\) (say, the left side) for all \(n\).

Each point \(q_n(X_n) \in e(X_n)\cap\gamma_n(X_n)\) lies either to the left or to the right of \(\tau_{e,f}(X_n)\) on \(e(X_n)\).
Passing to a further subsequence if needed, we may suppose that all points \(q_n(X_n)\) lie on the same side (say, the left side) of \(\tau_{e,f}(X_n)\) for all \(n\).

Set
\[
  \gamma^{(n)} := \gamma_n(X_n),\quad
  e^{(n)} := e(X_n),\quad
  f^{(n)} := f(X_n),\quad
  q^{(n)} := q_n(X_n),\quad
  \tau_{e,f}^{(n)} := \tau_{e,f}(X_n).
\]
Then \eqref{eq:max.number} becomes
\begin{equation}
\label{eq:assump}
  d_{X_n}\!\bigl(\tau_{e,f}^{(n)}, q^{(n)}\bigr) > n.
\end{equation}

\medskip\noindent
\emph{Step 2: Local notation near a cone point.}
We now introduce local notation around the endpoint (denoted by $p$) of \(e\) that lies on the same side of \(\tau_{e,f}^{(n)}\) as \(q^{(n)}\) (see Figure~\ref{fig:estimate-dist}).

\begin{itemize}
\item \(C^{(n)}\): the hyperbolic circle of \(X_n\) centered at \(p\) and passing through \(q^{(n)}\).
    \item \(r^{(n)}:= r_p(X_n)\): the radius coordinate of \(X_n\) at \(p\) (Definition~\ref{def:shear-radius}). \item \(c^{(n)}\): the circle of radius \(r^{(n)}\) centered at \(p\).
    \item \((e_1,\dots,e_m)\): the edges of \(G\) incident to \(p\), ordered counterclockwise and chosen with \(e_1 = e\). 
    \item \(e_i^{(n)}\): the geodesic realizations of $e_i$ on \(X_n\).
    \item \(f_i\): the face bounded by the consecutive edges \(e_i\) and \(e_{i+1}\) (indices taken modulo \(m\)). \item \(f_i^{(n)}\): the geodesic realization of $f_i$ on \(X_n\).
    \item \(\tau_{j,i}^{(n)}\): the tangency point of the incircle of \(f_i^{(n)}\) with the edge \(e_j^{(n)}\), where \(j \in \{i,i+1\}\).
\end{itemize}

\begin{figure}[htp]
    \centering
    \includegraphics[width=13cm]{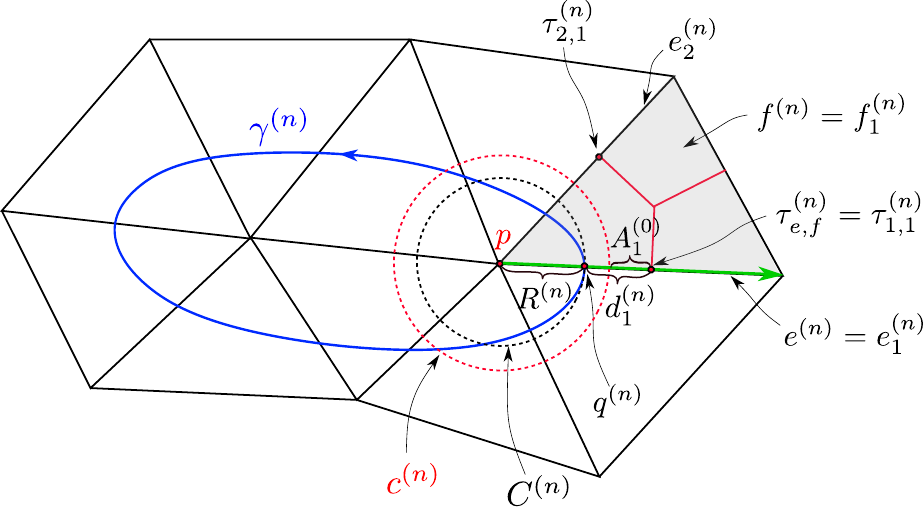}
    \caption{An example of \(\gamma^{(n)}\) on \(X_n\) with \(d_{X_n}(\tau_{1,1}^{(n)}, q^{(n)})>n\).}
    \label{fig:estimate-dist}
\end{figure}

\medskip\noindent
\emph{Step 3: A quantity invariant under peripheral stretch.}
By Definition~\ref{def:part.stret.}, the shear coordinates of \(X_n\) are independent of \(n\).
For each fixed \(i\), the position of the tangent point \(\tau_{i,i}^{(n)}\) relative to the vertex \(p\) and the circle \(c^{(n)}\) depends only on the shear data and the radius coordinate \(r^{(n)}\).
Under peripheral stretch, the shears remain fixed and the radius coordinate at \(p\) is scaled uniformly, so the signed distance from \(\tau_{i,i}^{(n)}\) to the circle \(c^{(n)}\) is independent of \(n\).
Thus there exists a constant \(A_i^{(0)}\), depending only on \(G\) and \(X_0\), such that
\begin{equation}
\label{eq:const}
  d_{X_n}\!\bigl(\tau_{i,i}^{(n)}, c^{(n)}\bigr) = A_i^{(0)}
  \quad\text{for all }n\text{ and }1\le i\le m.
\end{equation}

\medskip\noindent
\emph{Step 4: Shrinking of the circumference.}
Let \(L^{(n)}\) be the length of the circle \(C^{(n)}\).
We claim that
\begin{equation}
\label{eq:shrink.circum.}
  L^{(n)} \longrightarrow 0
  \quad\text{as }n\to\infty.
\end{equation}

Set \(d_i^{(n)} := d_{X_n}(\tau_{i,i}^{(n)}, C^{(n)})\).
In particular, by \eqref{eq:assump} we have
\begin{equation}
\label{eq:assump-dist-infty}
  d_1^{(n)} = d_{X_n}(\tau_{1,1}^{(n)}, C^{(n)})
           = d_{X_n}(\tau_{1,1}^{(n)}, q^{(n)}) \to \infty
  \quad\text{as }n\to\infty.
\end{equation}
Let \(R^{(n)} := d_{X_n}(p, q^{(n)})\) be the radius of \(C^{(n)}\).
Then
\[
  d_{X_n}(p, \tau_{1,1}^{(n)})
  = d_{X_n}(p, q^{(n)}) + d_{X_n}(q^{(n)}, \tau_{1,1}^{(n)})
  = R^{(n)} + d_1^{(n)}.
\]
The radius \(r^{(n)}\) of \(c^{(n)}\) satisfies
\[
  r^{(n)}
  = d_{X_n}(p, c^{(n)})
  = d_{X_n}(p, \tau_{1,1}^{(n)}) - d_{X_n}(\tau_{1,1}^{(n)}, c^{(n)})
  = R^{(n)} + d_1^{(n)} - A_1^{(0)},
\]
using \eqref{eq:const} for \(i=1\).

Write \(l^{(n)}\) for the length of \(c^{(n)}\).
Using elementary hyperbolic trigonometry for circles of radius \(R^{(n)}\) and \(r^{(n)}\), and combining \eqref{eq:assump-dist-infty}, we obtain
\begin{equation}
\label{eq:ratio_l_L_tend_to_0}
  \frac{L^{(n)}}{l^{(n)}}
  = \frac{\theta^{(n)} \sinh R^{(n)}}{\theta^{(n)} \sinh r^{(n)}}
  = \frac{1 - e^{-2R^{(n)}}}
         {e^{d_1^{(n)}-A_1^{(0)}} - e^{-(2R^{(n)}+d_1^{(n)}-A_1^{(0)})}}
  \le
  \frac{1}{e^{d_1^{(n)}-A_1^{(0)}} - e^{-(d_1^{(n)}-A_1^{(0)})}}
  \longrightarrow 0
\end{equation}
as \(n\to\infty\), where \(\theta^{(n)}\) is the cone angle of \(X_n\) at \(p\).

\begin{figure}[htp]
    \centering  \includegraphics[width=15.5cm]{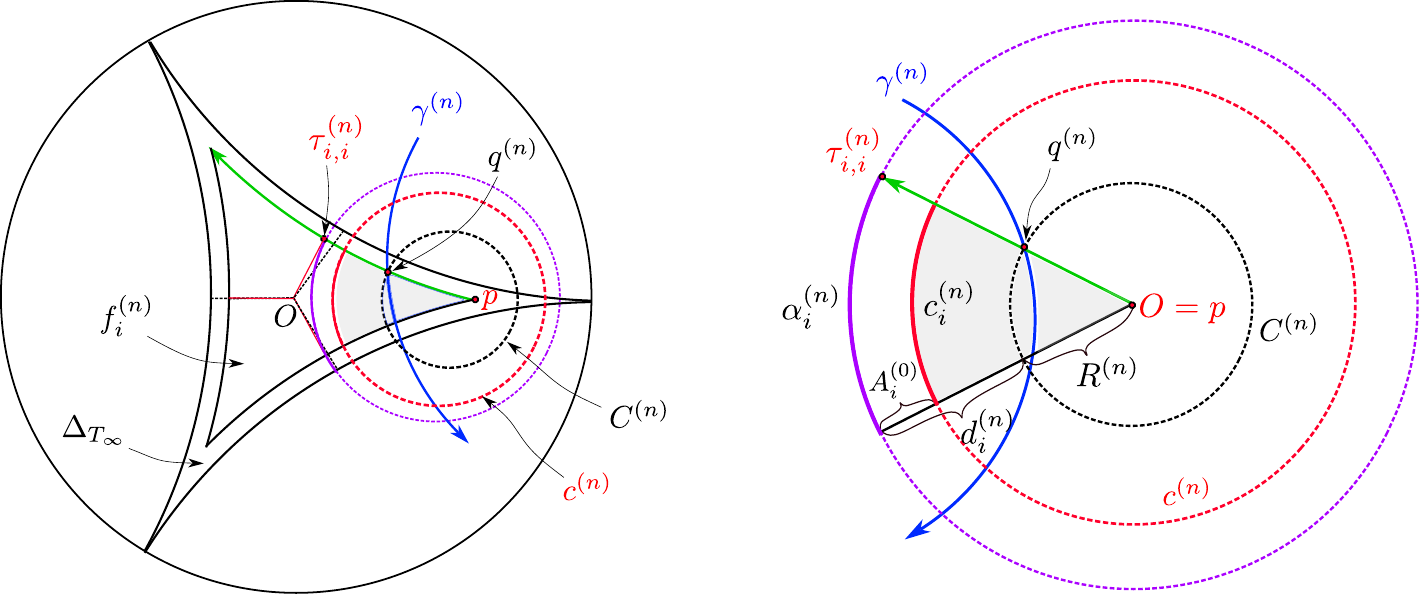}
    \caption{The left picture is the canonical representative of the geodesic triangle $f_i^{(n)}$, while the right one is a picture near $p$, with $p$ put in the origin $O$ of $\bH^2$.}
\label{fig:contradiction-dist}
\end{figure}

It remains to show that the sequence \((l^{(n)})_n\) is uniformly bounded above.
For each fixed \(i\), the triangles \(f_i^{(n)}\) converge, in the canonical normalization (see for instance Figure \ref{fig:contradiction-dist}) of Subsection~\ref{subsec:asymp.triang.}, to an ideal triangle \(f_i^{(\infty)}\).
In that normalization, the circle \(c^{(n)}\) has radius \(r^{(n)} \to \infty\), so its intersection \(c_i^{(n)} := c^{(n)} \cap f_i^{(n)}\) converges to a horocyclic arc \(h_i := h \cap f_i^{(\infty)}\), where \(h\) is a horocycle centered at the limiting ideal vertex of \(p\).

By the convergence of the canonical representatives and the fact that the distance from \(\tau_{i,i}^{(n)}\) to \(c^{(n)}\) is constant (cf.\ \eqref{eq:const}), the limiting horocycle \(h\) lies at a fixed distance \(A_i^{(0)}\) from a reference horocycle through the limiting tangency point.
Thus the limiting arc \(h_i\) has a finite length \(l_i^{(\infty)}\) depending only on the limiting ideal triangle \(f_i^{(\infty)}\) and the constant \(A_i^{(0)}\).
In particular, \(\ell_{X_n}(c_i^{(n)}) \to l_i^{(\infty)}\) as \(n\to\infty\), so the sequence \((\ell_{X_n}(c_i^{(n)}))_n\) is convergent and therefore uniformly bounded above.

Summing over the finitely many faces incident to \(p\), we obtain
\[
  l^{(n)} = \sum_{i=1}^m \ell_{X_n}(c_i^{(n)})
\]
as a finite sum of uniformly bounded sequences, hence \((l^{(n)})_n\) is uniformly bounded above.
Combining this with \eqref{eq:ratio_l_L_tend_to_0} yields \(L^{(n)} \to 0\), proving the claim \eqref{eq:shrink.circum.}.

\medskip\noindent
\emph{Final step: contradiction with the collar lemma.}
By Lemma~\ref{lm:Collar-lemma-for-hyperbolic-cone-surfaces}, there exists \(C \in (0,1]\) such that every disk centered at \(p\) with circumference less than \(C\) is disjoint from all simple closed geodesics on \(X_n\).
However, by \eqref{eq:shrink.circum.} there exists \(N>0\) such that for all \(n \ge N\) we have \(L^{(n)} < C\).
Thus, for \(n\ge N\), the circle \(C^{(n)}\) bounds a closed disk of circumference less than \(C\) that contains the point $q^{(n)}$ of a simple closed geodesic $\gamma^{(n)}$, contradicting the collar lemma.
This contradiction shows that our initial assumption was false, and the lemma follows.
\end{proof}

\subsection{Asymptotic limits of peripheral stretch and interior anti-stretch rays}

Recall that Theorem~\ref{thm:asym.behavior} states that for each \(G\in\TSp\) and each \(X_0\in\cTG\), the peripheral stretch ray \((X_t)_{t\geq 0}\) starting from \(X_0\) converges, in the length spectrum of simple closed curves, to the hyperbolic surface with cusps at each marked point \(p_i\in \up\) whose shear coordinates with respect to \(G\) are \(\cS_G(X_0)\), where \(\cS_G:=s_G\circ\cC_G\) (Definition~\ref{def:shear-radius}).

Before entering the technical estimates, we clarify the topology of convergence appearing in
Theorem~\ref{thm:asym.behavior}: it is the simple-curve length-spectrum topology, which is weaker than
the arc--curve length-spectrum topology induced by the full length data on \(\cAC\) (Definition~\ref{def:topology.on.T}).

\begin{remark}[Divergence in the arc--curve length-spectrum topology]
\label{rem:topology.convergence.light}
Let \((X_t)_{t\ge 0}\) be the peripheral stretch ray in \(\cT(G)\) based at \(X_0\).
By Definition~\ref{def:part.stret.}, the peripheral components of \(\cC_G(X_t)\) are scaled by \(e^t\),
hence for each \(p\in\underline p\) we have \(r_p(X_t)=e^t r_p(X_0)\to +\infty\).
Fix any edge \(e\in E(G)\) incident to \(p\).
Since \(\cC_G=(\iota_{G,\cdot})^{-1}\circ \ell^{(G)}_{\cdot}\), we have the identity
\(\ell_{X_t}(e)= i(\cC_G(X_t),e)\).
Moreover, because \(e\) crosses the peripheral annulus at \(p\) at least once,
its intersection with the peripheral component contributes at least \(r_p(X_t)\), so
\[
\ell_{X_t}(e)\ge r_p(X_t)\longrightarrow +\infty .
\]
Therefore \((X_t)_{t \geq 0}\) diverges in the arc--curve length-spectrum topology.
This is why Theorem~\ref{thm:asym.behavior} only asserts convergence after restricting
to the length spectrum of simple closed curves.
\end{remark}

To prove Theorem~\ref{thm:asym.behavior}, we first need an estimate (Lemma~\ref{lm:distance-close} below) on the distance between two points lying on two distinct edges of a hyperbolic triangle. Before stating this estimate, we introduce some notation in the Poincaré disk model \(\bH^2\) of the hyperbolic plane (see Figure~\ref{fig:dist-cts}). 

\begin{itemize}
\item Let \(r>0\), \(x,y\in \bR\) and \(\theta\in(-\pi,\pi)\).  Let \(O\) be the origin of \(\bH^2\) and let \(A\) be the point on the positive half of the \(x\)-axis with \(d_{\bH^2}(O,A)=r\). 
\item Let \(p=p(r,x)\) be the point on the geodesic ray intersecting the \(x\)-axis orthogonally at \(A\) with \(d_{\bH^2}(p,A)=|x|\). We take \(x>0\) (resp.\ \(x<0\)) if \(p\) lies strictly on the left (resp.\ right) of \(\overrightarrow{OA}\), and \(x=0\) if and only if \(p=A\). 
\item Let \(B\) be the point on the hyperbolic circle of radius \(r\) such that the directed segment \(\overrightarrow{OB}\) bounds an angle \(\theta\) with \(\overrightarrow{OA}\) in the anti-clockwise direction. 
\item Let \(q=q(\theta, r, y)\) be the point on the geodesic ray intersecting \(\overrightarrow{OB}\) orthogonally at \(B\) with \(d_{\bH^2}(q,B)=|y|\), where \(y>0\) (resp.\ \(y<0\)) if \(q\) lies on the left (resp.\ right) of \(\overrightarrow{OB}\), and \(y=0\) if and only if \(q=B\).
\end{itemize}

\begin{figure}[htp]
    \centering  \includegraphics[width=8cm]{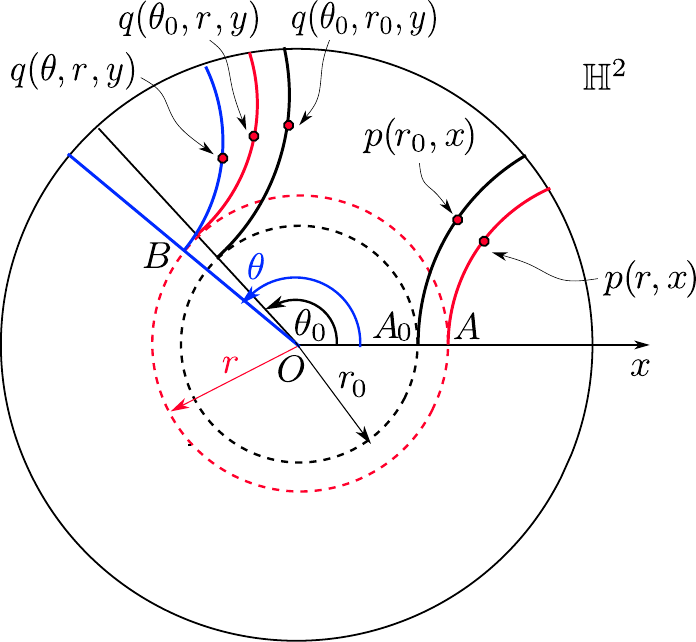}
    \caption{The point \(p(r,x)\) lies on the left of \(\overrightarrow{OA}\) with \(d_{\bH^2}(p(r,x),A)=x>0\), while the point \(q(\theta, r, y)\) lies on the right of \(\overrightarrow{OB}\) with \(d_{\bH^2}(q(\theta,r,y),B)=-y>0\). Here \(d_{\bH^2}(O,A)=d_{\bH^2}(O,B)=r\), and \(\overrightarrow{OB}\) bounds an angle \(\theta>0\) with \(\overrightarrow{OA}\) in the anti-clockwise direction.}
    \label{fig:dist-cts}
\end{figure}

\begin{lemma}\label{lm:distance-close}
Let \(r_0>0\), \(\theta_0\in (-\pi,\pi)\) and \(M>0\). For any \(\epsilon>0\), there exists \(\delta>0\) (depending on \(r_0\), \(\theta_0\), \(M\) and \(\epsilon\)) such that for all \(\theta,r>0\) with \(|\theta-\theta_0|<\delta\), \(|r-r_0|<\delta\) and all \(x,y\in [-M,M]\), we have 
\[
    \bigl|d_{\bH^2}(p(r,x),q(\theta, r, y))-d_{\bH^2}(p(r_0,x),q(\theta_0, r_0, y))\bigr|<\epsilon.
\]
\end{lemma}

\begin{proof}
The map
\[
(r,\theta,x,y)\longmapsto d_{\bH^2}\bigl(p(r,x),q(\theta,r,y)\bigr)
\]
is continuous on \((0,\infty)\times(-\pi,\pi)\times\bR\times\bR\), since \(p(r,x)\) and \(q(\theta,r,y)\) depend continuously on \(r,\theta,x,y\) and the hyperbolic distance is continuous in its arguments.

Fix \(r_0>0\), \(\theta_0\in(-\pi,\pi)\) and \(M>0\). Consider the compact set
\[
K:= [r_0-1,r_0+1]\times[\theta_0-1,\theta_0+1]\times[-M,M]\times[-M,M].
\]
By continuity, the distance function is uniformly continuous on \(K\). Hence for any \(\epsilon>0\) there exists \(\delta>0\) such that whenever \(|r-r_0|<\delta\), \(|\theta-\theta_0|<\delta\) and \(x,y\in[-M,M]\), we have
\[
\bigl|d_{\bH^2}(p(r,x),q(\theta,r,y)) - d_{\bH^2}(p(r_0,x),q(\theta_0,r_0,y))\bigr|<\epsilon,
\]
as required.
\end{proof}

We are now ready to prove Theorem~\ref{thm:asym.behavior}.

\begin{proof}[Proof of Theorem \ref{thm:asym.behavior}]
For each \(G\in\TSp\) and each \(X_0\in\cTG\), let
\[
X_t:=\text{stretch}^{\mathrm{per}}_G(X_0,t),\qquad t\ge 0,
\]
denote the peripheral stretch ray in \(\cTG\) starting from \(X_0\). Let \(X_{\infty}\in\cT_0(S,\up)\) denote the hyperbolic surface with cusps at each marked point \(p_i\in \up\), whose shear coordinates with respect to the triangulation \(G\) are \(\cS_G(X_0)\), where \(\cS_G=s_G\circ\cC_G\). 
It suffices to show that for any simple closed curve class \([\gamma]\in\cS(S,\up)\),
\[
\ell_{X_t}([\gamma])\longrightarrow \ell_{X_{\infty}}([\gamma])
\]
as \(t\rightarrow +\infty\). 

By Lemma~\ref{lm:estimate.triangle}, the cone angle at each marked point of \(X_t\) tends to zero as \(t\to\infty\). Without loss of generality, we may assume that the largest cone angle \(\theta_{\max}(X_t)\) is less than \(\pi\) for all sufficiently large \(t>0\).

\medskip\noindent\emph{Step 1: Notation.}
Fix \([\gamma]\in\cS(S,\up)\). We introduce the following notation:
\begin{itemize}
\item \(\gamma_G : [0, 1] \to \dot{S}\): a representative of \([\gamma]\) which intersects \(G\) transversely, with \(0 \leq t_1 < \dots < t_{m_{[\gamma]}} < 1\) such that each \(\gamma_{G} (t_i)\) lies in an edge (denoted by $e_i$) of \(G\) (note that \(e_i\) may coincide with \(e_j\) for \(i\ne j\)).
\item \(\gamma^{(t)}\) (resp.\ \(\gamma^{(\infty)}\)): the shortest geodesic representative of \([\gamma]\) on \(X_t\) (resp.\ \(X_{\infty}\)).
\item \(e_i^{(t)}\) (resp.\ \(e_i^{(\infty)}\)): the geodesic realization of \(e_i\) on \(X_t\) (resp.\ \(X_{\infty}\)).
\item \(q_i^{(t)}\) (resp.\ \(q_i^{(\infty)}\)): the intersection point of  \(\gamma^{(t)}\) (resp.\ \(\gamma^{(\infty)}\)) with the edge $e_i^{(t)}$ (resp. $e_i^{(\infty)}$). \item \(f_i^{(t)}\) (resp.\ \(f_i^{(\infty)}\)): the geodesic realization on \(X_t\) (resp.\ \(X_{\infty}\)) of the face bounded by \(e_i\) and \(e_{i+1}\).
\item \(\tau_{j,i}^{(t)}\)  (resp.\ \(\tau_{j,i}^{(\infty)}\)): the tangent point on the edge \(e_j^{(t)}\) (resp.\ \(e_j^{(\infty)}\)) to the incircle of the geodesic triangle \(f^{(t)}_i\) (resp.\ \(f^{(\infty)}_i\)), where \(j\in\{i,i+1\}\). 
\end{itemize}

\medskip\noindent\emph{Step 2: Construction of a comparison curve.}
For each \(t\) and \([\gamma]\), we construct a comparison simple closed curve \(\gamma^{(t,\infty)}\) on \(X_{\infty}\), homotopic to \([\gamma]\), as follows. For each \(i\), let \(q_{i}^{(t,\infty)}\in e_i^{(\infty)}\) and \(q_{i+1}^{(t,\infty)}\in e_{i+1}^{(\infty)}\) be the points such that their signed distances from the tangent points \(\tau_{i,i}^{(\infty)}\in e_i^{(\infty)}\), \(\tau_{i+1,i}^{(\infty)}\in e_{i+1}^{(\infty)}\) agree with those of \(q_i^{(t)}\), \(q_{i+1}^{(t)}\) from \(\tau_{i,i}^{(t)}\), \(\tau_{i+1,i}^{(t)}\), respectively (see Figure~\ref{fig:construct-curve}). Thus
\[
\begin{split}
     d_{X_{\infty}}(\tau_{i,i}^{(\infty)}, q_{i}^{(t,\infty)})
    &=d_{X_t}(\tau_{i,i}^{(t)}, q_{i}^{(t)})=:x_t,\\
     d_{X_{\infty}}(\tau_{i+1,i}^{(\infty)}, q_{i+1}^{(t,\infty)})
    &=d_{X_t}(\tau_{i+1,i}^{(t)},q_{i+1}^{(t)})=:y_t.
\end{split}
\]
Let \(\gamma^{(t,\infty)}\) be the union, over all \(i=1,\dots,m_{[\gamma]}\), of the geodesic segments in \(f_i^{(\infty)}\) joining \(q_i^{(t, \infty)}\) to \(q_{i+1}^{(t,\infty)}\).

\begin{figure}[htp]
    \centering  \includegraphics[width=18.5cm]{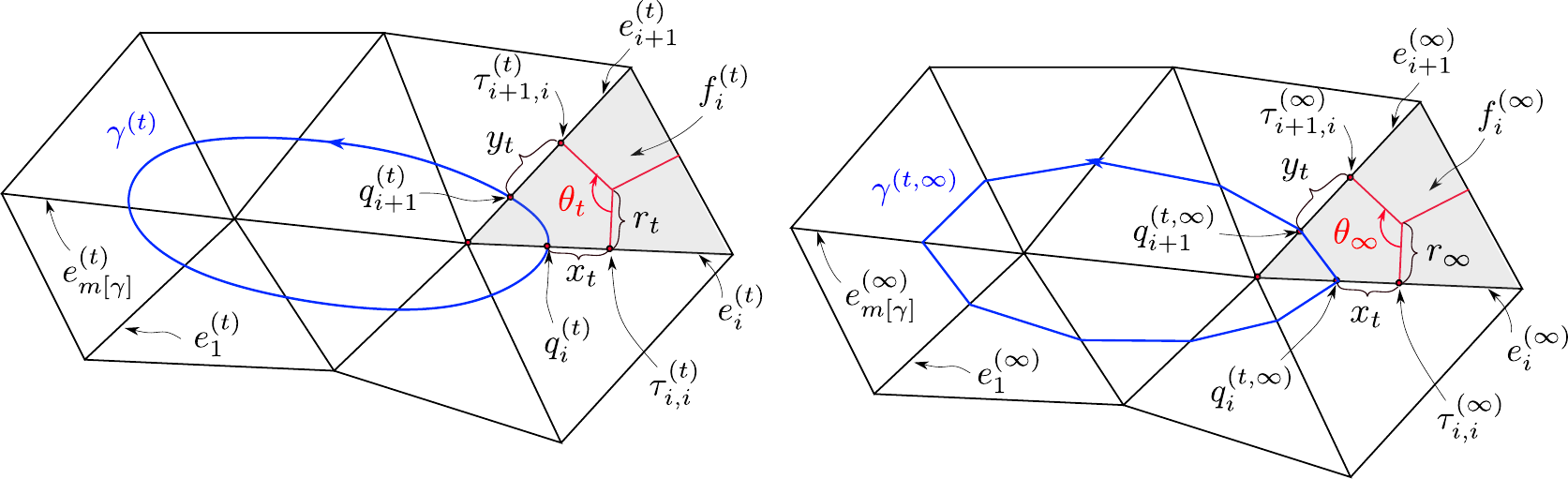}
    \caption{A comparison curve \(\gamma^{(t,\infty)}\) for \(\gamma^{(t)}\).}
    \label{fig:construct-curve}
\end{figure}

\medskip\noindent\emph{Step 3: Putting faces in canonical position.}
We now regard each triangle \(f_i^{(\infty)}\) and \(f_i^{(t)}\) in its canonical position in \(\bH^2\) (Definition~\ref{def:canonical-triangles}), as in Figure~\ref{fig:canonical-triangle}.
By Lemma~\ref{lm:estimate.triangle}, the angle between the two radial segments based at the center of \(f_i^{(\infty)}\) through \(\tau_{i,i}^{(\infty)}\) and \(\tau_{i+1,i}^{(\infty)}\) is
\[
\theta_{\infty}:=\frac{2\pi}{3},
\]
and the radius of the incircle of each ideal triangle \(f_i^{(\infty)}\) is
\[
r_{\infty}:={\rm arcosh}\Bigl(\frac{2\sqrt{3}}{3}\Bigr).
\]

\begin{figure}[htp]
    \centering  \includegraphics[width=17.5cm]{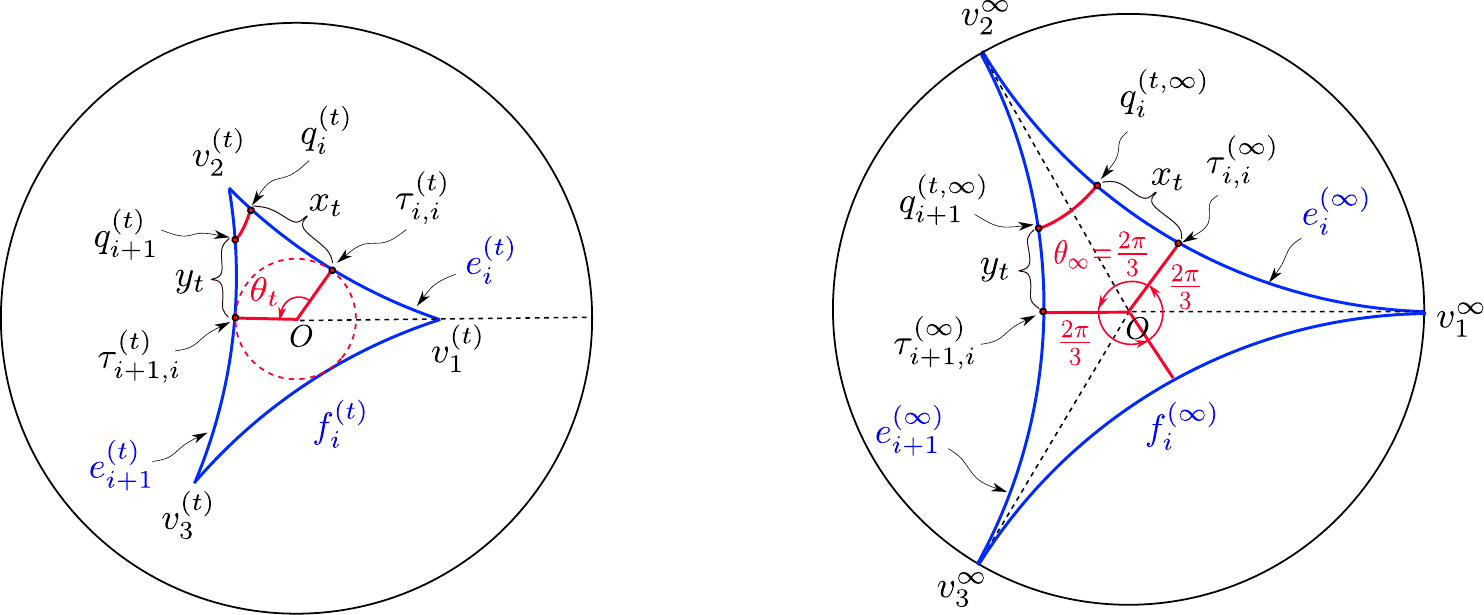}
    \caption{Canonical triangles \(f^{(t)}_i\) and \(f^{(\infty)}_i\).}
    \label{fig:canonical-triangle}
\end{figure}

After this normalization, the points \(q_i^{(t)}\), \(q_{i+1}^{(t)}\) (resp.\ \(q_{i}^{(t,\infty)}\), \(q_{i+1}^{(t,\infty)}\)) in \(f_i^{(t)}\) (resp.\ \(f_i^{(\infty)}\)) can be written in the form
\[
\begin{split}
q_i^{(t)} &= p(r_t,x_t), \qquad q_{i+1}^{(t)} = q(\theta_t, r_t,y_t),\\
q_i^{(t,\infty)} &= p(r_{\infty},x_{t}), \qquad q_{i+1}^{(t,\infty)} = q(\theta_{\infty}, r_{\infty}, y_{t}),
\end{split}
\]
where \(\theta_t\in (-\pi,\pi)\) is the angle from the radial segment based at the center of $f_i^{(t)}$ through \(\tau_{i,i}^{(t)}\) to that through \(\tau_{i+1,i}^{(t)}\), and \(r_t\) is the inradius of the triangle \(f^{(t)}_i\).

\medskip\noindent\emph{Step 4: Convergence.}
By Lemma~\ref{lm:estimate.triangle}, we have \(f_i^{(t)}\to f_i^{(\infty)}\) as \(t\to +\infty\), in canonical position. Thus \(r_t\rightarrow r_{\infty}\) and \(\theta_t\rightarrow \theta_{\infty}\) as \(t\rightarrow +\infty\). By Lemma~\ref{lm:dist-tangent-intersect}, there exists \(D>0\) such that \(|x_t|\leq D\) and  \(|y_t|\leq D\) for all sufficiently large \(t\). 

For any \(\epsilon>0\), applying Lemma~\ref{lm:distance-close} with \(r_0=r_{\infty}\), \(\theta_0=\theta_{\infty}\), \(M=D\) and \([\gamma]\) fixed, there exists \(T=T(\epsilon, m_{[\gamma]})>0\) such that for all \(t\geq T\) and each \(i=1,\dots,m_{[\gamma]}\),
\[
d_{X_{\infty}}(q_i^{(t,\infty)},q_{i+1}^{(t,\infty)})-\frac{\epsilon}{m_{[\gamma]}}
\leq d_{X_t}(q_i^{(t)},q_{i+1}^{(t)})
\leq d_{X_{\infty}}(q_i^{(t,\infty)},q_{i+1}^{(t,\infty)})+\frac{\epsilon}{m_{[\gamma]}}.
\]
Summing over \(i\) yields
\[
\begin{split}
  \ell_{X_t}([\gamma])=\ell_{X_t}(\gamma^{(t)})
  &=\sum_{i=1}^{m_{[\gamma]}}d_{X_t}(q_i^{(t)},q_{i+1}^{(t)})\\
  &\geq \sum_{i=1}^{m_{[\gamma]}}d_{X_{\infty}}(q_i^{(t,\infty)},q_{i+1}^{(t,\infty)})-\epsilon\\
  &\geq \ell_{X_{\infty}}([\gamma])-\epsilon.
\end{split}
\]

For the reverse inequality, start from the geodesic representative \(\gamma^{(\infty)}\) of \([\gamma]\) on \(X_{\infty}\) and define the corresponding distances between the tangent points \(\tau_{i,i}^{(\infty)}\), \(\tau_{i+1,i}^{(\infty)}\) and the intersection points \(q_i^{(\infty)}\), \(q_{i+1}^{(\infty)}\). As before, we define a comparison curve \(\gamma^{(\infty,t)}\) in \(X_t\), intersecting each face \(f_i^{(t)}\) at points \(q_{i}^{(\infty,t)}\), \(q_{i+1}^{(\infty, t)}\) such that
\[
\begin{split}
    d_{X_t}(\tau_{i,i}^{(t)}, q_{i}^{(\infty, t)})&=d_{X_{\infty}}(\tau_{i,i}^{(\infty)}, q_{i}^{(\infty)}),\\
    d_{X_t}(\tau_{i+1,i}^{(t)}, q_{i+1}^{(\infty, t)})&=d_{X_{\infty}}(\tau_{i+1,i}^{(\infty)}, q_{i+1}^{(\infty)}).
\end{split}
\]
Applying the same comparison argument as above yields
\[
  \ell_{X_{\infty}}([\gamma]) \geq \ell_{X_t}([\gamma])-\epsilon
\]
for all sufficiently large \(t\). Combining the two inequalities completes the proof of Theorem~\ref{thm:asym.behavior}.
\end{proof}

\begin{proof}[Proof of Theorem \ref{thm:asym.behavior1}]
Let \(X_{\infty}\) denote the hyperbolic circle-packed surface with centers at \(p_i\in \up\), whose radius coordinates with respect to the triangulation \(G\) are \(\cR_G(X_0)\), where \(\cR_G:=r_G\circ\cC_G\) (Definition~\ref{def:shear-radius}). By the definition of the interior anti-stretch deformation (Definition~\ref{def:part.stret.}), the shear coordinates of \(X_t\) tend to \(0\) as \(t\rightarrow -\infty\), while the radius coordinates of \(X_t\) remain constant for all \(t\leq 0\). 

By Corollary~\ref{cor:homeo.TS.shear.radius}, the shear--radius map \(\cSR_G\) gives a homeomorphism between \(\cT(G)\) and \(\Lambda_G\). The path \((X_t)_{t\leq 0}\) therefore converges to \(X_{\infty}\) as \(t\rightarrow -\infty\) in the shear--radius coordinates, hence also in Teichmüller space. The theorem follows.
\end{proof}

 \section{Maximal cone angles and universal triangulability}
\label{sec:max.cone.angle}

The deformation rays constructed in the previous section vary the cone angles, and
raise a natural extension problem: how far can one increase the angles before leaving
the domain where our triangulation-based coordinates are defined?
In the classical analytic theory, admissible cone-angle vectors (in a fixed conformal
class) are governed by the existence and uniqueness results of Heins, McOwen and
Troyanov for the Berger--Nirenberg problem \cite{Hei62,McO88,Tro91}.
Here we adopt a complementary viewpoint: for a fixed topological type \((g,n)\) and a
triangulation \(G\), we determine the maximal cone angles compatible with \(G\)-triangulability,
and, more generally, with the universally triangulable locus \(\cT_{\rm u.t.}\Sp\).

Fix a marked point \(p_i\in\up\).
Using the circular foliation coordinates of Theorem~\ref{thm:Teich.chart},
we address Question~\ref{q:maximal.angle} from the Introduction: we ask how
large the cone angle \(\theta_{p_i}(X)\) can be
\emph{(a)} inside a fixed chart \(\cT(G)\) for a given triangulation
\(G\in\TSp\), and
\emph{(b)} on the universally triangulable locus
\(\cT_{\rm u.t.}\Sp\), consisting of metrics that geodesically realizing every triangulation.
Both parts admit explicit answers in terms of the numbers of faces incident
to the marked points in geodesic triangulations of \((S,\up)\).

For each triangulation \(G\in\TSp\),
Proposition~\ref{prop:angle.face} computes the maximal cone angle
\(\theta_{p_i}(G)\) at \(p_i\) within the chart \(\cT(G)\).
Passing to the universally triangulable locus, we then consider the universally maximal cone angle (Definition~\ref{def:max.cone.angles}), determine its value (Proposition~\ref{prop:max.cone.angle}), and analyse
when this value is attained (Lemma~\ref{lem:achived}), and treat the
single-cone case \(n=1\), \(g\ge2\) separately
(Proposition~\ref{prop:n1ggeq2.cone.angle}), thereby proving
Theorem~\ref{thm:adm.cone.angle}.

For the reader’s convenience we briefly outline the structure of the section.
\begin{itemize}
  \item In Section~\ref{subsec:max.angle.single.chart} we bound the cone
  angle at a marked point inside a fixed chart \(\cT(G)\) and express
  \(\theta_{p_i}(G)\) in terms of the number of faces incident to \(p_i\)
  (Proposition~\ref{prop:angle.face}).
  \item In Section~\ref{subsec:universal.max.angle} we pass from a single
  chart to the universally triangulable locus and determine the corresponding
  maximal cone angles, obtaining Theorem~\ref{thm:adm.cone.angle} via
  Propositions~\ref{prop:max.cone.angle} and~\ref{prop:angle.face} together
  with Lemma~\ref{lem:achived}.
  \item In Section~\ref{subsec:case3} we deal with the single-cone case
  \(n=1\), \(g\ge2\), using a Dehn twist construction to show that in this case, those hyperbolic cone-metrics with cone angles strictly larger than \(\pi\) cannot be universally triangulable
  (Proposition~\ref{prop:n1ggeq2.cone.angle}).
\end{itemize}

\subsection{Maximal cone angles in a single chart}
\label{subsec:max.angle.single.chart}

We begin with part~\((a)\) of Question~\ref{q:maximal.angle}.
Fix a triangulation \(G\in\TSp\) and a marked point \(p\in\up\).
Our first goal is to understand the maximal cone angle at \(p\) that can be
realized by surfaces \(X\in\cT(G)\), and more generally under the joint
constraints coming from a family of triangulations.

\begin{definition}[Maximal cone angles] 
\label{def:max.cone.angles}
For \(G\in\TSp\) and \(p\in\up\), the \DEF{\em \(G\)-maximal cone angle at \(p\)}
is
\[
\theta_p(G):=\sup_{X\in\cT(G)}\theta_p(X),
\]
where \(\theta_p(X)\) denotes the cone angle at \(p\) of the hyperbolic
cone-surface \(X\).

For a non-empty subset \(\underline{G}\subset\TSp\), the
\DEF{\em \(\underline{G}\)-maximal cone angle at \(p\)} is
\[
\theta_p(\underline{G})
:=\sup_{X\in\cT(\underline{G})}\theta_p(X),
\qquad
\cT(\underline{G}):=\bigcap_{G\in\underline{G}}\cT(G).
\]
In particular, the \DEF{\em universally maximal cone angle at \(p\)} is
\[
\theta_p(\TSp)
:=\sup_{X\in\cT(\TSp)}\theta_p(X)
=\sup_{X\in\cT_{\rm u.t.}\Sp}\theta_p(X),
\]
and the \DEF{\em universally maximal cone angle} of \(\Sp\) is
\[
\Theta\Sp:=\max_{p\in\up}\theta_p(\TSp).
\]
\end{definition}

In order to answer Question~\ref{q:maximal.angle}(a), we now compute
\(\theta_p(G)\) explicitly for each triangulation \(G\in\TSp\).
The next proposition shows that the \(G\)-maximal cone angle at \(p\) is
determined purely by the number of faces of \(G\) incident to~\(p\).

\begin{proposition}[Maximal cone angle in a chart]
\label{prop:angle.face}
Let \(G\in\TSp\) be a triangulation and \(p\in\up\) a marked point.
Then
\[
\theta_p(G)=\pi\,|F_p(G)|,
\]
where \(|F_p(G)|\) is the number of faces of \(G\) incident to \(p\).
\end{proposition}

\begin{remark}
Proposition~\ref{prop:angle.face} shows that, for each triangulation
\(G\in\TSp\), the answer to Question~\ref{q:maximal.angle}(a) is purely
combinatorial: the maximal cone angle at \(p\) within \(\cT(G)\) depends only
on the number \(|F_p(G)|\) of faces of \(G\) incident to \(p\), via the
formula \(\theta_p(G)=\pi\,|F_p(G)|\).
\end{remark}

\begin{proof}
For any \(X\in\cT(G)\), the cone angle at \(p\) is the sum of the interior
angles at \(p\) of the faces in \(F_p(G)\).
Each of these angles is strictly less than \(\pi\) in a hyperbolic triangle,
so
\[
\theta_p(X)<\pi\,|F_p(G)|.
\]
Taking the supremum over all \(X\in\cT(G)\) gives the upper bound
\(\theta_p(G)\leq \pi\,|F_p(G)|\).

To see that this bound is sharp, we construct a sequence of hyperbolic
cone-surfaces \(X_n\in\cT(G)\) for which \(\theta_p(X_n)\to\pi\,|F_p(G)|\).
We do this on the foliation side and then transfer to metrics via the
circular foliation map \(\cC_G\).

By the product decomposition of \(\cMFp\) in~\eqref{map:prod.str.MF}, the
interior component and the peripheral components of a foliation class
\([F]\in\cMFp\) can be chosen independently.
Fix \(n\ge1\) and define a class \([F_n]\in\cMFp\) as follows.

\begin{itemize}
  \item For the peripheral components, set the radius of the component
  around \(p:=p_i\) to be \(r_i([F_n])=1/n\), while for \(p_j\neq p\) we fix
  \(r_j([F_n])=a_n>0\) (arbitrary but independent of the face adjacent to~\(p\)).

  \item For the interior component \([F_n^{(0)}]\), prescribe the edge
  intersection numbers
  \[
  i\big([F^{(0)}_n],e\big)
  =
  \begin{cases}
    1/n,&\text{if \(e\) has both endpoints at \(p\)},\\[0.2em]
    1+1/n,&\text{if \(e\) has exactly one endpoint at \(p\)},\\[0.2em]
    2,&\text{otherwise.}
  \end{cases}
  \]
\end{itemize}

For each face \(f\in F(G)\) with edges \(e_1,e_2,e_3\) and opposite vertices \(v_1,v_2,v_3\), let \(w_l^{(n)}:=i([F^{(0)}_n],e_l)\) and \(r_l^{(n)}:=r_{v_l}([F_n])\). 
Then the total edge intersections can be expressed as
\[
\ell_i^{(n)}:=i([F_n],e_i)=w_i^{(n)}+r_j^{(n)}+r_k^{(n)},\qquad\{i,j,k\}=\{1,2,3\}.
\]
A direct case-by-case check on the possible vertex types of \(f\)
(according to how many vertices are equal to \(p\)) shows that
\(\ell_i^{(n)}+\ell_j^{(n)}-\ell_k^{(n)} = w_i^{(n)}+w_j^{(n)}-w_k^{(n)}+2r_k^{(n)}\) is bounded below by a positive multiple of \(1/n\) in every case (see also below for more details), and hence the strict triangle
inequalities are satisfied for all faces.
Thus the edge data \((\ell_e)_{e\in E(G)}\) lies in \(\Omega_G\).

By Proposition~\ref{prop:edge.length.MF.biject.}, each such choice of edge
weights determines a unique foliation class \([F_n]\in\cMFp\), and by
Theorem~\ref{thm:Teich.chart} (via the circular foliation map \(\cC_G\)) it
corresponds to a unique cone-metric \(X_n\in\cT(G)\) whose edge lengths
agree with the edge intersection numbers of~\([F_n]\).

Now fix a face \(f\in F_p(G)\).  We claim that the angle contributed by \(f\) at \(p\) on \(X_n\)
tends to \(\pi\) as \(n\to\infty\). There are two combinatorial cases.

\smallskip\noindent
\emph{Case 1: \(f\) has exactly one vertex at \(p\).}
Write the vertices of \(f\) as \((p,q,r)\) with \(p\not= q\), \(p\not=r\) and opposite edges as \(qr\), \(rp\) and \(pq\). Let \(w^{(n)}_e:=i([F^{(0)}_n],e)\) and \(r^{(n)}_v:=r_v([F_n])\).
By construction,
\[
\ell_{qr}^{(n)}=w^{(n)}_{qr}+r^{(n)}_q+r^{(n)}_r,\qquad
\ell_{rp}^{(n)}=w^{(n)}_{rp}+r^{(n)}_r+r^{(n)}_p,\qquad
\ell_{pq}^{(n)}=w^{(n)}_{pq}+r^{(n)}_p+r^{(n)}_q,
\]
where \(w^{(n)}_{pq}=w^{(n)}_{rp}=1+1/n\), \(w^{(n)}_{qr}=2\), and \(r^{(n)}_p=1/n\), \(r^{(n)}_q=r^{(n)}_r=a_n\).
Hence
\[
\ell_{pq}^{(n)}+\ell_{rp}^{(n)}-\ell_{qr}^{(n)} = \frac{4}{n}>0,
\]
so the triangle \(f\) becomes asymptotically degenerate with
\(\ell_{qr}^{(n)}\to \ell_{pq}^{(\infty)}+\ell_{rp}^{(\infty)}\).
By the hyperbolic law of cosines, this implies that the interior angle of \(f\) at \(p\)
tends to \(\pi\) from below, while the other two angles tend to \(0\).

\smallskip\noindent
\emph{Case 2: \(f\) has an edge with both endpoints at \(p\).}
Then \(f\) is a self-folded face with vertices \((p,p,q)\).
Let \(e\) be the edge with both endpoints at \(p\), and let \(e'\) be the edge connecting \(p\) to \(q\)
(which appears twice in the boundary of \(f\)).
By construction, \(\ell_{X_n}(e)=w^{(n)}_e+2r^{(n)}_p=3/n\to0\) as \(n\to\infty\), whereas \(\ell_{X_n}(e')=a_n+1+2/n\) stays bounded away from \(0\).
Thus \(f\) is an isosceles triangle with a base tending to \(0\), so the angle at \(q\) tends to \(0\),
and the two base angles tend to \(\pi/2\). Therefore the total angle contributed by \(f\) at \(p\)
(the sum of these two base angles) tends to \(\pi\) from below.

\smallskip
In either case, the angle contribution of \(f\) at \(p\) tends to \(\pi\).
Since this holds for each \(f\in F_p(G)\), we have
\[
\lim_{n\to\infty}\theta_p(X_n)=\pi\,|F_p(G)|.
\]

Combining this with the upper bound proves that
\(\theta_p(G)=\pi\,|F_p(G)|\), as claimed.
\end{proof}

The following is an immediate consequence of
Proposition~\ref{prop:angle.face}.

\begin{corollary}
\label{cor:angle.supremum}
Let \(\underline{G}\subset\TSp\) be non-empty and \(p\in\up\).
Then
\[
\theta_p(\underline{G})
\leq \pi\,\min_{G\in\underline{G}}|F_p(G)|.
\]
\end{corollary}

\begin{proof}
By definition,
\[
\theta_p(\underline{G})
=\sup_{X\in\cT(\underline{G})}\theta_p(X)
=\sup_{X\in\bigcap_{G\in\underline{G}}\cT(G)}\theta_p(X)
\leq \sup_{X\in\cT(G)}\theta_p(X)
=\theta_p(G)
\]
for every \(G\in\underline{G}\).
Taking the minimum over \(G\in\underline{G}\) and applying
Proposition~\ref{prop:angle.face} yields
\[
\theta_p(\underline{G})
\leq \min_{G\in\underline{G}}\theta_p(G)
=\pi\,\min_{G\in\underline{G}}|F_p(G)|,
\]
as required.
\end{proof}

For Question~\ref{q:maximal.angle}(b), the case \(\underline{G}=\TSp\) in
Corollary~\ref{cor:angle.supremum} provides a universal upper bound on the
cone angle at \(p\) along the universally triangulable locus
\(\cT_{\rm u.t.}(S,\up)\).
The rest of the section is devoted to showing that this bound is sharp and
to analysing when it can be attained.

 \subsection{Universally maximal cone angles}
\label{subsec:universal.max.angle}

In this subsection we address (b) of Question~\ref{q:maximal.angle}, namely the maximal cone angles on the universally triangulable locus \(\cT_{\rm u.t.}(S,\up)\).
The computation of the universally maximal cone angle \(\Theta(S,\up)\) combines the combinatorial upper bound from Corollary~\ref{cor:angle.supremum} with a case-by-case analysis of when this bound is sharp and whether it can be attained.
As a preparation, we first describe the universally triangulable locus in the once-marked torus case \((g,n)=(1,1)\).

\begin{lemma}
\label{lem:n1g1.adm.eq.TS}
For \(g=n=1\), we have \(\cTp=\cT_{\rm u.t.}(S,\up)\).
\end{lemma}

\begin{proof}
Let \(\iota\) denote the hyperelliptic involution of the torus \(S\) with one marked point \(p_1\) (as shown in Figure~\ref{fig:hypellip_involution}).
For each \(X\in\cTp\), the map \(\iota:X\to X\) is biholomorphic; applying the Schwarz–Pick lemma to both \(\iota\) and \(\iota^{-1}\) shows that \(\iota\) is an isometry of the hyperbolic cone metric on \(X\), hence it sends geodesics to geodesics.

In the once-marked torus case, the mapping class of \(\iota\) coincides with element \(-I\in\mathrm{SL}(2,\mathbb{Z})\cong\mathrm{Mod}(S_{1,1})\).
In particular, \(\iota\) reverses the orientation of each oriented simple arc based at \(p_1\) and preserves its underlying unoriented isotopy class.
Therefore \(\iota\) preserves the underlying (unoriented) geodesic edge set of any geodesic triangulation of \(X\).

\begin{figure}[htp]
    \centering
    \includegraphics[width=6cm]{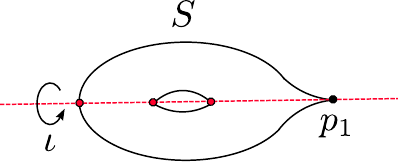}
    \caption{The hyperelliptic involution \(\iota\) of the torus \(S\) with one marked point \(p_1\) (a rotation of angle \(\pi\) around the dashed axis, which passes through four points of \(S\), one of which is \(p_1\)).}
    \label{fig:hypellip_involution}
\end{figure}

Let \(G\in\TSp\).
Then \(G\) is determined by three mutually distinct homotopy classes of simple arcs \(\alpha,\beta,\gamma\in\cA(S,\up)\), all based at \(p_1\).
These three arcs decompose \(S\) into two triangles, say \(\Delta_1\) and \(\Delta_2\), whose common edges are exactly the arcs \(\alpha,\beta,\gamma\).
By the discussion above, \(\iota\) sends the geodesic realization \(\Delta_1(X)\) of \(\Delta_1\) to the geodesic realization \(\Delta_2(X)\) of \(\Delta_2\), and sends the edges of \(\Delta_1(X)\) corresponding to \(\alpha,\beta,\gamma\) to the corresponding edges of \(\Delta_2(X)\) with opposite orientations.
Consequently, the interior angles of \(\Delta_1(X)\) and \(\Delta_2(X)\) satisfy
\[
\theta_{1,\alpha}=\theta_{2,\alpha},\quad
\theta_{1,\beta}=\theta_{2,\beta},\quad
\theta_{1,\gamma}=\theta_{2,\gamma},
\]
where \(\theta_{i,\alpha}\) (\(i=1,2\)) denotes the interior angle of \(\Delta_i(X)\) opposite to the edge coming from cutting along \(\alpha\), and similarly for \(\theta_{i,\beta}\) and \(\theta_{i,\gamma}\); see Figure~\ref{fig:torus_angle_eq}.

\begin{figure}[htp]
    \centering
    \includegraphics[width=7cm]{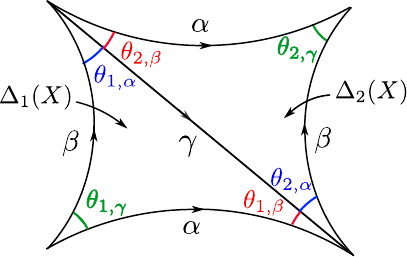}
    \caption{The interior angles of the geodesic triangles \(\Delta_i(X)\) for \(i=1,2\).}
    \label{fig:torus_angle_eq}
\end{figure}

Let \(A_i\) be the sum of the three interior angles of \(\Delta_i(X)\) for \(i=1,2\).
From the equalities above we have \(A_1=A_2\).
On the other hand, by construction the cone angle \(\theta_1(X)\) at \(p_1\) is the sum of the angle sums of \(\Delta_1(X)\) and \(\Delta_2(X)\), namely
\[
\theta_1(X)=A_1+A_2=2A_1.
\]
By the Gauss–Bonnet formula \eqref{eq:Gauss-Bonnet condition}, we have \(\theta_1(X)<2\pi\), so \(A_1=A_2<\pi\).
Thus both \(\Delta_1(X)\) and \(\Delta_2(X)\) are non-degenerate hyperbolic triangles, and $G$ has a geodesic realization on $X$.
Hence \(X\in\cT(G)\) for every \(G\in\TSp\), and therefore
\[
X\in\bigcap_{G\in\TSp}\cT(G)=\cT_{\rm u.t.}(S,\up).
\]

Since \(\cT_{\rm u.t.}\Sp\subset\cTp\) by definition, we conclude that \(\cTp=\cT_{\rm u.t.}\Sp\), as required.
\end{proof}

We can now compute the universally maximal cone angle \(\Theta(S,\up)\) appearing in Question~\ref{q:maximal.angle}(b).

\begin{proposition}
\label{prop:max.cone.angle}
If \(n=g=1\), then \(\Theta\Sp = 2\pi\). Otherwise, \(\Theta\Sp = \pi\).
\end{proposition}

\begin{proof}
Applying Corollary~\ref{cor:angle.supremum} with \(\underline{G} = \TSp\) gives, for every marked point \(p \in \up\),
\begin{equation}\label{eq:upbd.adm.max.cone.angle}
\theta_p(\TSp)\leq \pi\min_{G\in\TSp}\{|F_p(G)|\}.
\end{equation}

To show that this upper bound is sharp, fix a triangulation \(G\in\TSp\) and a marked point \(p\in\up\).
Let \((X_n)\) be the sequence in \(\cT(G)\) constructed in the proof of Proposition~\ref{prop:angle.face}, now with the peripheral radii \(a_n\) at points \(p_j\neq p\) chosen to be \(a_n = n\).
As \(n\to\infty\), the construction yields
\begin{equation}\label{eq:cone.angle.at.p}
\theta_p(X_n)\rightarrow\pi|F_p(G)|~\text{(from below),}
\qquad 
\theta_{p_j}(X_n)\rightarrow 0
\quad\text{for all } p_j\neq p.
\end{equation}

Assume for the moment that the sequence \((X_n)\) lies in the universally triangulable locus \(\cT_{\rm u.t.}\Sp\).
Then by the definition of \(\theta_p(\TSp)\) and \eqref{eq:cone.angle.at.p} we obtain
\[
\theta_p(\TSp)
=\sup_{X\in\cT_{\rm u.t.}\Sp}\theta_p(X)
\geq \sup_{n\in\bN^+}\theta_p(X_n)
=\pi|F_p(G)|
\geq\pi\min_{G\in\TSp}\{|F_p(G)|\}.
\]
Combined with the upper bound \eqref{eq:upbd.adm.max.cone.angle}, this shows that
\begin{equation}\label{eq:adm.max.cone.angle}
\theta_p(\TSp)
=\pi\min_{G\in\TSp}\{|F_p(G)|\}.
\end{equation}

We now analyze \(\Theta\Sp = \max_{p\in\up}\theta_p(\TSp)\) in three cases.

\medskip\noindent\textbf{Case 1: \(n\geq 2\).}
Here \(\up\) contains at least two distinct points.
Fix \(p:=p_i\in\up\) and choose an oriented arc \(\alpha\) connecting \(p\) to some \(p_j\neq p_i\).
Let \(\alpha'\) be an oriented loop based at \(p_j\) that is homotopic to \(\alpha^{-1}\cdot\alpha\).
Then \(\alpha,\alpha'\), and \(\alpha^{-1}\) bound a topological triangle \(\Delta\) (see Figure~\ref{fig:Delta}).
Choose a triangulation \(G_0\in\TSp\) that contains the interior of \(\Delta\) as a face.
By construction, \(|F_p(G_0)|=1\).

Taking \(G=G_0\) in the construction above, \eqref{eq:cone.angle.at.p} gives a sequence \((X_n)\subset\cT(G_0)\) with
\(\theta_p(X_n)\) tending to \(\pi\) from below and \(\theta_{p_j}(X_n)\to0\) for all \(p_j\neq p\).
Thus \(X_n\in\cT_{(0,\pi)}(S,\up)\) for all \(n\), and hence \(X_n\in\cT_{\rm u.t.}\Sp\) by Lemma~\ref{lem:adm_universal_triang}.
Together with \eqref{eq:adm.max.cone.angle} and \(\min_{G\in\TSp}|F_p(G)| = 1\), we obtain
\[
\theta_p(\TSp)=\pi
\qquad\text{for every }p\in\up.
\]
Therefore,
\[
\Theta\Sp
=\max_{p\in\up}\theta_p(\TSp)
=\pi.
\]

\begin{figure}[htp]
    \centering
\includegraphics[width=4cm]{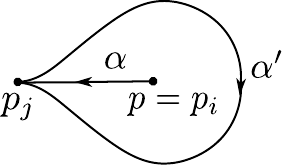}
    \caption{The construction of a triangle with a vertex equal to \(p\).}
\label{fig:Delta}
\end{figure}

\medskip\noindent\textbf{Case 2: \(n=g=1\).}
In this case every triangulation \(G\in\TSp\) has three edges and decomposes \(S\) into two triangles, each with all vertices at the single marked point \(p_1\).
By Lemma~\ref{lem:n1g1.adm.eq.TS} we have \(\cTp = \cT_{\rm u.t.}\Sp\), so the sequence \((X_n)\) constructed in \eqref{eq:cone.angle.at.p} is universally triangulable.
Applying \eqref{eq:adm.max.cone.angle} gives
\begin{equation}
\label{eq:n1g1.leq}
\theta_{p_1}(\TSp)
= \pi\min_{G\in\TSp}\{|F_{p_1}(G)|\}
= 2\pi,
\end{equation}
and hence
\[
\Theta\Sp
=\max_{p\in\up}\theta_p(\TSp)
=2\pi.
\]

\medskip\noindent\textbf{Case 3: \(n=1, g\geq 2\).}
Here there is a single marked point \(p_1\).
Proposition~\ref{prop:n1ggeq2.cone.angle} (proved in Subsection~\ref{subsec:case3}) states that for any
\(X\in\cT_{\rm u.t.}\Sp\) one has \(\theta_{p_1}(X)\leq\pi\), so
\begin{equation}
\label{eq:n1g.geq2.leq}
\Theta\Sp
=\sup_{X\in\cT_{\rm u.t.}\Sp}\theta_{p_1}(X)
\leq \pi.
\end{equation}

We now show that this upper bound is attained by some \(X'\in\cTp\) with cone angle \(\theta_{p_1}(X') = \pi\).
The Gauss--Bonnet constraint~\eqref{eq:Gauss-Bonnet condition} implies that hyperbolic cone-metrics with a single cone point of angle \(\pi\) exist on \((S,\up)\) (see \cite{McO88,Tro91}), so such \(X'\) indeed exists.
For any triangulation \(G\in\TSp\), the geodesic realization \(G_{X'}\) decomposes \((S,\up)\) into \(4g-4+2n = 4g-2 \geq 6\) geodesic triangles.
Since the total cone angle at \(p_1\) is \(\pi\), each of these triangles has angle sum \(<\pi\), and in particular the angle at each vertex is \(<\pi\).
Thus every triangle of \(G_{X'}\) is non-degenerate, so \(X'\in\cT(G)\) for every \(G\in\TSp\), i.e. \(X'\in\cT_{\rm u.t.}\Sp\).
Hence
\begin{equation}
\label{eq:n1g.geq2.geq}
\Theta\Sp
=\sup_{X\in\cT_{\rm u.t.}\Sp}\theta_{p_1}(X)
\geq\theta_{p_1}(X')
=\pi.
\end{equation}
Combining \eqref{eq:n1g.geq2.leq} and \eqref{eq:n1g.geq2.geq} gives \(\Theta\Sp = \pi\) in this case.

The three cases together complete the proof of Proposition~\ref{prop:max.cone.angle}.
\end{proof}

Beyond the value of \(\Theta(S,\up)\), Question~\ref{q:maximal.angle}(b) also asks whether this maximal cone angle is realized by some universally triangulable surface.
The next lemma records the attainability.

\begin{lemma}
\label{lem:achived}
The universally maximal cone angle \(\Theta\Sp\) is attained by a surface in \(\cT_{\rm u.t.}\Sp\) if and only if \(n=1\) and \(g\geq 2\).
\end{lemma}

\begin{proof}
We again distinguish three cases.

\medskip\noindent\textbf{Case 1: \(n\geq 2\).}
Here \(\Theta\Sp=\pi\) by Proposition~\ref{prop:max.cone.angle}.
We claim that this value cannot be achieved at any marked point.
Suppose otherwise that there exist \(p\in\up\) and \(X_0\in\cT_{\rm u.t.}\Sp\) with \(\theta_p(X_0)=\pi\).
Let \(G\in\TSp\) be the triangulation constructed in Case~1 of the proof of Proposition~\ref{prop:max.cone.angle} (see Figure~\ref{fig:Delta}), so that \(p\) is contained in exactly one face of \(G\).
Then the geodesic triangle of \(G_{X_0}\) containing \(p\) has interior angle \(\pi\) at \(p\) and hence degenerates to a geodesic segment.
Thus \(X_0\notin\cT(G)\), contradicting the assumption that \(X_0\in\cT_{\rm u.t.}\Sp\).

\medskip\noindent\textbf{Case 2: \(n=g=1\).}
In this case Proposition~\ref{prop:max.cone.angle} gives \(\Theta\Sp=2\pi\).
The Gauss--Bonnet condition~\eqref{eq:Gauss-Bonnet condition} implies that no hyperbolic cone-metric on a once-marked torus can have cone angle \(2\pi\) at the marked point, so \(\Theta\Sp\) is not attained.

\medskip\noindent\textbf{Case 3: \(n=1, g\geq 2\).}
As shown in Case~3 of the proof of Proposition~\ref{prop:max.cone.angle}, there exists \(X'\in\cT_{\rm u.t.}\Sp\) with \(\theta_{p_1}(X')=\pi\).
Thus \(\Theta\Sp=\pi\) is attained in this case.

The three cases together show that \(\Theta\Sp\) is attained if and only if \(n=1\) and \(g\geq 2\).
\end{proof}

We now combine Proposition~\ref{prop:angle.face} with the description of the universally triangulable locus to complete the proof of Theorem~\ref{thm:adm.cone.angle}.

\begin{proof}[Proof of Theorem \ref{thm:adm.cone.angle}]
Combining Lemma~\ref{lem:n1g1.adm.eq.TS}, Proposition~\ref{prop:max.cone.angle} and Lemma~\ref{lem:achived} yields both the value of \(\Theta(S,\up)\) and its attainability in all topological types.
This completes the proof of Theorem~\ref{thm:adm.cone.angle} and answers Question~\ref{q:maximal.angle}(b), modulo the special case \(n=1, g\ge 2\) which will be treated in Subsection~\ref{subsec:case3}.
\end{proof}

Finally, Corollary~\ref{cor:once.punctured.torus.global} follows directly from Theorem~\ref{thm:adm.cone.angle} together with the shear--radius coordinates of Corollary~\ref{cor:homeo.TS.shear.radius}.

\begin{proof}[Proof of Corollary~\ref{cor:once.punctured.torus.global}]
Assume that \(g = n = 1\).
By Theorem~\ref{thm:adm.cone.angle} we have \(\Theta\Sp = 2\pi\) and \(\cTp = \cT_{\rm u.t.}\Sp\).
By definition, \(\cT_{\rm u.t.}\Sp = \bigcap_{G \in \TSp} \cT(G) \subset \cTp\), so the equality \(\cTp = \cT_{\rm u.t.}\Sp\) implies that \(\cT(G) = \cTp\) for every triangulation \(G \in \TSp\).
On the other hand, Corollary~\ref{cor:homeo.TS.shear.radius} states that, for each \(G\), the shear--radius map
\(\cSR_G : \cT(G) \longrightarrow \Lambda_G\) is a homeomorphism onto \(\Lambda_G\).
Combining these two facts shows that, for every \(G \in \TSp\),
\(\cSR_G : \cTp \to \Lambda_G\) is a homeomorphism, which is exactly the content of Corollary~\ref{cor:once.punctured.torus.global}.
\end{proof}

 \subsection{The case of a single cone point and $g \geq 2$}
\label{subsec:case3}
In this subsection we assume \(n=1\) and \(g\ge 2\).
According to Proposition~\ref{prop:max.cone.angle} and
Lemma~\ref{lem:achived}, this is the only topological type in which the
universally maximal cone angle is achieved.
To complete the proof of Theorem~\ref{thm:adm.cone.angle} and
Question~\ref{q:maximal.angle}(b), it remains to prove the following
upper bound.

\begin{proposition}\label{prop:n1ggeq2.cone.angle}
  Let \(n=1\) and \(g\geq 2\). For any \(X\in\cT_{\rm u.t.}\Sp\), the cone angle at \(p_1\) is at most \(\pi\).  
\end{proposition}

To see this, we need the following preparations.

\begin{lemma}[Existence of curves avoiding cone singularity]
\label{lem:dist.bt.zero}
Let \(X\in\cT\Sp\) with \(\pi<\theta_{p_1}:=\theta_{p_1}(X)<2\pi\).
Then there exists a separating curve class \([\gamma]\in\cS\Sp\) such that
\[
d_X(p_1,\gamma_X)>0,
\]
where \(\gamma_X\) is the length-minimizing representative of \([\gamma]\) for \(X\).
\end{lemma}

\begin{proof}
Suppose for contradiction that \(d_X(p_1,\gamma_X)=0\) for every separating
curve class \([\gamma]\in\cS\Sp\).

Consider two separating curve classes \([\gamma]\) and \([\gamma']\) such that they
bound an annulus \(A\) containing \(p_1\) in its interior.
By our assumption, both geodesic representatives \(\gamma_X\) and \(\gamma'_X\)
pass through \(p_1\).

For any separating curve \(\eta\), the complement \(X \setminus \eta\) consists of two connected components. 
Exactly one of these components contains \(p_1\).
We define the \DEF{exterior component}, denoted by \(X_{\mathrm{ext}}(\eta)\), to be the component of \(X \setminus \eta\) that does not contain \(p_1\). 
If the length-minimizing geodesic representative \(\eta_X\) of \([\eta]\) passes through \(p_1\), we define the \DEF{exterior angle}, \(\theta_{\mathrm{ext}}(\eta)\),
to be the angle of the sector at \(p_1\) subtended by \(X_{\mathrm{ext}}(\eta)\). 

If \(\theta_{\mathrm{ext}}(\eta) < \pi\), one can deform the curve \(\eta\) slightly into the region \(X_{\mathrm{ext}}(\eta)\) (rounding the corner). 
This deformation is a valid isotopy because it moves the curve away from \(p_1\) without crossing it,
and it strictly decreases the length.
Therefore, for \(\gamma_X\) and \(\gamma'_X\) to be length-minimizing, their exterior angles must satisfy:
\[
\theta_{\mathrm{ext}}(\gamma) \ge \pi \quad \text{and} \quad \theta_{\mathrm{ext}}(\gamma') \ge \pi.
\]

Since \([\gamma]\) and \([\gamma']\) bound an annulus containing \(p_1\), their exterior components \(X_{\mathrm{ext}}(\gamma)\) and \(X_{\mathrm{ext}}(\gamma')\) are disjoint subsurfaces of \(X\). 
Consequently, the sectors at \(p_1\) corresponding to these components are disjoint. 
The total cone angle \(\theta_{p_1}\) is the sum of the angles of all sectors at \(p_1\), so we must have:
\[
\theta_{p_1} \ge \theta_{\mathrm{ext}}(\gamma) + \theta_{\mathrm{ext}}(\gamma').
\]
Combining the inequalities yields:
\[
\theta_{p_1} \ge \pi + \pi = 2\pi.
\]
This contradicts the assumption that \(\theta_{p_1} < 2\pi\).
The lemma follows. 
\end{proof}

\begin{construction}[Dehn–twisted triangles and associated angles]
\label{constr:triang.degen.}
Let \(S\) be an oriented closed surface of genus \(g\ge 2\) with a single
marked point \(p_1\) and let \(X\in\cT\Sp\) with
\(\pi<\theta_{p_1}(X)<2\pi\).
By Lemma~\ref{lem:dist.bt.zero} there exists a separating curve class
\([\gamma]\in\cS\Sp\) such that the geodesic representative \(\gamma_X\) of
\([\gamma]\) has positive distance from \(p_1\):
\[
d_X(p_1,\gamma_X)>0.
\]
The curve \(\gamma_X\) separates \(X\) into two components \(C_1\) and \(C_2\);
we assume that \(C_1\) contains \(p_1\) and that \(p_1\) lies on the left-hand
side of \(\gamma_X\) (with respect to the chosen orientation of \(\gamma_X\)).
Since \(\gamma_X\) is essential and \(g\ge 2\), the component \(C_2\) contains
at least one handle.

Choose two based loops \(\alpha_1,\alpha_2\) on \(\Sp\) with both endpoints at
\(p_1\) whose homotopy classes generate the fundamental group of a handle in
\(C_2\), and set
\[
\alpha_3 := \alpha_2^{-1}\alpha_1^{-1}.
\]
The three arcs \(\alpha_1,\alpha_2,\alpha_3\) bound a topological triangle
\(\Delta\subset S\), as illustrated in Figure~\ref{fig:contrad.triangle}.
For each \(m\in\mathbb N\), let
\[
\alpha_i^{(m)} := {\rm Twist}^m_{\gamma}(\alpha_i),
\qquad i=1,2,3,
\]
be the \(m\)-times right Dehn twist of \(\alpha_i\) along \(\gamma\).
The curves \(\alpha_1^{(m)},\alpha_2^{(m)},\alpha_3^{(m)}\) again bound a
triangle, denoted \(\Delta^{(m)}\).
We write \(\Delta^{(m)}(X)\) for the geodesic representative of
\(\Delta^{(m)}\) with respect to \(X\) and denote by \(\Theta^{(m)}(X)\) the
interior angle of \(\Delta^{(m)}(X)\) opposite to the geodesic representative
\(\alpha_3^{(m)}(X)\).

To analyse the angles at \(p_1\) we introduce the following notation
(see Figure~\ref{fig:alpha1}):
\begin{itemize}
  \item Let \(\eta_X\) be the unique shortest geodesic arc on \(X\) connecting
  \(p_1\) to \(\gamma_X\); it is perpendicular to \(\gamma_X\) at a point
  \(q\), and its length is
  \[
  d_X(p_1,q)=d_X(p_1,\gamma_X).
  \]
  \item For each \(i\in\{1,2,3\}\), write \(\alpha_i(X)\) for the geodesic
  representative of \(\alpha_i\) on \(X\).
  Let \(\alpha_{i,j}(X)\) be the \(j\)-th connected component of
  \(\alpha_i(X)\setminus\gamma_X\) that lies on $C_2$ (which lies on the right-hand side of $\gamma_X$),
  ordered along the orientation of \(\alpha_i(X)\), for
  \(j=1,\dots,k_i\), where \(k_i\) is the number of such components.
  Let \(\alpha_{i,j}'(X)\) be the shortest geodesic arc connecting
  \(\gamma_X\) to itself in the homotopy class of \(\alpha_{i,j}(X)\); we
  denote its endpoints by \(q_{i,j,+}\) and \(q_{i,j,-}\) along the
  orientation of \(\alpha_{i,j}'(X)\) inherited from \(\alpha_{i,j}(X)\).
  By construction (see Figure~\ref{fig:contrad.triangle}), we have
  \[
  k_1=k_2=1, \qquad k_3=2.
  \]
  \item For each \(i\in\{1,2,3\}\) and \(m\in\mathbb N\), let
  \(\zeta^{(m)}_{i,-}\) (resp.\ \(\zeta^{(m)}_{i,+}\)) be the angle at
  \(p_1\) (chosen in \((0,\pi)\)) bounded by \(\eta_X\) and the sub-arc of
  \(\alpha_i^{(m)}(X)\) with one endpoint at \(p_1\) and with orientation pointing
  outwards \(p_1\) (resp.\ towards \(p_1\)).
\end{itemize}
\end{construction}

\begin{figure}[htp]
    \centering
    \includegraphics[width=11cm]{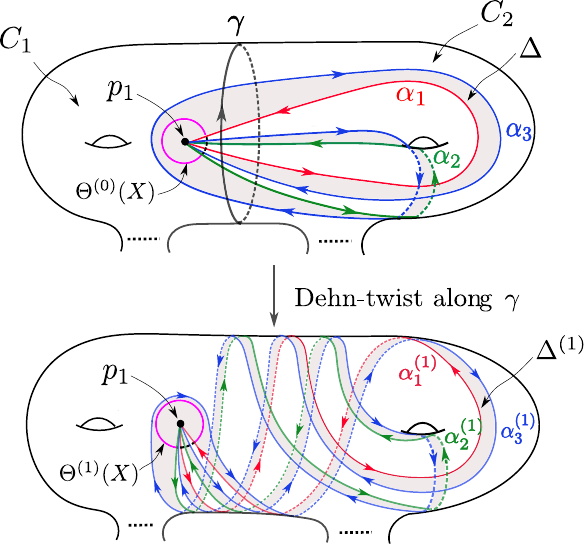}
    \caption{Examples of the curves \(\alpha_1^{(1)}, \alpha_2^{(1)}\) and \(\alpha_3^{(1)}\) constructed on \((S,p_1)\).}
    \label{fig:contrad.triangle}
\end{figure}

\begin{figure}[htp]
    \centering
    \includegraphics[width=10cm]{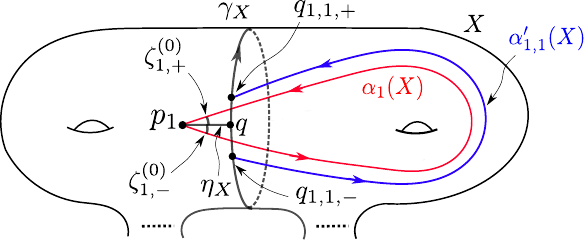}
    \caption{Example of \(\alpha_i(X)\) and \(\alpha'_i(X)\) for \(i=1\) (with \(d_X(p_1,\gamma_X)>0\)).}
    \label{fig:alpha1}
\end{figure}

The following lemma is classical in the smooth hyperbolic setting (angle of parallelism).
We include a short proof both for completeness and to fix notation.

\begin{lemma}[Angle of parallelism limit]
\label{lem:limit.angle.const}
Keep the setup and notation of Construction~\ref{constr:triang.degen.}.
Then for each \(i=1,2,3\), both angles \(\zeta^{(m)}_{i,+}\) and
\(\zeta^{(m)}_{i,-}\) tend to the same constant
\[
\zeta_{\infty}
:= \arcsin\bigl(1/\cosh d_X(p_1,\gamma_X)\bigr)
\]
as \(m\to\infty\).
\end{lemma}

\begin{proof}
Let \((\tilde{S},\tilde{\underline{p}})\) be the universal cover of \((S,\up)\),
endowed with the lifted hyperbolic cone metric from $X$. We denote this
metric simply by \(\tilde{X}\).
The geodesic \(\gamma_X\) lifts to a family of complete geodesics in \(\tilde{S}\);
choose one lift \(\tilde{\gamma}\) and a lift \(\tilde{p}_1\) of \(p_1\) such
that the perpendicular from \(\tilde{p}_1\) to \(\tilde{\gamma}\) has length
\(d:=d_X(p_1,\gamma_X)>0\), and denote its foot by \(\tilde{q}\).
We may isometrically identify a neighbourhood of \(\tilde{\gamma}\) in
\((\tilde{S},\tilde{X})\) with a neighbourhood of a geodesic in the upper
half-plane model \(\bH^2\), so that \(\tilde{\gamma}\) becomes the positive
\(y\)-axis and the segment \([\tilde{p}_1,\tilde{q}]\) is orthogonal to it and
has hyperbolic length \(d\), see Figure~\ref{fig:angle.const}.

For each \(i\) and \(j\), let \(\tilde{q}^{(m)}_{i,j,-}\) be the lift of the
point \(q_{i,j,-}\) obtained by following the deck transformation associated
to \(\gamma^{-m}\) from \(\tilde{q}\).
Then
\begin{equation}
\label{eq:twist.forward}
d_{\bH^2}(\tilde{q}, \tilde{q}^{(m)}_{i,j,-})
= m\cdot\ell_X(\gamma)+d_X(q,q_{i,j,-}).
\end{equation}
Similarly, let \(\tilde{\gamma}^{(m)}_{i,j}\) be the lift of \(\gamma_X\)
obtained by following the lift of \(\alpha_i^{(m)}(X)\) through the first
\(j\) components on the right-hand side of \(\gamma_X\), and denote by
\(\tilde{q}^{(m)}_{i,j,+}\) the endpoint on  \(\tilde{\gamma}^{(m)}_{i,j}\) of the common perpendicular between
\(\tilde{\gamma}^{(m)}_{i,j-1}\) and \(\tilde{\gamma}^{(m)}_{i,j}\) (where \(\tilde{\gamma}^{(m)}_{i,0}:=\tilde{\gamma}\)).
Finally, let \(\tilde{q}^{(m)}_i\) be the point on \(\tilde{\gamma}^{(m)}_{i,k_i}\)
obtained by following the lift of \(\gamma_X\) for \(m\) turns from
\(q_{i,k_i,-}\) back to \(q\).
Then
\begin{equation}
\label{eq:twist.back}
d_{\bH^2}(\tilde{q}^{(m)}_{i,k_i,+}, \tilde{q}^{(m)}_i)
= m\cdot \ell_{X}(\gamma)-d_X(q, q_{i,k_i,+}).
\end{equation}
(See Figure~\ref{fig:angle.const} for an illustration.)

\begin{figure}[htp]
    \centering
    \includegraphics[width=13cm]{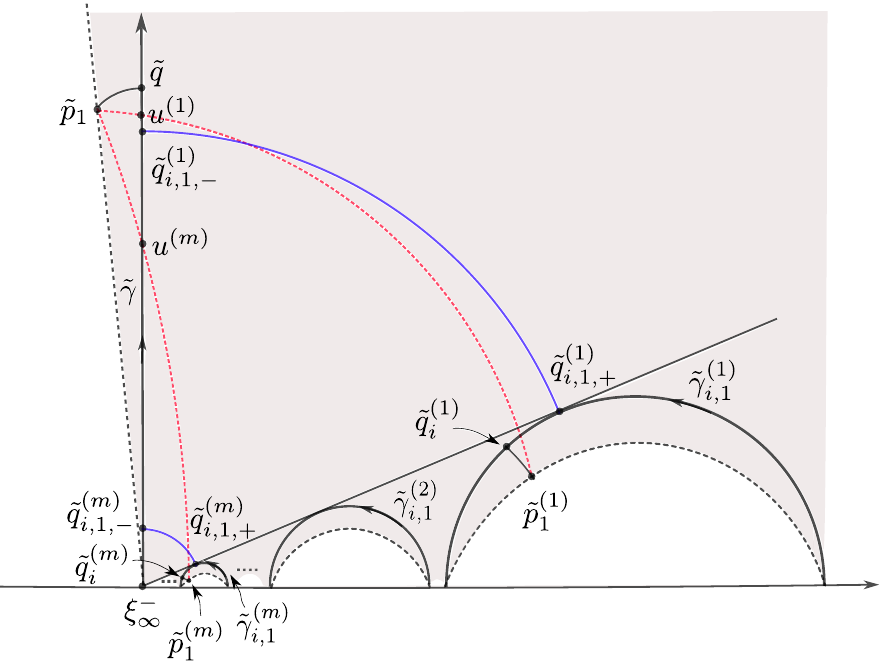}
    \caption{Lifts of \(\gamma_X\) and the arcs \(\alpha_i^{(m)}(X)\) in the universal cover (example for \(i=1\)).}
    \label{fig:angle.const}
\end{figure}

Equations~\eqref{eq:twist.forward} and~\eqref{eq:twist.back} show that, as
\(m\to\infty\), both
\(d_{\bH^2}(\tilde{q}, \tilde{q}^{(m)}_{i,k_i,-})\) and
\(d_{\bH^2}(\tilde{q}^{(m)}_{i,k_i,+},\tilde{q}^{(m)}_i)\) tend to \(+\infty\).
Hence the geodesics \(\tilde{\gamma}^{(m)}_{i,k_i}\) converge (in the
compactification \(\overline{\bH^2}\)) to a geodesic with one endpoint equal to
the repelling fixed point \(\xi^-_\infty\) of the deck transformation
associated to \(\gamma\).
The points \(\tilde{p}_1^{(m)}\), obtained by moving a distance \(d\) along the
perpendicular to \(\tilde{\gamma}^{(m)}_{i,k_i}\) from \(\tilde{q}^{(m)}_i\),
also converge to \(\xi^-_\infty\).
The geodesic segment joining \(\tilde{p}_1\) to \(\tilde{p}_1^{(m)}\) meets
\(\tilde{\gamma}\) at a point \(u^{(m)}\), and projects to \(\alpha_i^{(m)}(X)\).
Thus \(\zeta^{(m)}_{i,-}\) is exactly the angle at \(\tilde{p}_1\) of the
geodesic triangle with vertices \(\tilde{p}_1\), \(\tilde{q}\) and \(u^{(m)}\).

As \(m\to\infty\), the points \(u^{(m)}\) converge in \(\overline{\bH^2}\) to
\(\xi^-_\infty\), so these triangles converge to a right-angled hyperbolic
triangle with vertices \(\tilde{p}_1\), \(\tilde{q}\) and \(\xi^-_\infty\),
right angle at \(\tilde{q}\), and finite side length
\(d_{\bH^2}(\tilde{p}_1,\tilde{q})=d\).
In such a triangle, the angle \(\zeta_\infty\) at \(\tilde{p}_1\) is the
classical angle of parallelism at distance \(d\); standard hyperbolic
trigonometry (see e.g.\ \cite[Sec.~2.2]{Bus92}) gives
\[
\sin\zeta_\infty
= \frac{1}{\cosh d},
\]
i.e.
\[
\zeta_\infty
=\arcsin\!\bigg(\frac{1}{\cosh d}\bigg)
=\arcsin\!\bigg(\frac{1}{\cosh d_X(p_1,\gamma_X)}\bigg).
\]
Thus \(\zeta^{(m)}_{i,-}\to\zeta_\infty\) as \(m\to\infty\).

The same argument, applied to the lift of \(\alpha_i^{(m)}(X)\) oriented
towards \(p_1\), shows that the angles \(\zeta^{(m)}_{i,+}\) converge to the
same limit \(\zeta_\infty\).
This proves the lemma.
\end{proof}

\begin{corollary}\label{cor:triang.zero}
Keep the notation of Construction~\ref{constr:triang.degen.}. Then
\[
\Theta^{(m)}(X)\longrightarrow \theta_1(X) \quad\text{as } m\to\infty.
\]
\end{corollary}

\begin{proof}
Choose a sufficiently small metric neighbourhood \(U_{p_1}\) of \(p_1\) in \(X\).
For each \(m\), the geodesics \(\alpha^{(m)}_1(X)\), \(\alpha^{(m)}_2(X)\),
\(\alpha^{(m)}_3(X)\) intersect \(U_{p_1}\) in six geodesic segments with one
endpoint at \(p_1\).  We denote these by
\(\beta^{(m)}_{i,\pm}(X)\) (\(i=1,2,3\)), where the sign indicates whether
the segment is oriented towards (\(+\)) or away from (\(-\)) \(p_1\) along
\(\alpha^{(m)}_i(X)\).

By Lemma~\ref{lem:limit.angle.const}, as \(m\to\infty\) the outgoing
segments \(\beta^{(m)}_{i,-}(X)\) (\(i=1,2,3\)) become asymptotically
parallel: their directions at \(p_1\) all converge to a single ray that
makes angle \(\zeta_\infty\) with the reference geodesic \(\eta_X\).
Similarly, the incoming segments \(\beta^{(m)}_{i,+}(X)\) also converge to a
single ray at angle \(\zeta_\infty\) from \(\eta_X\).
Therefore, for any fixed pair of indices \((i,\sigma)\) and \((j,\tau)\) with
\(\sigma,\tau\in\{+,-\}\), the angle at \(p_1\) between
\(\beta^{(m)}_{i,\sigma}(X)\) and \(\beta^{(m)}_{j,\tau}(X)\) (chosen in $(0,\pi$)) tends to \(0\)
as \(m\to\infty\). 

On a small metric circle around \(p_1\), the six segments
\(\beta^{(m)}_{i,\pm}(X)\) determine six intersection points arranged in
cyclic order, and the cone angle \(\theta_1(X)\) is the sum of the six
interior angles between consecutive segments.  By Construction~\ref{constr:triang.degen.}, the triangle angle
\(\Theta^{(m)}(X)\) is precisely the interior angle at \(p_1\) between
\(\beta_{2,-}^{(m)}(X)\subset\alpha^{(m)}_2(X)\) and \(\beta_{1,+}^{(m)}(X)\subset\alpha^{(m)}_1(X)\), that is, between one
outgoing segment and one incoming segment.

Let \(\varphi^{(m)}(X)\) be the sum of the remaining five angles between
consecutive segments at \(p_1\).  Then
\[
\theta_1(X)
= \Theta^{(m)}(X) + \varphi^{(m)}(X)
\]
for all \(m\).  The convergence above shows that each of these five angles
tends to \(0\) as \(m\to\infty\), hence \(\varphi^{(m)}(X)\to 0\).  It
follows that \(\Theta^{(m)}(X)\to \theta_1(X)\), as required.
\end{proof}

\begin{proof}[Proof of Proposition \ref{prop:n1ggeq2.cone.angle}]
Suppose, for contradiction, that there exists
\(X\in\cT_{\rm u.t.}(S,\underline{p})\) with
\(\theta_{p_1}(X)\in (\pi,2\pi)\).
By Lemma~\ref{lem:dist.bt.zero} we can choose a separating simple closed
curve class \([\gamma]\) whose geodesic representative \(\gamma_X\) has
positive distance from \(p_1\).
Applying Construction~\ref{constr:triang.degen.} with this \(\gamma\), we
obtain, for each \(m\ge 1\), a triangle \(\Delta^{(m)}\) bounded by three
arcs \(\alpha^{(m)}_1,\alpha^{(m)}_2,\alpha^{(m)}_3\) based at \(p_1\), and
we denote by \(\Theta^{(m)}(X)\) the interior angle of the geodesic
triangle \(\Delta^{(m)}(X)\) at \(p_1\) opposite to \(\alpha^{(m)}_3(X)\).

By Corollary~\ref{cor:triang.zero},
\(\Theta^{(m)}(X)\to\theta_1(X)\) as \(m\to\infty\).
Set \(\epsilon_0 := \theta_1(X)-\pi > 0\).
There exists \(m_0\) such that for any $m\geq m_0$, 
\[
\theta_1(X) - \Theta^{(m)}(X) < \epsilon_0,
\]
hence
\[
\Theta^{(m_0)}(X)
> \theta_1(X) - \epsilon_0
= \theta_1(X) - (\theta_1(X)-\pi)
= \pi.
\]
Thus the geodesic triangle \(\Delta^{(m_0)}(X)\) has an interior angle
\(> \pi\) at \(p_1\) and is therefore degenerate.

Now choose a triangulation \(G\in\TSp\) that has \(\Delta^{(m_0)}\) as one
of its faces.  Since \(\Delta^{(m_0)}(X)\) is degenerate, $X$ does not admit a geodesic realization of $G$, so \(X\notin\cT(G)\).  This contradicts the assumption that
\(X\in\cT_{\rm u.t.}(S,\underline{p}) = \bigcap_{G\in\TSp}\cT(G)\).
Hence no universally triangulable surface \(X\) can have
\(\theta_{p_1}(X) \in (\pi,2\pi)\), and the cone angle at \(p_1\) must be
at most \(\pi\).
\end{proof}

This completes the proof of Proposition~\ref{prop:n1ggeq2.cone.angle}.
Together with Lemma~\ref{lem:n1g1.adm.eq.TS},
Proposition~\ref{prop:max.cone.angle}, and Lemma~\ref{lem:achived}, it
finishes the analysis of Question~\ref{q:maximal.angle}(b) in the single-cone
case \(n=1, g\ge 2\).
In particular, the universally maximal cone angle \(\Theta(S,\underline{p})\)
and its attainability are now completely determined for all topological
types \((g,n)\).

\printbibliography
\end{document}